\documentclass{article}
\usepackage[T1]{fontenc}
\usepackage{float}
\usepackage{hyperref}
\usepackage{comment}
\usepackage{float}
\usepackage[affil-it]{authblk}
\usepackage[T1]{fontenc}
\usepackage{lmodern}
\usepackage{amsmath,amssymb}
\usepackage{amssymb}
\usepackage{amsmath,amscd}
\usepackage{amsfonts}
\usepackage[utf8]{inputenc}
\usepackage{pstricks}
\usepackage{pstricks-add}
\usepackage{tikz}
\usepackage{url}
\usepackage{tensor}
\usepackage{here}
\usepackage{comment}
\usepackage{graphicx}
\usetikzlibrary{arrows}
\usepackage{MnSymbol}
\usepackage{stmaryrd}
\usepackage{shuffle}
\renewcommand{\S}{\ensuremath{\mathfrak{S}}}

\newcommand{\p}{\ensuremath{\mathcal{P}}}
\def\std{\mathrm{std}}

\def\N{{\mathbb N}}

\def\ta{\mathtt a}
\def\WSym{\mathbf{WSym}}
\def\BWSym{\mathbf{BWSym}}
\def\PiQSym{\mathrm{\Pi QSym}}

\def\CWSym{\mathbf{CWSym}}  
\def\BPiQSym{\mathrm{B\Pi QSym}}
\def\CPiQSym{\mathrm{C\Pi QSym}}

\def\block{\mathrm{block}}
\def\numb{\mathrm{num}}

\usepackage[all,cmtip]{xy}
\newcommand{\arinj}{\ar@{_{(}->}} 
\newcommand{\arinjrev}{\ar@{^{(}->}} 
\newcommand{\arsurj}{\ar@{->>}} 
\newcommand{\arelem}{\ar@{|->}} 
\newcommand{\arback}{\ar@{<-}} 
\newcommand\id{\operatorname{id}} 

\newtheorem{theorem}{Theorem}

\newtheorem{corollary}[theorem]{Corollary}

\newtheorem{definition}[theorem]{Definition}
\newtheorem{example}[theorem]{Example}

\newtheorem{lemma}[theorem]{Lemma}

\newtheorem{proposition}[theorem]{Proposition}
\newtheorem{remark}[theorem]{Remark}

\newenvironment{proof}[1][Proof]{\noindent\textbf{#1.} }{\ \rule{0.5em}{0.5em}}
  \newcommand{\keywords}[1]{%
  \par\smallskip
  \noindent\textbf{Keywords: }#1
}
\PassOptionsToPackage{dvipsnames}{xcolor}
\RequirePackage{xcolor} 
\definecolor{halfgray}{gray}{0.55} 
\definecolor{webgreen}{rgb}{0,.5,0}
\definecolor{webbrown}{rgb}{.6,0,0}
\definecolor{Noir}{cmyk}{0, 0, 0, 0}
\definecolor{Bleu}{cmyk}{100, 72, 0, 18}
\definecolor{Rouge}{cmyk}{0, 100, 80, 2}

\tikzstyle{Arete}=[Rouge!80,cap=round,line width=3pt]
\tikzstyle{Feuille}=[rectangle,draw=Noir!70,fill=Noir!20,minimum size=3mm,line width=2pt]
\tikzstyle{Operateur} = [rectangle,rounded corners,draw=Bleu!100,fill=Bleu!10,
minimum size=15mm,line width=3pt,font=\Huge]

\newcounter{Edge}
\newcounter{Vertex}

\def\bugdxuu#1#2#3
{
\draw (#1,#2) rectangle (#1+0.5,#2+0.5);
\node[](m#3) at (#1+0.25,#2+0.25) {\arabic{Vertex}\stepcounter{Vertex}};
\node[] (dd1#3) at (#1+0.1,#2+0.1){};
\node[] (d1#3) at (#1+0.1,#2-0.3){\tiny\arabic{Edge}};

\node[] (uu1#3) at (#1+0.1,#2+0.5){};
\node[] (u1#3) at (#1+0.1,#2+0.7){\tiny\arabic{Edge}\stepcounter{Edge}};

\node[] (uu2#3) at (#1+0.4,#2+0.5){};
\node[] (u2#3) at (#1+0.4,#2+0.7){\tiny\arabic{Edge}};

\node[] (dd2#3) at (#1+0.4,#2+0.1){};
\node[] (d2#3) at (#1+0.4,#2-0.3){\tiny$\genfrac{}{}{0pt}{}{\times}{\arabic{Edge}}\stepcounter{Edge}$};

\draw (dd1#3) -- (d1#3);
\draw (uu1#3) -- (u1#3);
\draw (uu2#3) -- (u2#3);
}

\def\bugdduu#1#2#3
{
\draw (#1,#2) rectangle (#1+0.5,#2+0.5);
\node[](m#3) at (#1+0.25,#2+0.25) {\arabic{Vertex}\stepcounter{Vertex}};
\node[] (dd1#3) at (#1+0.1,#2+0.1){};
\node[] (d1#3) at (#1+0.1,#2-0.3){\tiny\arabic{Edge}};

\node[] (uu1#3) at (#1+0.1,#2+0.5){};
\node[] (u1#3) at (#1+0.1,#2+0.7){\tiny\arabic{Edge}\stepcounter{Edge}};

\node[] (uu2#3) at (#1+0.4,#2+0.5){};
\node[] (u2#3) at (#1+0.4,#2+0.7){\tiny\arabic{Edge}};

\node[] (dd2#3) at (#1+0.4,#2+0.1){};
\node[] (d2#3) at (#1+0.4,#2-0.3){\tiny\arabic{Edge}\stepcounter{Edge}};

\draw (dd1#3) -- (d1#3);
\draw (uu1#3) -- (u1#3);
\draw (dd2#3) -- (d2#3);
\draw (uu2#3) -- (u2#3);
}

\def\bugddux#1#2#3
{
\draw (#1,#2) rectangle (#1+0.5,#2+0.5);
\node[](m#3) at (#1+0.25,#2+0.25) {\arabic{Vertex}\stepcounter{Vertex}};
\node[] (dd1#3) at (#1+0.1,#2+0.1){};
\node[] (d1#3) at (#1+0.1,#2-0.3){\tiny\arabic{Edge}};

\node[] (uu1#3) at (#1+0.1,#2+0.5){};
\node[] (u1#3) at (#1+0.1,#2+0.7){\tiny\arabic{Edge}\stepcounter{Edge}};

\node[] (uu2#3) at (#1+0.4,#2+0.5){};
\node[] (u2#3) at (#1+0.4,#2+0.7){\tiny$\genfrac{}{}{0pt}{}{\times}{\arabic{Edge}}$\stepcounter{Edge}};

\node[] (dd2#3) at (#1+0.4,#2+0.1){};
\node[] (d2#3) at (#1+0.4,#2-0.3){\tiny\arabic{Edge}};

\draw (dd1#3) -- (d1#3);
\draw (uu1#3) -- (u1#3);
\draw (dd2#3) -- (d2#3);
}

\def\bugdxuu#1#2#3
{
\draw (#1,#2) rectangle (#1+0.5,#2+0.5);
\node[](m#3) at (#1+0.25,#2+0.25) {\arabic{Vertex}\stepcounter{Vertex}};
\node[] (dd1#3) at (#1+0.1,#2+0.1){};
\node[] (d1#3) at (#1+0.1,#2-0.3){\tiny\arabic{Edge}};

\node[] (uu1#3) at (#1+0.1,#2+0.5){};
\node[] (u1#3) at (#1+0.1,#2+0.7){\tiny\arabic{Edge}\stepcounter{Edge}};

\node[] (uu2#3) at (#1+0.4,#2+0.5){};
\node[] (u2#3) at (#1+0.4,#2+0.7){\tiny\arabic{Edge}};

\node[] (dd2#3) at (#1+0.4,#2+0.1){};
\node[] (d2#3) at (#1+0.4,#2-0.3){\tiny$\genfrac{}{}{0pt}{}{\times}{\arabic{Edge}}$\stepcounter{Edge}};

\draw (dd1#3) -- (d1#3);
\draw (uu1#3) -- (u1#3);
\draw (uu2#3) -- (u2#3);
}

\def\bugdu#1#2#3
{
\draw (#1,#2) rectangle (#1+0.4,#2+0.5);
\node[](m#3) at (#1+0.2,#2+0.25) {\arabic{Vertex}\stepcounter{Vertex}};
\node[] (dd1#3) at (#1+0.2,#2+0.1){};
\node[] (d1#3) at (#1+0.2,#2-0.3){\tiny\arabic{Edge}};

\node[] (uu1#3) at (#1+0.2,#2+0.5){};
\node[] (u1#3) at (#1+0.2,#2+0.7){\tiny\arabic{Edge}\stepcounter{Edge}};

\draw (dd1#3) -- (d1#3);
\draw (uu1#3) -- (u1#3);
}

\def\bugddux#1#2#3
{
\draw (#1,#2) rectangle (#1+0.5,#2+0.5);
\node[](m#3) at (#1+0.25,#2+0.25) {\arabic{Vertex}\stepcounter{Vertex}};
\node[] (dd1#3) at (#1+0.1,#2+0.1){};
\node[] (d1#3) at (#1+0.1,#2-0.3){\tiny\arabic{Edge}};

\node[] (uu1#3) at (#1+0.1,#2+0.5){};
\node[] (u1#3) at (#1+0.1,#2+0.7){\tiny\arabic{Edge}\stepcounter{Edge}};

\node[] (uu2#3) at (#1+0.4,#2+0.5){};
\node[] (u2#3) at (#1+0.4,#2+0.7){\tiny$\genfrac{}{}{0pt}{}{\arabic{Edge}}{\times}$};

\node[] (dd2#3) at (#1+0.4,#2+0.1){};
\node[] (d2#3) at (#1+0.4,#2-0.3){\tiny\arabic{Edge}\stepcounter{Edge}};

\draw (dd1#3) -- (d1#3);
\draw (uu1#3) -- (u1#3);
\draw (dd2#3) -- (d2#3);

}

\def\bugxdxu#1#2#3
{
\draw (#1,#2) rectangle (#1+0.5,#2+0.5);
\node[](m#3) at (#1+0.25,#2+0.25) {\arabic{Vertex}\stepcounter{Vertex}};
\node[] (dd1#3) at (#1+0.1,#2+0.1){};
\node[] (d1#3) at (#1+0.1,#2-0.3){\tiny$\genfrac{}{}{0pt}{}{\times}{\arabic{Edge}}$};

\node[] (uu1#3) at (#1+0.1,#2+0.5){};
\node[] (u1#3) at (#1+0.1,#2+0.7){\tiny$\genfrac{}{}{0pt}{}{\arabic{Edge}}{\times}$\stepcounter{Edge}};

\node[] (uu2#3) at (#1+0.4,#2+0.5){};
\node[] (u2#3) at (#1+0.4,#2+0.7){\tiny\arabic{Edge}};

\node[] (dd2#3) at (#1+0.4,#2+0.1){};
\node[] (d2#3) at (#1+0.4,#2-0.3){\tiny\arabic{Edge}\stepcounter{Edge}};

\draw (dd2#3) -- (d2#3);
\draw (uu2#3) -- (u2#3);

}

\def\bugdxxx#1#2#3
{
\draw (#1,#2) rectangle (#1+0.5,#2+0.5);
\node[](m#3) at (#1+0.25,#2+0.25) {\arabic{Vertex}\stepcounter{Vertex}};
\node[] (dd1#3) at (#1+0.1,#2+0.1){};
\node[] (d1#3) at (#1+0.1,#2-0.3){\tiny\arabic{Edge}};

\node[] (uu1#3) at (#1+0.1,#2+0.5){};
\node[] (u1#3) at (#1+0.1,#2+0.7){\tiny$\genfrac{}{}{0pt}{}{\arabic{Edge}}{\times}$\stepcounter{Edge}};

\node[] (uu2#3) at (#1+0.4,#2+0.5){};
\node[] (u2#3) at (#1+0.4,#2+0.7){\tiny$\genfrac{}{}{0pt}{}{\arabic{Edge}}{\times}$};
\node[] (dd2#3) at (#1+0.4,#2+0.1){};
\node[] (d2#3) at (#1+0.4,#2-0.3){\tiny $\genfrac{}{}{0pt}{}{\times}{\arabic{Edge}}$\stepcounter{Edge}};

\draw (dd1#3) -- (d1#3);
}

\def\bugddxuux#1#2#3
{
\draw (#1,#2) rectangle (#1+0.7,#2+0.5);
\node[](m#3) at (#1+0.35,#2+0.25) {\arabic{Vertex}\stepcounter{Vertex}};
\node[] (dd1#3) at (#1+0.1,#2+0.1){};
\node[] (d1#3) at (#1+0.1,#2-0.3){\tiny\arabic{Edge}};

\node[] (uu1#3) at (#1+0.1,#2+0.5){};
\node[] (u1#3) at (#1+0.1,#2+0.7){\tiny\arabic{Edge}\stepcounter{Edge}};

\node[] (dd2#3) at (#1+0.35,#2+0.1){};
\node[] (d2#3) at (#1+0.35,#2-0.3){\tiny\arabic{Edge}};

\node[] (uu2#3) at (#1+0.35,#2+0.5){};
\node[] (u2#3) at (#1+0.35,#2+0.7){\tiny\arabic{Edge}\stepcounter{Edge}};

\node[] (uu3#3) at (#1+0.6,#2+0.5){};
\node[] (u3#3) at (#1+0.6,#2+0.7){\tiny$\genfrac{}{}{0pt}{}{\arabic{Edge}}{\times}$};

\node[] (dd3#3) at (#1+0.6,#2+0.1){};
\node[] (d3#3) at (#1+0.6,#2-0.3){\tiny$\genfrac{}{}{0pt}{}{\times}{\arabic{Edge}}$\stepcounter{Edge}};

\draw (dd1#3) -- (d1#3);
\draw (uu1#3) -- (u1#3);
\draw (dd2#3) -- (d2#3);
\draw (uu2#3) -- (u2#3);

}

\def\bugddduux#1#2#3
{
\draw (#1,#2) rectangle (#1+0.7,#2+0.5);
\node[](m#3) at (#1+0.35,#2+0.25) {\arabic{Vertex}\stepcounter{Vertex}};
\node[] (dd1#3) at (#1+0.1,#2+0.1){};
\node[] (d1#3) at (#1+0.1,#2-0.3){\tiny\arabic{Edge}};

\node[] (uu1#3) at (#1+0.1,#2+0.5){};
\node[] (u1#3) at (#1+0.1,#2+0.7){\tiny\arabic{Edge}\stepcounter{Edge}};

\node[] (dd2#3) at (#1+0.35,#2+0.1){};
\node[] (d2#3) at (#1+0.35,#2-0.3){\tiny\arabic{Edge}};

\node[] (uu2#3) at (#1+0.35,#2+0.5){};
\node[] (u2#3) at (#1+0.35,#2+0.7){\tiny\arabic{Edge}\stepcounter{Edge}};

\node[] (uu3#3) at (#1+0.6,#2+0.5){};
\node[] (u3#3) at (#1+0.6,#2+0.7){\tiny$\genfrac{}{}{0pt}{}{\arabic{Edge}}{\times}$};

\node[] (dd3#3) at (#1+0.6,#2+0.1){};
\node[] (d3#3) at (#1+0.6,#2-0.3){\tiny\arabic{Edge}\stepcounter{Edge}};

\draw (dd1#3) -- (d1#3);
\draw (uu1#3) -- (u1#3);
\draw (dd2#3) -- (d2#3);
\draw (dd3#3) -- (d3#3);
\draw (uu2#3) -- (u2#3);

}

\def\bugddduuu#1#2#3
{
\draw (#1,#2) rectangle (#1+0.7,#2+0.5);
\node[](m#3) at (#1+0.35,#2+0.25) {\arabic{Vertex}\stepcounter{Vertex}};
\node[] (dd1#3) at (#1+0.1,#2+0.1){};
\node[] (d1#3) at (#1+0.1,#2-0.3){\tiny\arabic{Edge}};

\node[] (uu1#3) at (#1+0.1,#2+0.5){};
\node[] (u1#3) at (#1+0.1,#2+0.7){\tiny\arabic{Edge}\stepcounter{Edge}};

\node[] (dd2#3) at (#1+0.35,#2+0.1){};
\node[] (d2#3) at (#1+0.35,#2-0.3){\tiny\arabic{Edge}};

\node[] (uu2#3) at (#1+0.35,#2+0.5){};
\node[] (u2#3) at (#1+0.35,#2+0.7){\tiny\arabic{Edge}\stepcounter{Edge}};

\node[] (uu3#3) at (#1+0.6,#2+0.5){};
\node[] (u3#3) at (#1+0.6,#2+0.7){\tiny\arabic{Edge}};

\node[] (dd3#3) at (#1+0.6,#2+0.1){};
\node[] (d3#3) at (#1+0.6,#2-0.3){\tiny\arabic{Edge}\stepcounter{Edge}};

\draw (dd1#3) -- (d1#3);
\draw (uu1#3) -- (u1#3);
\draw (dd2#3) -- (d2#3);
\draw (dd3#3) -- (d3#3);
\draw (uu2#3) -- (u2#3);
\draw (uu3#3) -- (u3#3);

}

\def\bugxdxxuu #1#2#3
{
\draw (#1,#2) rectangle (#1+0.7,#2+0.5);
\node[](m#3) at (#1+0.35,#2+0.25) {\arabic{Vertex}\stepcounter{Vertex}};
\node[] (dd1#3) at (#1+0.1,#2+0.1){};
\node[] (d1#3) at (#1+0.1,#2-0.3){\tiny$\genfrac{}{}{0pt}{}{\times}{\arabic{Edge}}$};

\node[] (uu1#3) at (#1+0.1,#2+0.5){};
\node[] (u1#3) at (#1+0.1,#2+0.7){\tiny$\genfrac{}{}{0pt}{}{\arabic{Edge}}{\times}$\stepcounter{Edge}};

\node[] (dd2#3) at (#1+0.35,#2+0.1){};
\node[] (d2#3) at (#1+0.35,#2-0.3){\tiny\arabic{Edge}};

\node[] (uu2#3) at (#1+0.35,#2+0.5){};
\node[] (u2#3) at (#1+0.35,#2+0.7){\tiny\arabic{Edge}\stepcounter{Edge}};

\node[] (uu3#3) at (#1+0.6,#2+0.5){};
\node[] (u3#3) at (#1+0.6,#2+0.7){\tiny\arabic{Edge}};

\node[] (dd3#3) at (#1+0.6,#2+0.1){};
\node[] (d3#3) at (#1+0.6,#2-0.3){\tiny$\genfrac{}{}{0pt}{}{\times}{\arabic{Edge}}$\stepcounter{Edge}};

\draw (dd2#3) -- (d2#3);
\draw (uu2#3) -- (u2#3);
\draw (uu3#3) -- (u3#3);

}

\def\bugxx#1#2#3
{
\draw (#1,#2) rectangle (#1+0.4,#2+0.5);
\node[](m#3) at (#1+0.2,#2+0.25) {\arabic{Vertex}\stepcounter{Vertex}};
\node[] (dd1#3) at (#1+0.2,#2+0.1){};
\node[] (d1#3) at (#1+0.2,#2-0.3){\tiny$\genfrac{}{}{0pt}{}{\times}{\arabic{Edge}}$};

\node[] (uu1#3) at (#1+0.2,#2+0.5){};
\node[] (u1#3) at (#1+0.2,#2+0.7){\tiny$\genfrac{}{}{0pt}{}{\arabic{Edge}}{\times}$\stepcounter{Edge}};
}

\def\bugxu#1#2#3
{
\draw (#1,#2) rectangle (#1+0.4,#2+0.5);
\node[](m#3) at (#1+0.2,#2+0.25) {\arabic{Vertex}\stepcounter{Vertex}};
\node[] (dd1#3) at (#1+0.2,#2+0.1){};
\node[] (d1#3) at (#1+0.2,#2-0.3){\tiny$\genfrac{}{}{0pt}{}{\times}{\arabic{Edge}}$};

\node[] (uu1#3) at (#1+0.2,#2+0.5){};
\node[] (u1#3) at (#1+0.2,#2+0.7){\tiny\arabic{Edge}\stepcounter{Edge}};

\draw (uu1#3) -- (u1#3);
}

\def\bugdx#1#2#3
{
\draw (#1,#2) rectangle (#1+0.4,#2+0.5);
\node[](m#3) at (#1+0.2,#2+0.25) {\arabic{Vertex}\stepcounter{Vertex}};
\node[] (dd1#3) at (#1+0.2,#2+0.1){};
\node[] (d1#3) at (#1+0.2,#2-0.3){\tiny\arabic{Edge}};

\node[] (uu1#3) at (#1+0.2,#2+0.5){};
\node[] (u1#3) at (#1+0.2,#2+0.7){\tiny$\genfrac{}{}{0pt}{}{\arabic{Edge}}{\times}$\stepcounter{Edge}};

\draw (dd1#3) -- (d1#3);
}

\def\bugddxx#1#2#3
{
\draw (#1,#2) rectangle (#1+0.5,#2+0.5);
\node[](m#3) at (#1+0.25,#2+0.25) {\arabic{Vertex}\stepcounter{Vertex}};

\node[] (dd1#3) at (#1+0.1,#2+0.1){};
\node[] (d1#3) at (#1+0.1,#2-0.3){\tiny\arabic{Edge}};

\node[] (uu1#3) at (#1+0.1,#2+0.5){};
\node[] (u1#3) at (#1+0.1,#2+0.7){\tiny$\genfrac{}{}{0pt}{}{\arabic{Edge}}{\times}\stepcounter{Edge}$};

\node[] (dd2#3) at (#1+0.4,#2+0.1){};
\node[] (d2#3) at (#1+0.4,#2-0.3){\tiny\arabic{Edge}};

\node[] (uu2#3) at (#1+0.4,#2+0.5){};
\node[] (u2#3) at (#1+0.4,#2+0.7){\tiny$\genfrac{}{}{0pt}{}{\arabic{Edge}}{\times}\stepcounter{Edge}$};

\draw (dd1#3) -- (d1#3);
\draw (dd2#3) -- (d2#3);
}

\def\bugxdxx#1#2#3
{
\draw (#1,#2) rectangle (#1+0.5,#2+0.5);
\node[](m#3) at (#1+0.25,#2+0.25) {\arabic{Vertex}\stepcounter{Vertex}};
\node[] (dd1#3) at (#1+0.1,#2+0.1){};
\node[] (d1#3) at (#1+0.1,#2-0.3){\tiny$\genfrac{}{}{0pt}{}{\times}{\arabic{Edge}}$};

\node[] (uu1#3) at (#1+0.1,#2+0.5){};
\node[] (u1#3) at (#1+0.1,#2+0.7){\tiny$\genfrac{}{}{0pt}{}{\arabic{Edge}}{\times}$\stepcounter{Edge}};

\node[] (uu2#3) at (#1+0.4,#2+0.5){};
\node[] (u2#3) at (#1+0.4,#2+0.7){\tiny$\genfrac{}{}{0pt}{}{\arabic{Edge}}{\times}$};

\node[] (dd2#3) at (#1+0.4,#2+0.1){};
\node[] (d2#3) at (#1+0.4,#2-0.3){\tiny\arabic{Edge}\stepcounter{Edge}};

\draw (dd2#3) -- (d2#3);
}

\def\bugxxxx#1#2#3
{
\draw (#1,#2) rectangle (#1+0.5,#2+0.5);
\node[](m#3) at (#1+0.25,#2+0.25) {\arabic{Vertex}\stepcounter{Vertex}};
\node[] (dd1#3) at (#1+0.1,#2+0.1){};
\node[] (d1#3) at (#1+0.1,#2-0.3){\tiny$\genfrac{}{}{0pt}{}{\times}{\arabic{Edge}}$};

\node[] (uu1#3) at (#1+0.1,#2+0.5){};
\node[] (u1#3) at (#1+0.1,#2+0.7){\tiny$\genfrac{}{}{0pt}{}{\arabic{Edge}}{\times}$\stepcounter{Edge}};

\node[] (dd2#3) at (#1+0.4,#2+0.1){};
\node[] (d2#3) at (#1+0.4,#2-0.3){\tiny$\genfrac{}{}{0pt}{}{\times}{\arabic{Edge}}$};

\node[] (uu2#3) at (#1+0.4,#2+0.5){};
\node[] (u2#3) at (#1+0.4,#2+0.7){\tiny$\genfrac{}{}{0pt}{}{\arabic{Edge}}{\times}$};
}

\def\mcomp#1#2{{\displaystyle\mathop\medstar^{#2}_{#1}}}
\def\comp#1#2{{\displaystyle\mathop\star^{#2}_{#1}}}

\title{Hopf Algebras of 
$B$-Diagrams and Boson Normal Ordering: Exploring the Dual Structures}
 \author{Ali Chouria%
  \thanks{Electronic address: \texttt{ali.chouria1@univ-rouen.fr};}} 
\affil{University of Jendouba, ISLAIB,\\
UR 22ES12: Modeling Optimization and Augmented Engineering, 9000 Béja,
Tunisia. \\ GR$^2$IF, University of Rouen Normandie,France.}

\author{Jean-Gabriel Luque%
  \thanks{Electronic address: \texttt{jeangabriel.luque@univ-rouen.fr} Corresponding author}}
\affil{GR$^2$IF, University of Rouen Normandie,\\ Avenue de l’universit\'{e}, 76801 Saint-\'{E}tienne-du-Rouvray, France.}
\def\eqref#1{(\ref{#1})}
\def\shuffle{\sqcup\mathchoice{\mkern-8mu}{\mkern-8mu}{\mkern-3.2mu}{\mkern-3.8mu}\sqcup}
\begin{document}
\maketitle

\begin{abstract}
We consider the Hopf algebra of $B$-diagrams as an algebra projecting onto the Heisenberg algebra and designed to encode the combinatorics of the bosonic normal-ordering problem. 
In order to understand and generalize the properties of the algebra of noncommutative symmetric polynomials viewed as a Hopf subalgebra of
the Hopf algebra linearly spanned by $B$-diagrams, we describe and study its dual Hopf algebra. 
This construction also allows us to establish connections with combinatorial Hopf algebras based on colored set partitions.
\end{abstract}


\section{Introduction}
For bosons, creation and annihilation operators first arise in the quantum harmonic oscillator as raising and lowering operators, and are later extended to quantum field theory as field operators (see, e.g. the introductory chapter of \cite{Mahan}). Within this framework, quantum many-body states are described in a Fock basis, obtained by assigning occupation numbers to each single particle state. The second-quantization formalism introduces creation ($a^\dag$) and annihilation ($a$) operators  to generate and manipulate these Fock states.
These operators satisfy the relation
\begin{equation}\label{eq_aa}
[a,a^\dag] = 1.
\end{equation}
The algebra generated by the operator $a^\dag$ and $a$ is called the Heisenberg Algebra and is preciselly defined by generators and relation following equality \eqref{eq_aa}.
The normal-ordering problem consists in characterizing an operator of the Heisenberg algebra in a unique way by expressing it as a linear combination of operators of the form $(a^\dagger)^n a^m$. The underlying mechanism relies on the repeated use of equality~\eqref{eq_aa} and gives rise to a remarkably rich combinatorial structure. This combinatorial structure was first uncovered by Katriel~\cite{Katriel}, who showed that applying the normal-ordering problem to operators of the form $(a^\dag a)^n$ leads to the Stirling numbers of the second kind $S_{n,k}$, which count the number of partitions of an $n$-element set into $k$ blocks:
\begin{equation}\label{eq-Katriel}
(a^\dag a)^n=\sum_{k=1}^nS_{n,k}(a^\dag)^ka^k.
\end{equation}
In a similar spirit, Blasiak and several coauthors explored, in a series of papers
\cite{PHDBlasiak,BPS,BPS2,BHPSD0,BPDSHP}, the combinatorial interpretations of the bosonic normal-ordering problem. 
The combinatorial objects they introduced, called \emph{bugs}, generalize set partitions in a certain sense, and suitable gluing constructions between these objects allow one to interpret the product of monomials in the Heisenberg algebra. 
In particular, the paper~\cite{BPDSHP} exhibits a Hopf algebra structure whose bases are indexed by bugs and whose product projects onto the product of the Heisenberg algebra. In a previous paper~\cite{BCL2016}, we investigated another Hopf-algebraic construction, closely related to that of Blasiak \emph{et al.}, but somewhat more general and, above all, connected to other families of combinatorial Hopf algebras existing in the literature. In this algebra, the combinatorial objects are $B$-diagrams (which play the role of the bugs), and their branching and decomposition rules endow the space they span with a Hopf algebra structure. In this setting, Katriel’s formula arises from the projection of a subalgebra isomorphic to the algebra of noncommutative symmetric functions introduced by Wolf~\cite{Wolf}, whose Hopf algebra structure was later described by Rosas and Sagan~\cite{RS}. 
This algebra also appears in the literature under the name $\WSym$ (see, for instance, \cite{HNT2008}).
Rather surprisingly, this algebra turns out to be a Hopf subalgebra of the Hopf algebra $\mathcal B$ of $B$-diagrams, and its dual Hopf algebra $\PiQSym$ is, in a natural way, isomorphic to a Hopf subalgebra of the dual Hopf algebra $\mathcal B^*$ of $\mathcal B$. 
Note that this is not the configuration one would ordinarily expect, since under duality arrows are reversed and the usual construction would rather lead to a quotient of $\mathcal B^*$ isomorphic to $\PiQSym$. 
This observation motivated us to complete the description of the Hopf algebras introduced in our previous paper \cite{BCL2016} by placing a particular emphasis on the notion of duality.

 The precise algebraic framework in which our observations take place is described in Section~\ref{sec-background}. 
The Hopf algebras we consider are cocommutative Hopf algebras of noncommutative polynomials, and the property of interest is that the product in the dual can be described explicitly on the original space, in a way similar to the shuffle product for free associative algebra. 
Lemma~\ref{lemsubdual} provides the general formula for this dual product. 
Proposition~\ref{prop_subdual} formalizes the fact that, under suitable conditions, a Hopf subalgebra and its dual can be realized explicitly on the same underlying space, which is itself a subspace of the original algebra. 
This generalizes the observation made for $\WSym$ and its dual $\PiQSym$, viewed as Hopf subalgebras of $\mathcal B$ and $\mathcal B^*$, respectively. In Section~\ref{sec-hopfdiag}, we first recall the original construction of the combinatorial objects ($B$-diagrams) and of the algebra $\mathcal B$. 
We take this opportunity to improve and simplify the notation introduced in~\cite{BCL2016}, in order to make it more readable and to streamline the constructions. 
The Hopf algebra structure of $\mathcal B$ is described in Subsection~\ref{sec-HopfB}, while the structure of its dual Hopf algebra is given in Subsection~\ref{subsec-dualbalg}. 
One of the key results of Section~\ref{sec-hopfdiag} is the introduction, in Subsection~\ref{sec-multbas}, of a multiplicative basis of $\mathcal B$. 
It is in this basis that the noncommutative polynomial Hopf algebra structure will be described in a natural way, making it possible, later on, to explicitly implement Proposition~\ref{prop_subdual}. The connection between the algebra $\mathcal B$ and families of “classical” combinatorial Hopf algebras is made explicit in Section~\ref{sec-diagtopart} through colored set partitions and the Hopf algebras introduced in Subsection~\ref{subsec-colorpart}. 
Note that the previous subsection (Subsection~\ref{subsec-wsym}) revisits in more detail the observation made earlier about $\WSym$. 
This construction will serve as a model in what follows, as it is generalized to Hopf algebras of colored set partitions. 
The main result of this section is Theorem~\ref{CPIQ} in Subsection~\ref{subsec-finiteset}. 
It is this result that establishes the link between the Hopf subalgebras of $\mathcal B$ that are isomorphic to Hopf algebras of colored set partitions and Proposition~\ref{prop_subdual}. Finally, we illustrate the construction on several examples in Section~\ref{sec-examples}.

\section{Background and preliminary results on Hopf algebras\label{sec-background}}
Throughout the paper, unless otherwise specified, all structures (algebras, coalgebras, bialgebras) are taken over complex vector spaces.
\subsection{Bialgebras and Hopf algebras}

Bialgebras and Hopf algebras (see \cite{Hopf2020} and \cite{HNT2008}) are algebraic structures that naturally appears in various domains such as quantum group theory, noncommutative geometry, and the study of symmetries in algebraic topology. 
Intuitively, a bialgebra combines both an algebra and a coalgebra structure on the same vector space in a compatible manner, while a Hopf algebra introduces an additional map $S$, called the antipode, which behaves as a generalized inverse for the convolution product.

A bialgebra $H$ is a vector space endowed simultaneously with an algebra structure $(m, u)$ and a coalgebra structure $(\Delta, \varepsilon)$ that are compatible in the sense that $\Delta$ is an algebra morphism (equivalently, $m$ is a coalgebra morphism). 
This compatibility is usually expressed as a commutative diagram (see Figure~\ref{fig:bialgebra-diagrams}).
\begin{figure}[ht]
    \centering

\begin{equation*}
\xymatrix{
H \otimes H \otimes H \otimes H \ar[rr]^{\id \otimes \tau \otimes \id} &&  H \otimes H \otimes H \otimes H  \ar[d]^{m \otimes m}\\
H \otimes H \otimes \ar[u]^{\Delta \otimes \Delta} \ar[r]^{m} & H \ar[r]^{\Delta} & H \otimes H     \\
}
\end{equation*}
    \caption{Commutative diagram}
    \label{fig:bialgebra-diagrams}
\end{figure}

The compatibility between the algebra and coalgebra structures of a bialgebra is expressed by the following identity:
\begin{equation}
\Delta \circ m = (m \otimes m) \circ (\id \otimes \tau \otimes \id) \circ (\Delta \otimes \Delta),
\end{equation}
where $\tau : B \otimes A \to A \otimes B$ denotes the map defined by $\tau(b \otimes a) = a \otimes b$ for all $a \in A$ and $b \in B$, with $A$ and $B$ graded modules.
This relation expresses that the comultiplication $\Delta$ preserves the algebraic structure.

\paragraph{Compatibility of the unit and counit.}
The counit $\varepsilon$ is required to be an algebra morphism:
\begin{equation}
\varepsilon \circ m = \varepsilon \otimes \varepsilon, 
\qquad 
\varepsilon \circ u = \id,
\end{equation}
and the unit $u$ must be a coalgebra morphism:
\begin{equation}
\Delta \circ u = u \otimes u,
\qquad 
\varepsilon \circ u = \id.
\end{equation}

\medskip

A \emph{Hopf algebra} is a bialgebra 
$(H, m, u, \Delta, \varepsilon)$ 
over a field $K$ with a linear map $S : H \longrightarrow H$, called the \emph{antipode}, satisfying the commutativity of the diagram in Figure~\ref{fig:hopf-diagram}. 
Equivalently, $S$ satisfies
\begin{equation}
m \circ (S \otimes \mathrm{id}) \circ \Delta 
= 
u \circ \varepsilon 
= 
m \circ (\mathrm{id} \otimes S) \circ \Delta.
\end{equation}
This condition ensures that the antipode acts as a generalized inverse with respect to the convolution product induced by the bialgebra structure.

\begin{figure}[ht]
    \centering

\begin{equation}
\label{antipode-diagram}
\xymatrix{
&H \otimes H \ar[rr]^{S \otimes \id_H}& &H \otimes H \ar[dr]^m& \\
H \ar[ur]^\Delta \ar[rr]^{\epsilon} \ar[dr]^\Delta& & K \ar[rr]^{u} & & H\\
&A \otimes H \ar[rr]^{\id_H \otimes S}& &H \otimes H \ar[ur]^m& \\
}
\end{equation}
    \caption{Commutative diagram characterizing the antipode of a Hopf algebra.}
    \label{fig:hopf-diagram}
\end{figure}

\paragraph{Duality of Hopf Algebras.}
The notion of duality plays a central role in the theory of Hopf algebras, reflecting the deep symmetry between algebraic and coalgebraic structures. 
To any  Hopf algebra $H = (H, m, u, \Delta, \varepsilon, S)$, one  associates the \emph{dual Hopf algebra} $H^{*}$, constructed on the Sweedler dual\footnote{The space of endomorphisms vanishing on a finite codimension ideal.} of the space $H$, in which the roles of multiplication and comultiplication are interchanged.

If $H=\bigoplus_{k}H_k$ is a locally finite-dimensional graded Hopf algebra (i.e. graduations is compatible with the product and the coproduct and 
$\dim H_k<\infty$ for each $k$), its linear graded dual space 
\begin{equation}H^{*} = \bigoplus_{k}H^*_k\end{equation}
coincides with the Sweedler dual and naturally inherits a Hopf algebra structure defined as follows:
\begin{equation}
\begin{array}{rcll}
(f \cdot g)(x) &=& (f \otimes g)\big(\Delta(x)\big), \qquad &\forall x \in H, \\
\Delta_{H^{*}}(f)(x \otimes y) &=& f\big(m(x \otimes y)\big), \qquad &\forall x, y \in H, \end{array}\end{equation}
\begin{equation}
u_{H^{*}} = \varepsilon, \qquad \varepsilon_{H^{*}}(f) = f(u_H),\mbox{ and } 
S_{H^{*}} = (S_H)^{*}.
\end{equation}
Hence, in this case, the dual $H^{*}$ is still a Hopf algebra, called the \emph{dual Hopf algebra} of $H$. The duality of Hopf algebras interchanges multiplication and comultiplication, turning algebra morphisms into coalgebra morphisms and vice versa. 

\subsection{The Specific Type of Hopf Algebra Considered}
A bialgebra $H=(V,*,\Delta,\eta,\varepsilon,S)$ is said \emph{graded} if the underlying space splits as $V=\bigoplus_{n\in\mathbb N}V_n$ the graded components $V_n$ are compatible with the bialgebra structure in the sense that for any $(u,v)\in V_n\times V_m$ we have $u*v\in V_{n+m}$, and for any $u\in V_n$, $\Delta(u)\in\bigoplus_{i+j=n}V_i\otimes V_j$.  In what follows, we require that $V_0=\mathbb C$, i.e. $H$ is a connected  bialgebra. In that case, there existe a unique antipode which endows $H$ with a structure of Hopf algebra (see e.g. \cite{Hopf2020} Proposition 1.4.16). The antipode is obtained by applying the recursive formula $S(u)=u$ if $u\in V_0$ and \begin{equation}S(u)=-\sum_{u^{(1)}\neq u}S(u^{(1)})u^{(2)}\end{equation} for $\Delta(u)=\sum u^{(1)}\otimes u^{(2)}$, in the Sweedler notation. It is the fact that the algebra is graded that allows the existence of an antipode, and the fact that the algebra is \emph{connected} that guarantees the unicity of the construction. To summarize, in the cases that interest us, we will not need to describe the entire Hopf algebra structure but only its structure of bialgebra.\\
In the cases we consider, we require that the dimensions of the graded spaces $V_n$ be finite. In that case, we say that $H$ is a \emph{finite-dimensional graded} Hopf algebra.
As a consequence, its graded dual algebra can be realized on the same underlying space. In fact, because the spaces associated with each grade are finite, their dual spaces are isomorphic. Therefore, they can be identified, and the product and coproduct of the dual $H^\star$ can be expressed in terms of those of the original Hopf algebra $H$.\\
Our goal is to recover certain combinatorial Hopf algebras, as well as their duals, from the Hopf algebra of diagrams. The notion of combinatorial Hopf algebras is a heuristic notion which means that the considered bialgebras are based on combinatorial objects. Out of the several approaches to a formal definition of this concept, we retain the formulation put forward by Loday and Ronco \cite{LR} which defined a combinatorial Hopf algebra as a Hopf algebra which
is free (or cofree) and equipped with a given isomorphism to the free algebra
over the indecomposables (resp. the cofree coalgebra over the primitives). For our purpose, we require the additional property that the Hopf algebra is subject to the Cartier--Milnor--Moore--Quilien theorem \cite{MM}.\\
To summarize, we will consider connected cocommutative Hopf algebras of finite type (and thus subject to the Cartier--Milnor--Moore theorem), whose underlying algebras are free, i.e. isomorphic to a (noncommutative) polynomial algebra \( \mathbb{C}\langle \Sigma \rangle \), where \( \Sigma = \bigcup_{n \geq 1} \Sigma_n \) is a graded alphabet with finite graded components \( \Sigma_n \).

\subsection{Examples of combinatorial Hopf algebras}
We present examples of Hopf algebras illustrating the prototype introduced above.
\subsubsection{Non commutative polynomials\label{subsub-ncpoly}}
First consider the bialgebra $\mathbb C\langle \Sigma\rangle$ of noncommutative polynomials that is the space generated  the elements of the free monoid $\Sigma^*$ endowed with the catenation product and the coproduct $\Delta$ defined by $\Delta(x)=x\otimes 1+1\otimes x$ for each $x\in\Sigma$. The dual bialgebra can be realized on the same space by introducing the shuffle product defined recursively by
\begin{equation}\label{eq-shufflerec}1\shuffle u=u\shuffle 1=u\mbox{ and }xu\shuffle yv=x\cdot(u\shuffle yv)+y(xu\shuffle v)\end{equation}
for any $u,v\in \Sigma^*$ and $x,y\in\Sigma$, and the  deconcatenation coproduct defined by 
\begin{equation}\Delta(u)=\sum_{u=u^{(1)}\cdot u^{(2)}}u^{(1)}\otimes u^{(2)}.\end{equation} This algebra carries a natural grading given by the word length and the pairing can be seen as the unique inner product satisfying $(u,v)=1$ if $u=v$ and $0$ otherwise, for any $u,v\in\Sigma^*$.
\subsubsection{Word symmetric Hopf algebra\label{subsub-wsym}}
The algebra of symmetric functions in noncommuting variables was first defined in 1936 \cite{Wolf}.
It was later revisited and further developed in the 2000s within the framework of combinatorial Hopf algebras (see, e.g., \cite{RS}).
Following \cite{HNT2008}, we denote this algebra by $\WSym$.
Its construction and fundamental properties are recalled below.\\
This Hopf algebra is constructed on the space generated by  elements $\Phi^\pi$, where $\pi$ is a set partition. The product is the shifted concatenation defined as 
\begin{equation}
\Phi^\pi \Phi^{\pi'}= \Phi^{\pi\cup \pi'[n]}
\end{equation}
where $\pi$ is a set partition of $\{1,\ldots,n\}$ and $\pi'[n]$ denotes the partition obtained by shifting each element of $\pi'$ by $n$ (i.e. $\pi'[n]=\{\{i_1+n,\ldots, i_k+n\}\mid \{i_1,\ldots, i_k\}\in \pi'\}$).
Hence, the algebra $(\WSym,\cdot)$ is isomorphic to the algebra of noncommutative polynomials $\mathbb C\langle \Sigma_P\rangle$, where $\Sigma_P$ is the set of letters $a_\pi$ indexed by indivisible set partitions, that is, partitions $\pi$ such that  for any $n$, $\pi=\pi_1\cup \pi_2[n]$ implies $\pi_1=\emptyset$ or $\pi_2=\emptyset$). The linear isomorphism $\iota$ is defined inductively by sending each $\Phi^\pi$ to $a_\pi$ when $\pi$ is indivisible and by $\iota(\Phi^{\pi\cup\pi'[n]})=\iota(\Phi^{\pi})\iota(\Phi^{\pi'})$. The definition of $\iota$ is straightforwardly compatible with the structure of algebra.
For instance,
\begin{equation}\iota(\Phi^{\{\{1\},\{2,4\},\{3,5\},\{6,7\}\}})=a_{\{\{1\}\}}a_{\{\{1,3\},\{2,4\}\}}a_{\{\{1,2\}\}},\end{equation} because \begin{equation}\Phi^{\{\{1\},\{2,4\},\{3,5\},\{6,7\}\}}=\Phi^{\{\{1\}\}}\Phi^{\{\{1,3\},\{2,4\}\}}\Phi^{\{\{1,2\}\}}.\end{equation}
To be more precise, for each word  $w\in\Sigma_P^*$, there exists a partition $\pi$ such that $w=\iota(\Phi^\pi)$, indeed $a_{\pi_1}\cdots a_{\pi_k}=\iota(\Phi^{\pi_1}\cdots\Phi^{\pi_k})$. Hence, the space $\mathbb C\langle\Sigma_P\rangle$ is spanned by the words $w_\pi:=\iota(\Phi^\pi)$, where $\pi$ is a set partition.\\ 
The coproduct $\Delta$ of $\WSym$ is defined by
\begin{equation}\label{eq-DeltaW}\Delta(\Phi^{\pi})
    = \sum_{I\sqcup J=\{1,\dots,n\}} \Phi^{\std(\pi|_I)} \otimes \Phi^{\std(\pi|_J)},
  \end{equation}
  where \(\pi|_I = \{\pi_{i_1}, \dots, \pi_{i_k}\}\) if \(I = \{i_1, \dots, i_k\}\), and \(\pi = \{\pi_1, \dots, \pi_n\}\); the notation \(\operatorname{std}\) denotes the standardization, which maps a partition of any finite set of integers \(E\) to a partition of \(\{1, \dots, \#E\}\) by applying the unique increasing bijection \(\varphi_E : E \to \{1, \dots, \#E\}\) to each block, elementwise; the notation $\sqcup$ is the disjoint union, i.e. $I\sqcup J=\{1,\dots,n\}$ means $I\cup J=\{1,\dots,n\}$ and $I\cap J=\emptyset$.  

By remarking that $\Delta(\Phi^{\pi\cup\pi'[n]})=\Delta(\Phi^\pi)\Delta(\Phi^{\pi'})$, the coproduct is transposed to $\mathbb C\langle\Sigma_P\rangle$ by setting
{\small
\begin{equation}\label{eq-Cop}
\begin{array}{rcl}\Delta(a_\pi)&=&\displaystyle
\sum_{I\sqcup J=\{1,\dots,n\}} w_{\std(\pi|_I)} \otimes w_{\std(\pi|_J)}\\&=&\displaystyle\sum_{\pi_1,\pi_2}\#\{(\pi'_1,\pi'_2)\mid \pi=\pi'_1\sqcup\pi'_2, \std(\pi'_1)=\pi_1\mbox{ and }\std(\pi'_2)=\pi_2\}w_{\pi_1}\otimes w_{\pi_2}.\end{array}
\end{equation}
For instance, we have
\begin{equation}\begin{array}{rcl}
\Delta(a_{\{\{1,3\},\{2,5\},\{4,6\}\}})&=&
a_{\{\{1,3\},\{2,5\},\{4,6\}\}}\otimes 1+ a_{\{\{1,2\}\}}\otimes(2a_{\{\{1,3\},\{2,4\}\}}+a^2_{\{\{1,2\}\}})\\
&&+(2a_{\{\{1,3\},\{2,4\}\}}+a^2_{\{\{1,2\}\}})\otimes a_{\{\{1,2\}\}} + 1\otimes a_{\{\{1,3\},\{2,5\},\{4,6\}\}}.
\end{array}
\end{equation}
}
The graded dual $\PiQSym$ of $\WSym$ is formally generated by the elements that form the dual basis of $\Phi^\pi$. Nevertheless, this algebra can be realized on the same space as $\WSym$ by defining its  product  by
\begin{equation}\label{prodPiQSym}
\Phi^{\pi_1}\Cup\Phi^{\pi_2}=\sum_{\pi}\#\{(\pi'_1,\pi'_2)\mid \pi=\pi'_1\sqcup\pi'_2, \pi_1=\std(\pi'_1),\pi_2=\std(\pi'_2)\}\Phi^\pi,
\end{equation}
and its coproduct is given by
\begin{equation}\label{eq-bulletw}
\Delta_\bullet(\Phi^\pi)=\sum_{\pi=\pi_1\uplus\pi_2}\Phi^{\std(\pi_1)}\otimes\Phi^{\std(\pi_2)}.
\end{equation}
It is easy to check that
\begin{equation}\langle \Phi^{\pi'}\Cup\Phi^{\pi''}\mid \Phi^{\pi}\rangle=\langle\Phi^{\pi'}\otimes\Phi^{\pi''}\mid\Delta(\Phi^\pi)\rangle\end{equation} and
\begin{equation}\langle \Phi^{\pi'}\Phi^{\pi''}\mid \Phi^\pi\rangle = \langle \Phi^{\pi'}\otimes\Phi^{\pi''}\mid\Delta_\bullet(\Phi^{\pi})\rangle,\end{equation} where $\langle\cdot\mid\cdot\rangle$ denotes the scalar product for which the basis of the $\Phi^\pi$ is orthonormal.

\subsubsection{Bi-words symmetric functions\label{subsub-BWsym}}
The Hopf algebra $\BWSym$ \cite{BCLM} is a combinatorial Hopf algebra whose basis
elements are indexed by set partitions into lists $\Pi=\{[i^1_1,\dots,i^1_{\ell_1}],\dots,[i^k_1,\dots,i^k_{\ell_k}]\}$ satisfying $k\geq 0$, $\ell_j\leq 1$ for each $1\leq j\leq k$, the integers $i^1_1,\dots,i^1_{\ell_1},\dots,i^k_1,\dots,i^k_{\ell_k}$ are distinct, and $\llbracket 1,n\rrbracket=\{i^j_t\mid 1\leq j\leq k, 1\leq t\leq \ell_j\}$ for a certain $n\in\N$; we let $\Pi\vDash n$ denote this property. As for the set partitions we define $\Pi\uplus\Pi'=\Pi\cup\{[i_1+n,\dots,i_k+n]:[i_1,\dots,i_k]\in\Pi'\}$ for $\Pi\vDash n$ and we will say that $\Pi$ is \emph{indivisible} if $\Pi=\Pi'\uplus\Pi''$ for some set partitions into lists $\Pi', \Pi''$ implies either $\Pi'=\Pi$ or $\Pi''=\Pi$.\\
The Hopf algebra $\BWSym$ is freely generated by the set $\{\Phi^\Pi\mid\Pi\vDash n, n\geq 0\}$. The product structure of this algebra is given as follows:
\begin{equation}\label{bwsymproduct}
\Phi^\Pi\Phi^{\Pi'}=\Phi^{\Pi\uplus\Pi'}
\end{equation}
and the coproduct
\begin{equation}\label{bwsymcoproduct}
\Delta(\Phi^\Pi)=\sum_{I\sqcup J=\{1,\dots,n\}}\Phi^{\std(\Pi|_I)}\otimes \Phi^{\std(\Pi|_J)},
\end{equation}
with the  usual notation \begin{equation}\std(\hat\Pi)=\Big\{[\varphi_{F}(i_1),\dots,\varphi_F(i_k)]\mid [i_1,\dots,i_k]\in \hat\Pi\Big\}\end{equation} where $F=\displaystyle\bigcup_{[i_1,\ldots,i_k]\in \hat\Pi}\{i_1,\ldots, i_k\}$.\\
These operations endow $\BWSym$ with a structure of a graded, connected, cocommutative Hopf algebra.\\
The graded dual $\BPiQSym$ of $\BWSym$ can be implemented in the same space by setting 
\begin{equation}\label{eq-shufflebw}
\Phi^{\Pi_1}\Cup\Phi^{\Pi_2}=\sum_{{\Pi'}}\#\Big\{(\Pi'_1,\Pi'_2)\mid  \Pi'_1\sqcup\Pi'_2=\Pi', \std(\Pi'_1)=\Pi_1, \std(\Pi'_2)=\Pi_2 \Big\} \Phi^{\Pi'},
\end{equation}
and 
\begin{equation}\label{eq-cocatbw}
\Delta_{\bullet}(\Phi^\Pi)=\sum_{\Pi=\Pi_1\uplus\Pi_2}\Phi^{\std(\Pi_1)\otimes\std(\Pi_2)}.
\end{equation}


Note that the algebra $\WSym$ can be realized as the subalgebra $\WSym^B$ of $\BWSym$ generated by the elements $\Phi^\Pi$, where $\Pi$ is a partition into increasing lists. The morphism $\aleph_w$ maps
$
\Phi^{\{\{i^1_1,\ldots,i^1_{\ell_1}\},\ldots,\{i^k_1,\ldots,i^k_{\ell_k}\}\}}
$
to
$
\Phi^{\{[i^1_1,\ldots,i^1_{\ell_1}],\ldots,[i^k_1,\ldots,i^k_{\ell_k}]\}},
$
where each block is strictly ordered, i.e.
$
i_1^1 < \cdots < i_{\ell_1}^1,\,
i_1^2 < \cdots < i_{\ell_2}^2,\,
\ldots,\,
i_1^k < \cdots < i_{\ell_k}^k,$
for example, the maps $\aleph_w$ sends $\Phi^{\{\{5, 2\}, \{1,3\}, \{4\}\}}$ to $\Phi^{\{[1,3], [2, 5], [4]\}}.$
Since the coproduct $\Delta$ preserves the order in the block, the algebra $\WSym^B$ is, in fact, a sub-Hopf algebra of $\BWSym$. Furthermore, comparing Equality \eqref{eq-DeltaW} and \eqref{bwsymcoproduct} we show that $\aleph_w$ is an isomorphism of Hopf algebras. By remarking that the space $\WSym^B$ is stable for the product $\Cup$ and the coproduct $\Delta_\bullet$ and comparing equalities \eqref{prodPiQSym} and \eqref{eq-bulletw} to equalities \eqref{eq-shufflebw} and \eqref{eq-cocatbw}, we show that the Hopf algebra $\PiQSym$ and $(\WSym^B,\Cup,\Delta_\bullet)$ are isomorphic.
\subsection{Cocommutative Hopf algebras of noncommutative polynomials\label{App_HopfPol}}
We now turn to the study of a more general Hopf algebraic framework that encompasses the examples introduced in the previous subsection.\\
We recall that any bigebra $H=(V,\times,\Delta)$, where $V=\bigoplus_{n\geq 0}V_n$ is graded in finite dimension, has a structure of Hopf algebra since the antipode can be deduced by construction. The dual Hopf algebra $H^\star=(V^\star,\times_{\Delta},\Delta_{\times})$ is constructed on the graded dual $V^\star=\bigoplus_{n\geq 0}V_n^\star$ of $V$ with the product $\times_\Delta$ defined by $(u\times_\Delta v,w)=(u\otimes v,\Delta w)$ and the coproduct $\Delta_\times$ defined by $(\Delta_\times w,u\otimes v)=(w,u\times v)$. The point is that, since $V$ is graded with in finite dimension, it is isomorphic to its graded dual, so one can define the dual of the Hopf algebra directly on $V$. More precisely, this means that there exists a Hopf algebra $\widetilde {H^\star}=(V,\widetilde \times,\widetilde \Delta)$ that is isomorphic to $H^*$.

A Hopf subgebra $H'$ of $H$ is both a subalgebra and a subcoalgebra of $H$. Notice that the general definition requires the antipode to map $H'$ to itself.  In our case, the Hopf algebras being graded in finite dimension, the antipode of $H'$ is deduced from its structures of algebra and coalgebra. 
Therefore, the property of stability under the antipode is automatically satisfied. By duality,
we obtain that ${H'}^\star$
is isomorphic to a quotient of $\widetilde{H^\star}$. 
We now focus on a special case where ${H'}^\star$ is isomorphic to a Hopf subalgebra of $\widetilde{H^\star}$.

We consider a connected, cocommutative Hopf algebra of finite type (and thus subject to the Cartier--Milnor--Moore theorem \cite{MM}), whose underlying algebra is the polynomial algebra \( \mathbb{C}\langle \Sigma \rangle \), where \( \Sigma = \bigcup_{n \geq 1} \Sigma_n \) is a graded alphabet with finite graded components \( \Sigma_n \).
Let us denote by $\mathcal H=(\mathbb C\langle \Sigma\rangle,\cdot,\Delta)$ such a Hopf algebra. Since the space $\mathbb C\langle\Sigma\rangle$ is graded in finite dimension, it is isomorphic to its graded dual space (i.e. the direct sum of the spaces of its graded component). Hence, we can endow $\mathbb C\langle \Sigma\rangle$ with a bigebra structure $(\mathbb C\langle \Sigma\rangle,\Cup,\Delta_\cdot)$ that is isomorphic to the dual Hopf algebra $\mathcal H^\star$ of $\mathcal H$. Considering the pairing defined on each pair of words $u_1, u_2\in\Sigma^*$ by $(u_1,u_2)=1$ if $u_1=u_2$ and  $0$ otherwise, the product $\Cup$ and the coproduct $\Delta_\cdot$ are easily constructed. Indeed, $\Delta_\cdot$ is the deconcatenation coproduct defined by
\begin{equation}
(u_1\otimes u_2,\Delta_.(w))=(u_1u_2,w)=\left\{\begin{array}{cl}
1&\mbox{if } w=u_1u_2\\
0&\mbox{if } 0.
\end{array}\right.
\end{equation}
and $\Cup$ is completely characterized by $(u_1\Cup u_2,w)=(u_1\otimes u_2, \Delta(w))$. More precisely,
\begin{lemma}\label{lemCup}
The product $u_1\Cup u_2$ of two words $u_1,u_2\in\Sigma^*$ is inductively defined by
\begin{equation}\label{eq-CuplemCup}
u_1\Cup u_2=\left\{\begin{array}{cl}
1&\mbox{if } u_1=u_2=1\\
\displaystyle\sum_{\genfrac{}{}{0pt}{}{u_1=p_1v_1, u_2=p_2v_2}{ p_1\neq 1\ \mathrm{or}\ p_2\neq 1}}\sum_{a\in \Sigma}(p_1\otimes p_2,\Delta (a))a\cdot (v_1\Cup v_2)&\mbox{otherwise}.
\end{array}\right.
\end{equation}
\end{lemma}
\begin{proof}
The base case of the induction, namely \(1 \Cup 1 = 1\), holds because \((1 \otimes 1, w) = 1\) if \(w = 1\), and vanishes otherwise, as \(\mathcal{H}\) is a connected Hopf algebra.

To address the general case, we assume that \( u_1 \neq 1 \) or \( u_2 \neq 1 \).

In that case, $(u_1\Cup u_2,w)=(u_1\otimes u_2,\Delta(w))\neq 0$ implies   $w\neq 1$. Let $w=a\cdot w'$ be a non empty word, with $a\in\Sigma$. We have
\begin{equation}\begin{array}{rcl}
(u_1\Cup u_2,w)&=&(u_1\otimes u_2,\Delta(w))\\
&=&(u_1\otimes u_2,\Delta(a)\Delta(w))\\
&=&\displaystyle\Bigg(u_1\otimes u_2,
\bigg(\sum_{\genfrac{}{}{0pt}{}{p_1,p_2\in\Sigma^*}{ p_1\neq 1\ \mathrm{or}\ p_2\neq 1}}(p_1\otimes p_2,\Delta(a))p_1\otimes p_2\bigg)\cdot \Delta(w')
\Bigg)\\
&=&\displaystyle \sum_{\genfrac{}{}{0pt}{}{u_1=p_1v_1, u_2=p_2v_2}{ p_1\neq 1\ \mathrm{or}\p_2\neq 1} }(p_1\otimes p_2,\Delta(a))(v_1\otimes v_2,\Delta(w')).
\end{array}
\end{equation}
So we obtain 
\begin{equation}
\begin{array}{rcl}
u_1\Cup u_2&=& \displaystyle\sum_{w\in\Sigma^*\setminus\{1\}}(u_1\otimes u_2,w)\\
&=&\displaystyle
\sum_{\genfrac{}{}{0pt}{}{a\in \Sigma}{ w'\in\Sigma^*}}(u_1\otimes u_2,aw')aw'\\
&=&\displaystyle \sum_{\genfrac{}{}{0pt}{}{u_1=p_1v_1, u_2=p_2v_2}{ p_1\neq 1\ \mathrm{or}\ p_2\neq 1}}\sum_{a\in\Sigma}(p_1\otimes p_2,\Delta(a))a\cdot \sum_{w'\in\Sigma^*}(v_1\otimes v_2,\Delta(w'))\\
&=&
\displaystyle\sum_{\genfrac{}{}{0pt}{}{u_1=p_1v_1, u_2=p_2v_2}{ p_1\neq 1\ \mathrm{or}\ p_2\neq 1}}\sum_{a\in \Sigma}(p_1\otimes p_2,\Delta (a))a\cdot (v_1\Cup v_2),
\end{array}
\end{equation}
as expected.
\end{proof}
\begin{example}\rmfamily
We consider the shuffle algebra on non commutative polynomials as described in Section
\ref{subsub-ncpoly}. Equality \eqref{eq-shufflerec} falls within the scope of Lemma~\ref{lemCup}. Indeed, we remark that $\Delta(x)=x\otimes1 + 1\otimes x$ (i.e. every letter is primitive), hence $(p_1\otimes p_2,\Delta(a))\neq 0$ if and only if $\{p_1,p_2\}=\{1,a\}$. From this property, the sum of Equality \eqref{eq-CuplemCup} is straightforwardly equivalent to  \eqref{eq-shufflerec}.
\end{example}
\begin{example}\rmfamily
Let us investigate the case of the combinatorial Hopf algebra $\WSym$ 
described in section \ref{subsub-wsym}.
Obviously the bigebra $\mathcal H=(\mathbb C\langle \Sigma_P\rangle,\cdot,\Delta)$ is isomorphic to $\WSym$ but some letters $a_\pi$ are not primitive.\\
If we consider the product $\Cup$ defined in Lemma \ref{lemCup} and the coproduct of deconcatenation $\Delta_\cdot$, the algebra $\mathcal H^\star$
is isomorphic to $(\mathbb C\langle \Sigma_P\rangle,\Cup,\Delta_\cdot)$ and is also isomorphic to the Hopf algebra $\PiQSym$, dual of $\WSym$. The explicit isomorphism sends each $w_\pi$ to $\Psi_\pi$. Indeed, by duality one has
\begin{equation}\label{prodSigmaPstar}
w_{\pi_1}\Cup w_{\pi_2}=\sum_{\pi}\#\{(\pi'_1,\pi'_2)\mid \pi=\pi'_1\sqcup\pi'_2, \std(\pi'_1)=\pi_1\mbox{ and }\std(\pi'_2)=\pi_2\}w_\pi,
\end{equation}
and the comparison with equality \eqref{prodPiQSym} shows $w_\pi\rightarrow\Psi_\pi$ is an isomorphism of algebra. This isomorphism is obviously compatible with the structure of cogebra.
\end{example}

Let $\Sigma' \subset \Sigma$ be a subalphabet such that $\mathbb{C}\langle \Sigma' \rangle$ is compatible with the coproduct~$\Delta$, 
in the sense that the restriction of~$\Delta$ to $\Sigma'$ maps into 
$\mathbb{C}\langle \Sigma' \rangle \otimes \mathbb{C}\langle \Sigma' \rangle$.
By duality, the space $\mathbb C\langle\Sigma'\rangle$ is stable for the product $\Cup$.
As a consequence, a direct application of Lemma \ref{lemCup} shows  the following result:
\begin{lemma}\label{lemsubdual}
The bialgebra $(\mathbb C\langle\Sigma'\rangle, \Cup,\Delta_\cdot)$ is a Hopf algebra which is isomorphic to the dual of the bialgebra $(\mathbb C\langle\Sigma'\rangle,\cdot,\Delta)$ and for any $u_1, u_2\in {\Sigma'}^*$ such that $u_1\neq 1$ or $u_2\neq 1$, we have :
\begin{equation}\label{defcupprod}
u_1\Cup u_2=\sum_{\genfrac{}{}{0pt}{}{u_1=p_1v_1, u_2=p_2v_2}{ p_1\neq 1\ \mathrm{or}\ p_2\neq 1}}\sum_{a\in \Sigma'}(p_1\otimes p_2,\Delta (a))a\cdot (v_1\Cup v_2).
\end{equation}
\end{lemma}
\begin{example}\rmfamily
Consider the subspace \begin{equation}\WSym_{\mathrm{even}}=\mathrm{span}\{\Phi^{\{\pi_1,\dots,\pi_k\}}\mid k\in\mathbb N, \forall i\in\{1,\dots,k\}, \#\pi_i\in2\mathbb N\}\end{equation} of $\WSym$.
This subspace is clearly stable by throught the product and the coproduct. So it is the Hopf subalgebra of $\WSym$ generated by the elements $\Phi^\pi$ where $\pi$
is an indivisible partition having only parts with even sizes.
Obviously, $\WSym_{\mathrm{even}}$ is isomorphic as a bialgebra to the subalgebra
 $\mathbb C\langle \Sigma'_P\rangle\subset \mathbb C\langle \Sigma\rangle$ with \begin{equation}\Sigma_P'=\{a_{\pi}\mid \pi\mbox{ is indivisible and have only parts with even sizes}\}.\end{equation} From Lemma \ref{lemsubdual}, the dual Hopf algebra of $(\mathbb C\langle \Sigma'_P\rangle,\cdot,\Delta)$ is isomorphic to the bialgebra $(\mathbb C\langle \Sigma'_P\rangle,\Cup, \Delta_\cdot)$. We deduce that $\WSym^\star_{\mathrm{even}}$ is isomorphic to the subalgebra $\PiQSym_{\mathrm{even}}$ of $\PiQSym$ generated by the elements $\Psi_\pi$ where $\pi$ is a partition having only parts with even size. For instance, compare
{\small \begin{equation}\begin{array}{l}
a_{\{\{1,2\}\}}a_{\{\{1,3\},\{2,4\}\}}\Cup a_{\{\{1,2\}\}} =
a_{\{\{1,2\}\}}a_{\{\{1,3\},\{2,4\}\}}a_{\{\{1,2\}\}}+2a_{\{\{1,2\}\}}^2a_{\{\{1,3\},\{2,4\}\}}\\+ 2a_{\{\{1,3\},\{2,4\}\}}^2+2a_{\{\{1,4\},\{2,3\}\}}a_{\{\{1,3\},\{2,4\}\}}+2a_{\{\{1,2\}\}}a_{\{\{1,3\},\{2,5\},\{4,6\}\}}\\+
3a_{\{\{1,2\}\}}a_{\{\{1,4\},\{2,5\},\{3,6\}\}}+
2a_{\{\{1,2\}\}}a_{\{\{1,4\},\{3,5\},\{2,6\}\}}
+a_{\{\{1,2\}\}}a_{\{\{2,4\},\{3,5\},\{1,6\}\}}\\
+a_{\{\{1,2\}\}}a_{\{\{1,3\},\{2,6\},\{4,5\}\}}+
2a_{\{\{1,2\}\}}a_{\{\{2,4\},\{3,6\},\{1,5\}\}}+
a_{\{\{1,2\}\}}a_{\{\{1,5\},\{2,6\},\{3,4\}\}} \\+
a_{\{\{1,2\}\}}a_{\{\{1,5\},\{4,6\},\{2,3\}\}}
\end{array}
\end{equation}
}
to
{
\begin{equation*}
\begin{array}{l}
\Psi_{\{\{1,2\},\{3,5\},\{4,6\}\}}\Psi_{\{\{1,2\}\}}=
\Psi_{\{\{1,2\},\{3,5\},\{4,6\},\{7,8\}\}}+2\Psi_{\{\{1,2\},\{3,4\},\{5,7\},\{6,8\}\}}+\\
+ 2\Psi_{\{\{1,3\},\{2,4\},\{5,7\},\{6,8\}\}}+2\Psi_{\{\{1,4\},\{2,3\},\{5,7\},\{6,8\}\}}+2\Psi_{\{\{1,2\},\{3,5\},\{4,7\},\{6,8\}\}}+\\
3\Psi_{\{\{1,2\}\},\{3,6\},\{4,7\},\{5,8\}\}}+
2\Psi_{\{\{1,2\},\{3,6\},\{5,7\},\{4,8\}\}}
+\Psi_{\{\{1,2\},\{4,6\},\{5,7\},\{3,8\}\}}+\\
\Psi_{\{\{1,2\},\{3,5\},\{4,8\},\{6,7\}\}}
+
2\Psi_{\{\{1,2\},\{4,6\},\{5,8\},\{3,7\}\}}+
\Psi_{\{\{1,2\},\{3,7\},\{4,8\},\{5,6\}\}}  +\\
\Psi_{\{\{1,2\},\{3,7\},\{6,8\},\{4,5\}\}}
\end{array}
\end{equation*}
}
\end{example}

\begin{proposition}\label{prop_subdual}
Let $\mathcal H=(V,\times,\Delta)$ be a connected Hopf algebra graded in finite dimension and denote by $\mathcal H^\star=(V^\star, \times_\Delta,\Delta_\times)$ its graded dual. We assume that $(V,\times)$ is free and generated by some $\Sigma\subset V$ (i.e., $V$ is isomorphic to $\mathbb C\langle \Sigma\rangle$). Let $\Sigma'\subset \Sigma$ such that $\Delta$ sends $\Sigma'$ to $V'\otimes V'$, where $V'$ is the underlying space of the subalgebra of $(V,\times)$ generated by $V$.\\
Denote by $\mathcal H'$ the (sub) Hopf algebra of $V$ generated by $\Sigma'$, then the dual of $\mathcal H'$ is isomorphic to the subalgebra of $\mathcal H'$ generated by the elements $a^\star$, satisfying $(a,a^*)=1$ and $(v,a^*)=0$ for $v\neq a$, for $a\in\Sigma'$. 
\end{proposition}
\begin{proof}
The hypothesis means that the Hopf algebra $\mathcal H$ is isomorphic to a Hopf algebra $\mathcal H_w=(\mathbb C\langle\Sigma\rangle,\cdot,\Delta)$ where $\cdot$ is the concatenation product. From Lemma \ref{lemCup}, the Hopf algebra $\mathcal H^*$ can be explicitly realized on the same space and so is isomorphic to the Hopf algebra $\mathcal H_w^*=(\mathbb C\langle X\rangle, \Cup,\Delta_\bullet)$, where $\Delta_\bullet$ is the coproduct of coconcatenation. Hence, by Lemma \ref{lemsubdual}, the Hopf algebra $\mathcal H'$ is isomorphic to $(\mathbb C\langle \Sigma'\rangle, \cdot,\Delta)$ and 
$\mathcal H'^\star$ is isomorphic to $(\mathbb C\langle \Sigma'\rangle, \Cup,\Delta_\bullet)$. As a consequence,  $\mathcal H'^*$ is isomorphic to the sub-Hopf algebra of $\mathcal H^*$ generated by the elements $a^*$ for $a\in\Sigma'$. 
\end{proof}
\section{The Hopf algebra of B-diagrams}\label{sec-hopfdiag}
\subsection{What B-diagrams are\label{Bdiag}}

We recall the combinatorial structure of $B$-diagrams as defined in \cite{BCL2016}, adopting a slightly modified but equivalent formulation that is more convenient for our purposes.

Let $|E|$ denote the cardinal of the set $E$, $\llbracket a,b\rrbracket:=\{a,a+1,\dots,b-1,b\}$ for any pairs of integers $a\leq b$, $E=E'\uplus E''$ when $E=E'\cup E''$ and $E'\cap E''=\emptyset$.
\begin{definition}
Let $\lambda=[\lambda_1,\cdots,\lambda_k]\in\mathbb N\setminus\{0\})^n$ be a vector, for any $k\in\llbracket 1,|\lambda|\rrbracket$, where $|\lambda|=\lambda_1+\cdots+\lambda_n$, we consider the unique pair $(\block_\lambda(k),\numb_{\lambda}(k))$ such that $\block_\lambda(k)\in\llbracket 1,n\rrbracket$, $\numb_\lambda(k)\in\llbracket 1,\lambda_{\block_\lambda(k)}\rrbracket$ and
$k=\lambda_1+\cdots+\lambda_{\numb_\lambda(k)-1}+\numb_\lambda(k)$.
\end{definition}
For convenience, any partial map $\phi:\llbracket 1,n\rrbracket\rightarrow E$ (for some $n \in \mathbb N$, and some set $E$) is represented as a word $w_1\cdots w_n$  over the alphabet $E\cup\{\sqcup\}$, where $w_i=\phi(i)$ when $i\in\phi^{-1}(E)$ and $w_i=\sqcup$ otherwise. For instance, the word $\sqcup3\sqcup21$ represents the partial map sending $2$ to $3$, $4$ to $2$ and $5$ to $1$; the elements $1$ and $3$ have no image by $\phi$.
\begin{example}
Consider $\lambda=[4,2,1,5,2]$. We represent $\block_\lambda$ and $\numb_\lambda$ as the following words :
\begin{equation*}
\block_\lambda=\overbrace{1111}^4\overbrace{22}^2\overbrace{3}^1\overbrace{44444}^5\overbrace{55}^2\mbox{ and }\numb_\lambda=\overbrace{1234}^4\overbrace{12}^2\overbrace{1}^1\overbrace{12345}^5\overbrace{12}^2.
\end{equation*}
\end{example}
\begin{definition}\label{DBDiag}
A \emph{$B$-diagram} is a $5$-tuple $G=(n, \lambda,\varphi, F^\uparrow, F_\downarrow)$ such that 
\begin{enumerate}
 \item $n$ is an integer,
 \item $\lambda=[\lambda_1,\dots,\lambda_n]$ with $\lambda_i\in\N\setminus\{0\}$ for each $i$,
 \item $\varphi$ is a bijection from  $E^\uparrow(G)$ to $E_\downarrow(G)$, for some $E^\uparrow(G),\  E_\downarrow(G)\subset \llbracket 1,|\lambda|\rrbracket$, satisfying  $\block_\lambda(a)<\block_\lambda(\varphi(a))$,  for any $a\in E^\uparrow(G)$.
 \item $F^\uparrow\subset \llbracket 1,|\lambda|\rrbracket\setminus E^\uparrow(G)$, $F_\downarrow\subset\llbracket 1,|\lambda|\rrbracket\setminus E_\downarrow(G)$\end{enumerate}
The integer $n$ is called the \emph{number of vertices} of $G$ and is also denoted by $|G|$. The number  $\omega(G)=|\lambda|$ is called the \emph{weight}. The set of the \emph{edges} of $G$ is $E(G)=\{(a,\varphi(a))\mid a\in E^\uparrow(G)\}$ with $a\in E^\uparrow(G)$. The number of edges is denoted by $\tau(G)=\#\{(a,\varphi(a))\mid a\in E^\uparrow(G)\}$.
The elements of $F^\uparrow$ (resp. $F_\downarrow$) are called the \emph{outer} (resp. \emph{inner}) \emph{free half-edges}. We set $f^\uparrow(G)=\#F^\uparrow(G)$ and 
$f_\downarrow(G)=\#F_\downarrow(G)$.\\
The set of \emph{cut outer} (resp. \emph{inner}) \emph{half edges} of $G$ is $C^\times(G)=\llbracket1,|\lambda|\rrbracket\setminus (E^\uparrow(G)\cup F^\uparrow$ (resp. 
$C_\times(G)=\llbracket1,|\lambda|\rrbracket\setminus (E_\downarrow(G)\cup F_\downarrow$).
We set also $c^\times(G)=\#C^\times(G)$ and $c_\times(G)=\#C_\times(G)$.
Where there is no ambiguity, references to the name of the diagram ($G$) can be omitted.\\
When needed we also use the notation $F^\uparrow(G)=F^\uparrow$ and $F_\downarrow(G)=F_\downarrow$.\\
We denote by $\mathbb B$ the set of $B$-diagrams.
\end{definition}

Graphically, a $B$-diagram is represented as a graph with $n$ vertices. The vertex $i$ has exactly $\lambda_i$ inner (resp. outer) cut or not cut half-edges  labeled
 by $\llbracket \lambda_1+\cdots+\lambda_{i-1}+1,\lambda_1+\cdots+\lambda_i\rrbracket$. Remark that for any $a\in  \llbracket \lambda_1+\cdots+\lambda_{i-1}+1,\lambda_1+\cdots+\lambda_i\rrbracket$, we have $\block_\lambda(a)=i$. The inner and the outer cut half edges are denoted by $\times$. An element of $E$ is represented by an edge relying an outer uncut half edge $a$ of a vertex $i$ to an inner uncut half edge $b$ of a vertex $j$ with $i<j$.

The number of $B$-diagrams $G$ such that $\omega(G)=p$ and $f^\uparrow(G) = q$ is given by the following induction (Sloane's sequence A265199\cite{Sloane}),
\begin{equation}\label{enumdiag}
d_{p,q}=\sum_{i=1}^p\sum_{j=0}^i\sum_{k=0}^i\sum_{\ell=0}^j \ell!\binom j\ell\binom {q-k+\ell}\ell\binom ij\binom ik d_{p-i,q-k+\ell},
\end{equation}
with the special cases $d_{0,0}=1$ and $d_{p,q}=0$ if $p,q\leq 0$ and $(p,q)\neq (0,0)$, such that 
$\omega(G)$ is the number of half-edges and $f^\uparrow(G)$ is the number of free (non used) outer half-edges.

\begin{example}
For instance, the $B$-diagram \begin{equation*}G=(4,[3,2,2,1],7486\sqcup\sqcup\sqcup\sqcup,\{5,6\}, \{1,2,3,5\})\end{equation*} is represented in Figure \ref{DiagExpPer1}.

 Here we have $|G|=4$, $\omega(G)=8$, $E^\uparrow=\{1,2,3,4\}$, $E_\downarrow=\{4,6,7,8\}$,
 \begin{equation}[\varphi(1),\varphi(2),\varphi(3),\varphi(4)]=[7,4,8,6],\end{equation}
  $C^\times=\{7,8\}$, and $C_\times=\emptyset$.
\begin{figure}[ht]
\centering
\includegraphics[scale=0.25]{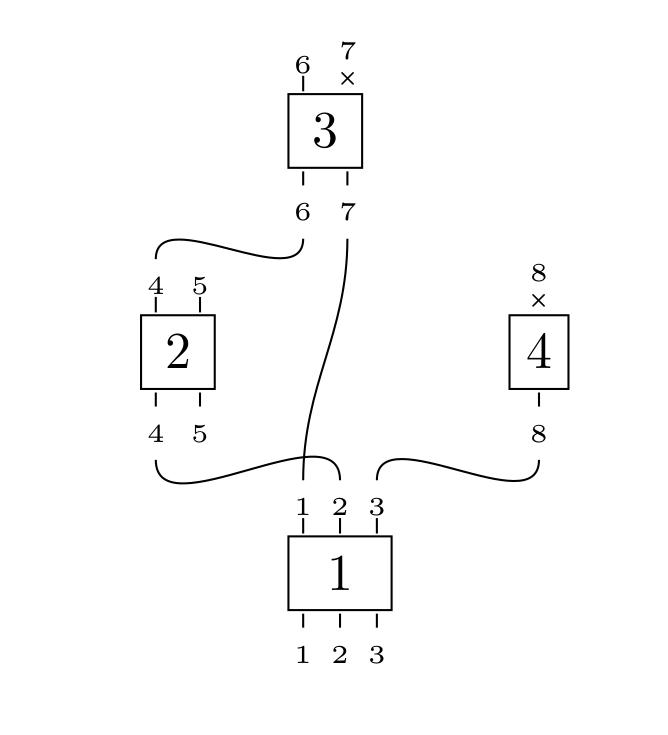}
\caption{The $B$-diagram $(4,[3,2,2,1],7486\sqcup\sqcup\sqcup\sqcup,\{5,6\}, \{1,2,3,5\})$ .}
\label{DiagExpPer1}
\end{figure}
\end{example}

\subsection{The Hopf algebra of B-diagrams\label{sec-HopfB}}

We now recall how to interpret in an algebraic way the constructions on $B$-diagrams.
\subsubsection{Algebraic structure}
\begin{definition}
Let $G=(n,\lambda,\varphi,F^\uparrow,F_\downarrow)$ and $G'=(n',\lambda',\varphi',F'^\uparrow,F'_\downarrow)$ be two $B$-diagrams. For any $k\geq 0$, any strictly increasing sequence $a_1<\dots<a_k$ in $F^\uparrow$ and any $k$-tuple of distinct integers $b_1,\dots,b_k$ in $F_\downarrow$, we define the \emph{composition} $\mcomp{a_1,\dots,a_k}{b_1,\dots,b_k}$ by
\begin{equation}
\begin{array}{c}
G'\\
\mcomp{a_1,\dots,a_k}{b_1,\dots,b_k}\\
G
\end{array}=G''=(n+n',[\lambda_1,\dots,\lambda_n,\lambda'_1,\dots,\lambda'_n],\varphi'',F''^\uparrow,F''_\downarrow)\end{equation} with
\begin{enumerate}
\item $\varphi: E^\uparrow(G'')\rightarrow E_\downarrow(G'')$, where 
\begin{itemize}
    \item $E^\uparrow(G'')=E^\uparrow(G)\cup\{i+\omega(G)\mid i\in E^\uparrow(G')\} \cup \{a_1,\dots,a_k\}$,
\item $E_\downarrow(G'')=E_\downarrow(G)\cup\{i+\omega(G)\mid i\in E_\downarrow(G')\} \cup \{b_1+\omega(G),\dots,b_k+\omega(G)\}$,
\item
$\varphi''(i)=\varphi(i)$ for any $1\leq i\leq \omega(G)$, 
\item
$\varphi''(i+\omega(G))=\varphi'(i)+\omega(G)$ for any $1\leq i\leq \omega(G')$, and
\item $\varphi''(a_i)=b_i+\omega(G)$ for any $1\leq i\leq k$.\end{itemize}
\item $F''^\uparrow=(F^\uparrow\setminus \{a_1,\cdots,a_k\} \cup \{i+\omega(G)\mid i\in F'^\uparrow\}$, and
\item $F''_\downarrow=F_\downarrow\cup  \{i+\omega(G)\mid i\in F'_\downarrow\setminus\{b_1,\cdots,b_k\}\}$.
%
\end{enumerate}

We define also \begin{equation}
\begin{array}{c}
G'\\
\comp{a_1,\dots,a_k}{b_1,\dots,b_k}\\
G
\end{array}=\left\{
\begin{array}{ll}
\begin{array}{c}
G'\\
\mcomp{a_1,\dots,a_k}{b_1,\dots,b_k}\\
G
\end{array}&\mbox{ if } a_1<\cdots<a_k\in F^{\uparrow}(G)
\mbox{ and } b_1,\dots, b_k\in F_{\downarrow}(G')\\
0&\mbox{ otherwise}
\end{array}
\right.
\end{equation}
\end{definition}
Note that the special case where $k=0$ corresponds to a simple juxtaposition $(|)$ of the $B$-diagrams (see Figure \ref{comp2} for an example). We set $G'|G'':=\begin{array}{c} G''\\\medstar\\G'\end{array}$.
 The operation $|$ endows the set of the $B$-diagrams $\mathbb B$ with a structure of  monoid  whose unity is $\varepsilon$.

\begin{figure}[ht]
\centering
\includegraphics[scale=0.24]{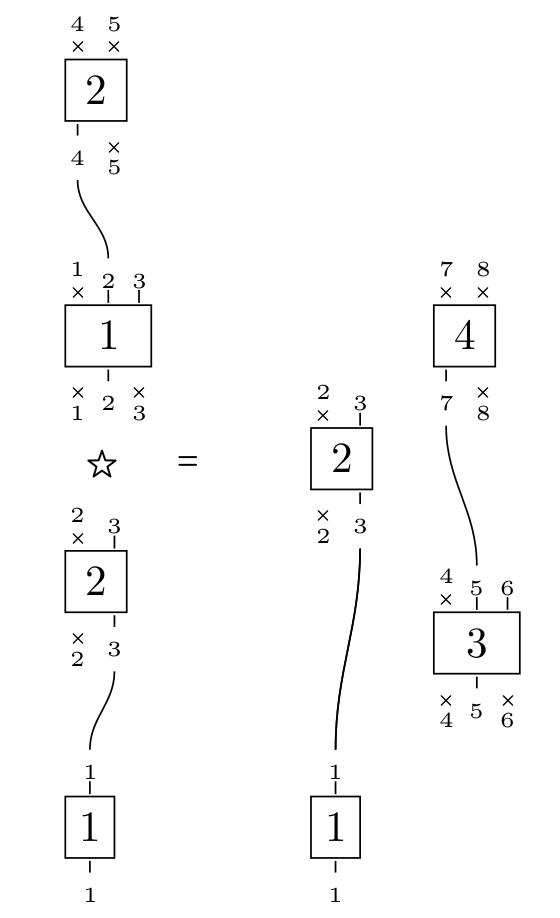}
\caption{An example of composition when $k=0$ \label{comp2}}
\end{figure} 
 
\begin{definition}
A $B$-diagram $G$ is \emph{indivisible} if $G=G'|G''$ implies $G'=G$ or $G''=G$. Let $\mathbb G$ denote the set of indivisible $B$-diagrams. 
\end{definition}
It is worth noting that the space $\mathbb C[\mathbb B]$ of $B$-diagrams, equipped with the juxtaposition operation extended by distributivity to any linear combinations, is an algebra $\mathcal B$ isomorphic to the free algebra $\mathbb C\langle \mathbb B_{i}\rangle$ generated by the set $\mathbb B_{i}$ of indivisible $B$-diagrams.

For any pair of $B$-diagrams $(G,G')$, we define
{\small
\begin{equation}\label{ProductBDiag1}
G\medstar G':=\Bigg\{\begin{array}{c}
G'\\
\mcomp{a_1,\dots,a_k}{b_1,\dots,b_k}\\
G
\end{array}\Bigg| a_1<\dots<a_k\in F^\uparrow(G), b_1,\dots,b_k\in F'_\downarrow(G')\mbox{ distinct and }k\geq 0\Bigg\}
\end{equation}
The algebra of $B$-diagrams $\mathcal B$ was defined in \cite{BCL2016}, we recall its definition here: this algebra is formally generated by the $B$-diagrams $G$ using the product $\star$ defined by
\begin{equation}\label{starHW2}
G'\star G''=\sum_{\genfrac{}{}{0pt}{}{1\leq a_1<\cdots<a_k\leq\omega(G')}{ 1\leq b_1,\dots,b_k\leq\omega(G''),\mbox{ \tiny distinct}}} \begin{array}{c}G''\\\comp{a_1,\dots,a_k}{b_1,\dots,b_k}\\G'\end{array}.
\end{equation}
}
We recall \cite{BCL2016} that the algebra $\mathcal B$ is free  on the indivisible $B$-diagrams and the empty $B$-diagram $\varepsilon=(0,[],\emptyset\rightarrow\emptyset,\emptyset,\emptyset)$ of size zero is the unity of this algebra.
 \subsubsection{Coproduct and structure of Hopf algebra}
 Let us recall some notions defined in \cite{BCL2016}.
\begin{definition}
A $B$-diagram $G$ is \emph{connected} if and only if for each $1\leq i<j\leq |G|$, there exists a sequence
of edges $(a_1,b_1),\dots,(a_k,b_k)\in E(G)$ satisfying
\begin{enumerate}
\item $i\in\{\block_\lambda(a_1),\block_\lambda(b_1)\}$ and $j\in\{\block_\lambda(a_k),\block_\lambda(b_k)\}$,
\item for each $1\leq \ell<k$ one has
{\small \begin{equation}\block_\lambda(a_\ell)\in\{\block_\lambda(a_{\ell+1}),\block_\lambda(b_{\ell+1})\}\end{equation}} or
{\small \begin{equation} \block_\lambda(b_\ell)\in\{\block_\lambda(a_{\ell+1}),\block_\lambda(b_{\ell+1})\}.\end{equation}}
\end{enumerate}
\end{definition}
\begin{definition}\label{subdiag}
Let $G=(n,\lambda,\varphi,F^\uparrow,F_\downarrow)$ be a $B$-diagram. A \emph{sub-$B$-diagram} of $G$ is completely determined by a sequence $1\leq i_1<\cdots<i_{n'}\leq n$. More precisely, we define the $B$-diagram $G[i_1,\dots,i_{n'}]=(n',\lambda',\varphi',F'^\uparrow,F'_\downarrow)$ in the following way:  we consider the set
\begin{equation}\hat H=\bigcup_{k=1}^{n'}\{a\in\llbracket 1,\omega(G)\rrbracket\mid i_k\in \block(a)\}\end{equation} and the only increasing bijection $\phi: \hat H\rightarrow \llbracket1,\lambda_{i_1}+\cdots+\lambda_{i_{n'}}\rrbracket$. Hence, we set
\begin{enumerate}
\item $\lambda'=[\lambda_{i_1},\dots,\lambda_{i_{n'}}]$,
\item $\varphi'$ is the only   bijection $\phi(E^\uparrow(G)\cap \hat H)\rightarrow \phi(E_\downarrow(G)\cap \hat H)$ such that $\varphi'\circ\phi=\phi\circ\varphi$.
\item $F'^\uparrow=\phi\big(F^\uparrow\cap \hat H\big)$,\item $F'_\downarrow=\phi\big(F_\downarrow\cap \hat H\big)$.
\end{enumerate}
If $I=[i_1,\dots,i_{n'}]$ is an increasing sequence of vertices of $G$, we denote by  $\complement_G I=[j_1,\dots,j_{n-n'}]$, the only increasing sequence, called \emph{complement of $I$ in $G$}, such that $\{j_1,\dots,j_{n-n'}\}=\llbracket 1,|G|\rrbracket\setminus \{i_1,\dots,i_{n'}\}$.
\end{definition}
Informally, a sub-diagram is constructed from a subset \(E'\) of vertices by keeping all and only the edges joining two vertices of $E'$, and then relabelling the vertices from $1$.
\begin{example}
Let $G$ be the $B$-diagram of Figure \ref{DiagExpPer1}. Set $i_1=1$ and $i_2=3$, we consider $G[1,3] = (n',\lambda', \varphi', F'^\uparrow, F'_\downarrow)$. Following Definition \ref{subdiag}, we have
\begin{enumerate}
\item $\lambda'=[3,2]$,
\item $\varphi'$ sends $1$ to $5$ and respectively $2,3,4,5$ to $\sqcup$, hence $\varphi' = (5\sqcup\sqcup\sqcup\sqcup)$,
\item $F'^\uparrow =\phi\big(\{2,3,6\}\big)=\{(2,3,4)\}$,
\item $F'_\downarrow =\phi\big(\{1,2,3,6\}\big)=\{1,2,3,4\}$.
\end{enumerate}
 We find $G[1,3]=(2,[3,2],5\sqcup\sqcup\sqcup\sqcup,\{2,3,4\},\{1,2,3,4\})$ (see Figure \ref{ExpSubdiag}).
\begin{figure}[ht]
\centering
\includegraphics[scale=0.2]{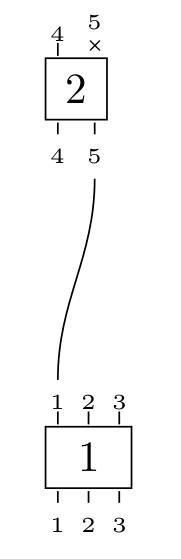}

\caption{The $B$-diagram $G[1,3]$ of $G= (4,[3,2,2,1],7486\sqcup\sqcup\sqcup\sqcup,\{5,6\}, \{1,2,3,5\})$\label{ExpSubdiag}}
\end{figure}
\end{example}
\begin{remark}
It is immediate to check that every sub-diagram of a $B$-diagram is itself a $B$-diagram.
\end{remark}
\begin{definition}\label{DefConnectecComponent}A \emph{connected component} of a $B$-diagram $G$ is a sequence $i_1<\cdots<i_{n'}$ such that $G[i_1,\dots,i_{n'}]$ is a connected sub $B$-diagram which is maximal in the sense that if we add any vertex $i$ in the sequence $i_1<\cdots<i_{n'}$ then we obtain a sequence $i'_1<\cdots<i'_{n'+1}$ such that $G[i'_1,\dots,i'_{n'+1}]$ is not connected. Let $\mathrm{Connected}(G)$ denote the set of the connected components of $G$.
\end{definition}

\begin{example}\rmfamily
The $B$-diagram in Figure \ref{DiagExpPer1} is connected whilst the $B$-diagram in Figure \ref{bdiag3122} has two connected components $[1,2,3]$ and $[4]$.
\begin{figure}[ht]
\centering
\includegraphics[scale=0.23]{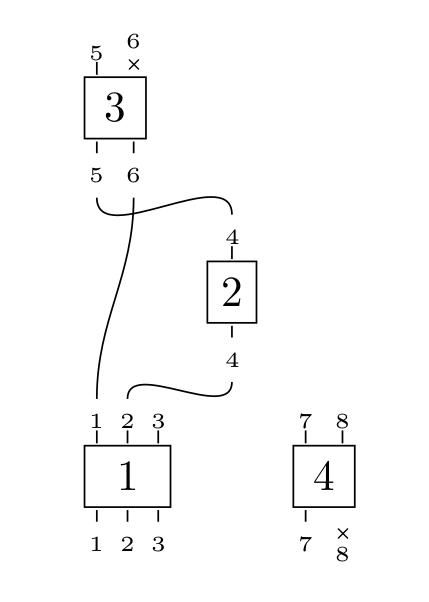}
\caption{A non connected $B$-diagram\label{bdiag3122}}
\end{figure}
\end{example}
\begin{definition}
A sequence $1\leq i_1<\dots<i_{n'}\leq n$ is \emph{isolated} in $G$  if for each $(a,b)\in E$ implies that $\block_\lambda(a)$ and $\block_\lambda(b)$ are both in
$\{i_1,\dots,i_{n'}\}$ or both in $\complement_G \{i_1,\dots,i_{n'}\}$.\\
 In other words, if $1\leq i_1<\dots<i_{n'}\leq n$ is isolated in $G$ if and only if this sequence splits into an union of connected components of $G$.\\ We denote by $\mathrm{Iso}(G)$ the set of isolated sequences of $G$ and set 
 \begin{equation}\mathrm{Split}(G)=\big\{\big(I,{\complement}_GI\big)\mid I\in\mathrm{Iso}(G)\big\}\end{equation}
 \end{definition}

  In \cite{BCL2016}, we consider  the linear map $\Delta:\mathcal B\longrightarrow\mathcal B\otimes\mathcal B$ defined by
\begin{equation}\label{DeltaDef1}
\Delta(G)=\sum_{I\in\mathrm{Iso}(G)}G[I]\otimes G[{\complement}_G I].\end{equation}
Equivalently, 
\begin{equation}\label{DeltaDef2}
\Delta(G)=\sum_{\mathcal I\subset\mathrm{Connected}(G)}G\Bigg[\bigcup_{I\in\mathcal I}I\Bigg] \otimes 	G\Bigg[{\complement}_G\bigcup_{I\in\mathcal I}I\Bigg].\end{equation}
One has
\begin{equation}\label{DeltaDef3}
\Delta(G)=\sum_{\genfrac{}{}{0pt}{}{I\cup J=\llbracket 1,|G|\rrbracket}{ I\cap J=\emptyset}}G\langle I\rangle\otimes G\langle J\rangle,\end{equation}
where
\begin{equation}
G\langle I\rangle=\left\{\begin{array}{ll}G[I]&\mbox{ if }I\in\mathrm{Iso}(G)\\0&\mbox{ otherwise} \end{array}\right. .
\end{equation}
Obviously, $\Delta$ is a coassociative, cocomutative product and $\epsilon$ is its counity.

We set $\mathcal B_k=\mathrm{span}\{G\mid \omega(G)=k\}$.  Note that $\mathcal B$ splits into the direct sum $\mathcal B=\bigoplus_k\mathcal B_k$ and the dimension of each space $\mathcal B_k$ is finite. The algebra $(\mathcal B,\star)$ is a graded with finite dimensional graded component.
The unit of this algebra is the empty $B$-diagram $\varepsilon$.

We proved in \cite{BCL2016} that $\mathcal B$ is a graded bialgebra with finite dimensional graded component. Hence, $(\mathcal B,\star,\Delta)$ is a graded Hopf algebra.
\subsection{The dual Hopf algebra of B-diagrams}\label{subsec-dualbalg}
Since the dimension of the space $\mathcal B_n$ generated by the $B$-diagrams with exactly $n$ half edges is finite, it is isomorphic to its dual space $\mathcal B^*_n$. The space $\mathcal B^*_n$ is generated by the endomorphisms $D_G$ satisfying $D_G(G')=\delta_{G,G'}$,  for any pair of $B$-diagrams $(G,G')$.\\
Let us set $\mathcal B^*:=\bigcup_n \mathcal B^*_n$. The goal of this section is to describe a Hopf algebra structure on $\mathcal B^*$ which dualizes the Hopf algebra of $\mathcal B$.
\subsubsection{Permuting vertices in a B-diagram}

In the aim to define properly an action of the symmetric group on the vertices, we need to consider more general diagrams. 
\begin{definition} We define a \emph{(simple-loop)-free diagram}, or $F$-diagram, as a 5-tuple $G=(n,\lambda,\varphi,F^\uparrow,F_\downarrow)$ that meets all the conditions of  Definition \ref{DBDiag} except point 3, in which  $\block_\lambda(a)<\block_\lambda(\varphi(a))$ is replaced by
$\block_\lambda(a)\neq \block_\lambda(\varphi(a))$.
\end{definition} 
\begin{example}
Figure \ref{counterexGs} shows a $F$-diagram which is not a $B$-diagram.\end{example}
We consider the action of the symmetric group $\S_n$ on the component of a vector $\lambda=[\lambda_1,\cdots,\lambda_n]$, i.e.
\begin{equation}\lambda^\sigma=[\lambda_{\sigma(1)},\dots,\lambda_{\sigma(n)}],\end{equation}
or equivalently $\lambda^\sigma_i=\lambda_{\sigma^{-1}(i)}$, for any $i\in\llbracket 1,n\rrbracket$. 
Let $\sigma[\lambda]\in\S_{|\lambda|}$ be the permutation such that \begin{equation}[v^1_{1},\dots,v^1_{\lambda_1},\dots,v^n_{1},\dots,v^n_{\lambda_n}]^{\sigma[\lambda]}=[v^{\sigma(1)}_1,\dots,v^{\sigma(1)}_{\lambda_{\sigma(1)}},\dots,v^{\sigma(n)}_1,\dots,v^{\sigma(n)}_{\lambda_{\sigma(n)}}].\end{equation}
If $\varphi$ is a partial function represented by the word $w=w_1\cdots w_n$ and $\sigma\in\S_n$ is a permutation, we denote by $\varphi^\sigma$ the partial function represented by $w^{\sigma}=w_{\sigma(1)}\cdots w_{\sigma(n)}.$\\
With this notation, we have 
\begin{equation}\label{blocksigma}\block_{\lambda^\sigma}=(\sigma^{-1}(\block_\lambda(1))\cdots\sigma^{-1}(\block_\lambda(|\lambda|))^{\sigma[\lambda]}\end{equation}
and 
\begin{equation}\label{numsigma}\numb_{\lambda^\sigma}=\numb_{\lambda}^{\sigma[\lambda]}.\end{equation}

\begin{example}
    For instance, let
    $\lambda=[3,2,2,1]$ and $\sigma=2314\in \S_4$.
    One has $\lambda^\sigma=[2,2,3,1]$ and $\sigma[\lambda]=45671238$. 
    We have
     \begin{equation*}\begin{array}{r}(\sigma^{-1}(1)\sigma^{-1}(1)\sigma^{-1}(1)\sigma^{-1}(2)\sigma^{-1}(2)\sigma^{-1}(3)\sigma^{-1}(3)\sigma^{-1}(4))^{45671238}=\\
     (33311224)^{45671238}=11223334=\block_{\lambda^\sigma}.\end{array}\end{equation*} 
    Also, we have  $\numb_\lambda^{\sigma[\lambda]}=(12312121)^{45671238}=12121231=\numb_{\lambda^\sigma}$.
\end{example}
\begin{definition}
Let $G=(n,\lambda,\varphi,F^\uparrow,F_\downarrow)$ be a $F$-diagram and $\sigma\in \S_n$. We define $G^\sigma:=(n,\lambda^\sigma,\widetilde\varphi, {\widetilde F}^\uparrow,{\widetilde F}_\downarrow)$ as the $5$-tuple satisfying
\begin{itemize}
\item $i\in \widetilde F^\uparrow$ if and only if $\sigma[\lambda](i)\in F^\uparrow$,
\item $i\in \widetilde F_\downarrow$ if and only if $\sigma[\lambda](i)\in F_\downarrow$,
\item $\widetilde \varphi(i)=j$ if and only if $\varphi(\sigma[\lambda](i))=\sigma[\lambda](j)$. In another words, $\widetilde\varphi$ is represented by the word $(\mu(\varphi(1))\cdots \mu(\varphi(\omega(G))))^{\sigma[\lambda]}$ where $\mu(\sqcup)=\sqcup$ and $\mu(i)=(\sigma[\lambda])^{-1}(i)$ if $i\in\llbracket1,\omega(G)\rrbracket$.
\end{itemize}
\end{definition}
Remark that if $G$ is a $B$-diagram, $G^\sigma$ is not necessarily a $B$-diagram.
%
\begin{example}
Let us denote by $G$ the $B$-diagram defined in 
figure \ref{DiagExpPer1} and set $\sigma=(123)\in\S_4$. We have
\begin{equation*}G^\sigma=(n,[2,2,3,1],3\sqcup\sqcup\sqcup418\sqcup,\{2,3\},\{2,5,6,7\}),\end{equation*} as shown in Figure \ref{counterexGs}.
Indeed $[3,2,2,1]^{(123)}=[2,2,3,1]$, $\sigma[\lambda]=45671238$, $\sigma[\lambda]^{-1}=56712348$ and hence \begin{equation*}
\widetilde\varphi=(4183\sqcup\sqcup\sqcup\sqcup)^{45671238}=3\sqcup\sqcup\sqcup418\sqcup,
\end{equation*}
$\widetilde F^{\uparrow}=\{\sigma[\lambda]^{-1}(5),\sigma[\lambda]^{-1}(6)\}=\{2,3\}$, and
\begin{equation*}\widetilde F_{\downarrow}=\{\sigma[\lambda]^{-1}(1),\sigma[\lambda]^{-1}(2),\sigma[\lambda]^{-1}(3),\sigma[\lambda]^{-1}(5)\}=\{5,6,7,2\}.\end{equation*} We notice that it
is clearly not a $B$-diagram.
\begin{figure}[ht]
\centering
\includegraphics[scale=0.25]{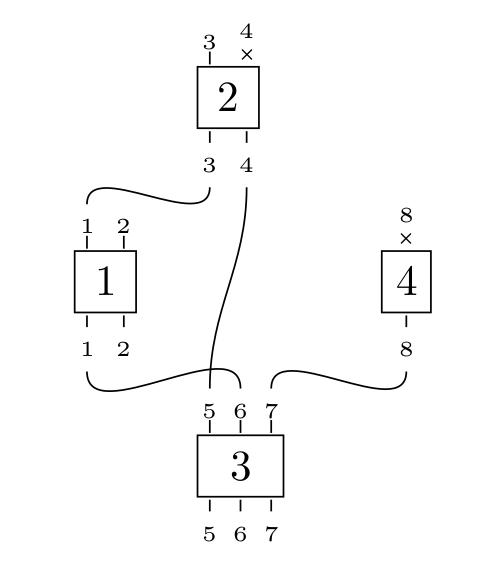}
\caption{The $5$-tuple $(4,[3,2,2,1],7486\sqcup\sqcup\sqcup\sqcup,\{5,6\},\{1,2,3,5\})^{(123)}$ is not a $B$-diagram.}
\label{counterexGs}
\end{figure}
\end{example}
The notion of sub-diagram extends naturally to $F$-diagrams. The following result shows the compatibility between the sub-diagram and the action of the symetric group.
\begin{lemma}
If $G'=G[i_1,\dots,i_{n'}]$ is a sub-diagram of $G$ then \begin{equation}G^\sigma[\sigma^{-1}(i_1),\dots,\sigma^{-1}(i_{n'})]=G'.\end{equation}
\end{lemma}
\begin{proof}
Let $G=(n,\lambda,\varphi,F^\uparrow,F_\downarrow)$ be a $F$-diagram and 
\begin{equation}G'=G[i_1,\cdots,i_{n'}]=(n',\lambda',\varphi',F'^\uparrow,F'_\downarrow)\end{equation} be a sub-diagram of $G$. Let us set $G^\sigma[\sigma^{-1}(i_1),\cdots,\sigma^{-1}(i_{n'})]=(n',\lambda'',\varphi'',F''^\uparrow,F''_\downarrow)$. For any $k\in\{1,\dots,n\}$, we have $\lambda^\sigma_{k}=\lambda_{\sigma(k)}$. Hence, the result follows from \begin{equation}\lambda''=[\lambda^\sigma_{\sigma^{-1}(i_1)},\dots,\lambda^\sigma_{\sigma^{-1}(i_{n'})}]=[\lambda_{i_1},\dots,\lambda_{i_n}]=\lambda'.\end{equation}
\end{proof}
We recall the definition of the shuffle product on sets
\begin{equation}
\left\{\begin{array}{lll}  \epsilon \shuffle u = &  u \shuffle \epsilon = \{u\} &\mbox{ if }\epsilon \mbox{ denotes the empty set } \\
(au) \shuffle (bv) = &  a(u \shuffle bv) \cup  b(au \shuffle v)  &\mbox{ otherwise} \end{array}\right.
\end{equation}
\begin{lemma}\label{Relation1}
Let $(a,b) \in \{ 1, 2, \dots ,n\}^2 \cup \{ n+1, n+2, \dots, n+n'\}^2$ and $\sigma\in 12\cdots n\shuffle (n+1)\cdots (n+n')$. We have $a<b$ if and only if $\sigma^{-1}(a) < \sigma^{-1}(b)$.
\end{lemma}
In the following lemma, we consider a permutation as words.
\begin{lemma}\label{DualBDiag}
Let $G$ and $G'$ be two $B$-diagrams and \begin{equation}\sigma \in 1\cdots |G|\shuffle (|G|+1)\cdots (|G|+|G'|)\end{equation} then $(G|G')^{\sigma}$ is a $B$-diagram such that \begin{equation}(G|G')^{\sigma}[\sigma(1)\cdots \sigma(|G|)]=G\mbox{ and  }(G|G')^{\sigma}[\sigma(|G|+1)\cdots \sigma(|G|+|G'|)]=G'.\end{equation}
\end{lemma}
\begin{proof}
Let $G=(n,\lambda,\varphi,F^\uparrow,F_\downarrow)$ and
$G'=(n',\lambda',\varphi',F'^\uparrow,F'_\downarrow)$. For any $\sigma\in\S_{n+n'}$, we have  $(G|G')^\sigma=(n+n',(\lambda\lambda')^\sigma,\varphi'',F''^\uparrow,F''_\downarrow)$ with
\begin{equation}\varphi''= (\mu(\varphi(1))\cdots\mu(\varphi(\omega(G)))\mu(\varphi'(1)+\omega(G))\cdots\mu(\varphi'(\omega(G'))+\omega(G)))^{\sigma [ \lambda \lambda']},\end{equation}
where $\mu(\sqcup)=\sqcup$ and $\mu(i)=\sigma[\lambda\lambda']^{-1}(i)$ for any $i\in\llbracket1,\omega(G)+\omega(G')\rrbracket$,
\begin{equation}F''^\uparrow =\{\sigma[\lambda\lambda']^{-1}(i)\mid i\in F^\uparrow\}\cup
\{\sigma[\lambda\lambda']^{-1}(i+\omega(G))\mid i\in F'^\uparrow\},\end{equation}
and \begin{equation}F''_\downarrow=\{\sigma[\lambda\lambda']^{-1}(i)\mid i\in F_\downarrow\}\cup
\{\sigma[\lambda\lambda']^{-1}(i+\omega(G))\mid i\in F'_\downarrow\}.\end{equation}
We suppose now that $\sigma\in 1\cdots n\shuffle (n+1)\cdots (n+n')$.
One has to prove that for any $a\in\llbracket 1,\omega(G)+\omega(G')\rrbracket^2$, $\block_{(\lambda\lambda')^\sigma}(a)<\block_{(\lambda\lambda')^\sigma}(\varphi''(a))$ when $\varphi''(a)\neq\sqcup$. 
Let $a\in\llbracket 1,\omega(G)+\omega(G')\rrbracket$. By construction, we have
{\small\begin{equation}\label{Eq_blocklambda''}
\block_{(\lambda\lambda')^\sigma}(a)=\left\{\begin{array}{ll}
\sigma^{-1}(\block_\lambda(\sigma[\lambda\lambda'](a)))\qquad\mbox{ if } \sigma[\lambda\lambda'](a)\in\llbracket 1,\omega(G)\rrbracket\\
\sigma^{-1}(\block_{\lambda'}(\sigma[\lambda\lambda'](a)-\omega(G))+n)\\\qquad\qquad\qquad\mbox{ if } \sigma[\lambda\lambda'](a)\in\llbracket \omega(G)+1,\omega(G)+\omega(G')\rrbracket
\end{array}
\right.
\end{equation}}
and, when  $\varphi''(a)\neq\sqcup$,
{\small \begin{equation}\label{eq_phi''}
\varphi''(a)=\left\{
\begin{array}{l}
\sigma[\lambda\lambda']^{-1}(\varphi(\sigma[\lambda\lambda'](a))\qquad\mbox{ if } \sigma[\lambda\lambda'](a)\in\llbracket 1,\omega(G)\rrbracket\\
\sigma[\lambda\lambda']^{-1}(\varphi'(\sigma[\lambda\lambda'](a)-\omega(G))+\omega(G))\\\qquad\qquad\qquad \mbox{ if }\sigma[\lambda\lambda'](a)\in\llbracket \omega(G)+1,\omega(G)+\omega(G')\rrbracket.
\end{array}\right.
\end{equation}}

From Equality \eqref{eq_phi''}, one has to examine two cases:
\begin{enumerate}
\item If $\sigma[\lambda\lambda'](a)\in\llbracket 1,\omega(G)\rrbracket$ then 
\begin{equation}\varphi''(a)=\sigma[\lambda\lambda']^{-1}(\varphi(\sigma[\lambda\lambda'](a))).\end{equation}
 Since $G$ is a $B$-diagram, one has \begin{equation}\block_\lambda(\sigma[\lambda\lambda'](a))<\block_\lambda(\varphi(\sigma[\lambda\lambda'](a))).\end{equation} Hence, by Lemma \ref{Relation1} and Equality \eqref{Eq_blocklambda''}, one obtains
\begin{equation}
\begin{array}{rcl}
\block_{(\lambda\lambda')^\sigma}(a)&=&\sigma^{-1}(\block_\lambda(\sigma[\lambda\lambda'](a)))\\&<&\sigma^{-1}(\block_\lambda(\varphi(\sigma[\lambda\lambda'](a)))).\end{array}\end{equation}
But by Equality \eqref{eq_phi''}, we have also $\varphi(\sigma[\lambda\lambda'](a))=\sigma[\lambda\lambda'](\varphi''(a))$ and so using again Equality \eqref{Eq_blocklambda''}, we obtain
{\small
\begin{equation}
\block_{(\lambda\lambda')^\sigma}(a)<\sigma^{-1}
(\block_\lambda(\sigma[\lambda\lambda'](\varphi''(a))))=\block_{(\lambda\lambda')^\sigma}(\varphi''(a)).
\end{equation}
}
\item  If $\sigma[\lambda\lambda'](a)\in\llbracket \omega(G)+1,\omega(G)+\omega(G')\rrbracket$ then 
\begin{equation}\varphi''(a)=\sigma[\lambda\lambda']^{-1}(\varphi'(\sigma[\lambda\lambda'](a)-\omega(G))+\omega(G)).\end{equation} Since $G'$ is a $B$-diagram, one has
 \begin{equation}\block_{\lambda'}(\sigma[\lambda\lambda'](a)-\omega(G))<\block_{\lambda'}(\varphi'(\sigma[\lambda\lambda'](a)-\omega(G))).\end{equation} Hence, by Lemma \ref{Relation1} and Equality \eqref{Eq_blocklambda''}, one obtains
\begin{equation}\begin{array}{rcl}\block_{(\lambda\lambda')^\sigma}(a)&=&\sigma^{-1}(\block_{\lambda'}(\sigma[\lambda\lambda'](a)-\omega(G))+n)\\&<&\sigma^{-1}(\block_{\lambda'}(\varphi'(\sigma[\lambda\lambda'](a)-\omega(G))+n)).\end{array}\end{equation}
But by Equality \eqref{eq_phi''}, we have also \begin{equation}\varphi'(\sigma[\lambda\lambda'](a)-\omega(G))=\sigma[\lambda\lambda'](\varphi''(a))-\omega(G)\end{equation} and so using again Equality \eqref{Eq_blocklambda''}, we obtain
{\small \begin{equation}
\block_{(\lambda\lambda')^\sigma}(a)<\sigma^{-1}
(\block_\lambda(\sigma[\lambda\lambda'](\varphi''(a)-\omega(G))+n)=\block_{(\lambda\lambda')^\sigma}(\varphi''(a)).
\end{equation}}
\end{enumerate}
This shows that $(G|G')^{\sigma}$ is a $B$-diagram. 
\end{proof}
\begin{example}\rmfamily
Let \begin{equation*}G=(3,[3,1,2],64\sqcup5\sqcup\sqcup,\{3,5\},\{1,2,3\}) \text{ and }
G'=(1,[2],\sqcup\sqcup,\{1,2\},\{1\}).\end{equation*}
Then 
$
(G|G')=(4,[3,1,2,2],64\sqcup\sqcup5\sqcup\sqcup\sqcup\sqcup, \{3,5,7,8\},
\{1,2,3,7\})$.
This diagram is illustrated in Figure \ref{bdiag3122}.\\
We consider also the permutation $\sigma=1423\in 123\shuffle 4$. We have $\sigma[\lambda\lambda']=12378456$ and we check easily that 
\begin{equation*}
(G|G')^\sigma=(4,[3,2,1,2],86\sqcup\sqcup\sqcup7\sqcup\sqcup,\{3,4,5,7\},\{1,2,3,4\}),
\end{equation*}
is a $B$-diagram. It corresponds to the following graphical representation:
\begin{figure}[ht]
\centering
\includegraphics[scale=0.27]{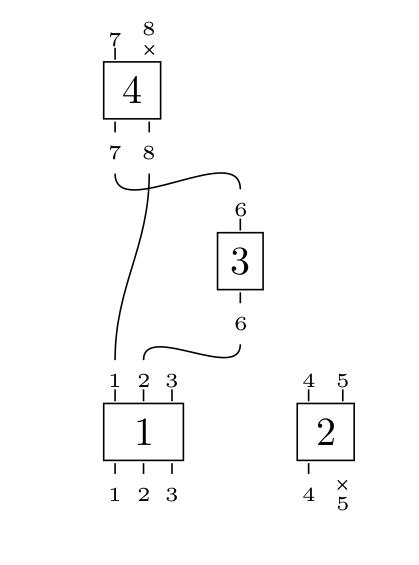}
\caption{The $B$-diagram $(G|G')^\sigma$ \label{PermDiag}}
\end{figure}
\end{example}
%
%

\subsubsection{On the product, coproduct, and Hopf algebra structure}
Let us define the product $\Cup$ on $\mathcal{B}^*$ by setting 
\begin{equation}\label{dualproduct}
D_G\Cup D_{G'}=\sum_{\sigma\in 1\dots n\shuffle n+1\dots n+n'}D_{(G|G')^\sigma}.
\end{equation}
for any $G=(n,\lambda,\varphi,F^\uparrow,F_\downarrow)$ and
$G'=(n',\lambda',\varphi',F'^\uparrow,F'_\downarrow)$.


\begin{example}\rmfamily
Figure \ref{Excprod2} gives an example of product in $\mathcal B^*$.
\begin{figure}[ht]
\centering
\includegraphics[scale=0.27]{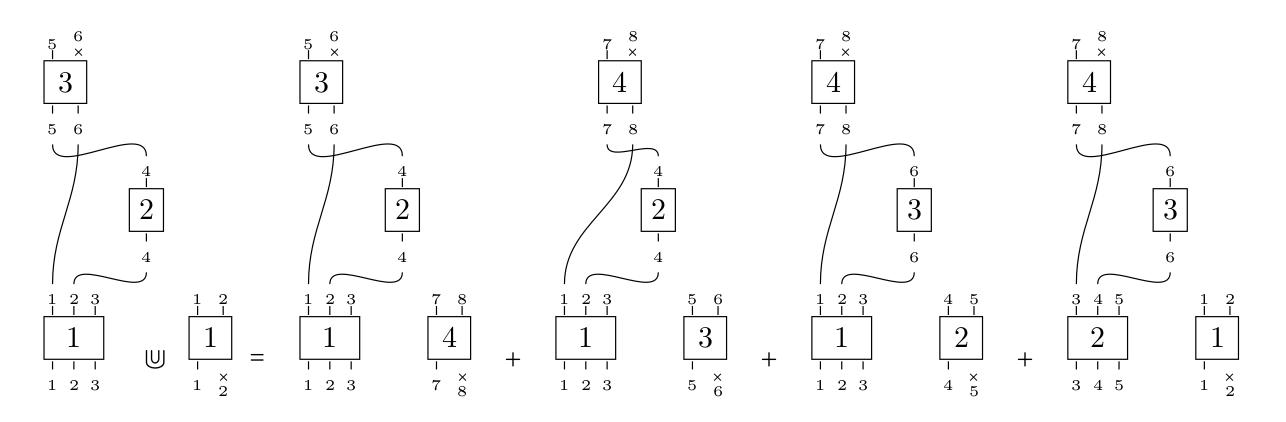}
\caption{An example of product on $\mathcal B^*$ \label{Excprod2}}
\end{figure}
\end{example} 
The product $\Cup$ is associative and commutative.
Indeed, 
{\small
\begin{equation}
D_G\Cup D_{G'}=\sum_{\sigma\in 1\dots n\shuffle n+1\dots n+n'}D_{(G|G')^\sigma}
=\sum_{\sigma\in 1\dots n'\shuffle n'+1\dots n+n'}D_{(G'|G)^\sigma}=D_{G'}\Cup D_{G},
\end{equation}
}
and
\begin{equation}
\begin{array}{rcl}
D_G\Cup (D_{G'}\Cup D_{G''})
&=&\displaystyle
\sum_{\rho\in 1\dots n\shuffle n+1\dots n+n'\shuffle n+n'+1\dots n+n'+n''}D_{(G|G'|G'')^\rho}\\
&=&(D_G\Cup D_{G'})\Cup D_{G''}
\end{array}
\end{equation}
The coproduct on $\mathcal B^*$ is defined by
\begin{equation}
\Delta_*(D_G)=\sum_{G\in G'\medstar G''} D_{G'}\otimes D_{G''}
\end{equation}

\begin{example}\rmfamily
The  coproduct $\Delta_*$ of  $(3,[1,2,2],3\sqcup\sqcup\sqcup\sqcup,\{3,4\},\{1,4,5\})$ is described in Figure \ref{Excomp22}.
\begin{figure}[ht]
\centering
\includegraphics[scale=0.260]{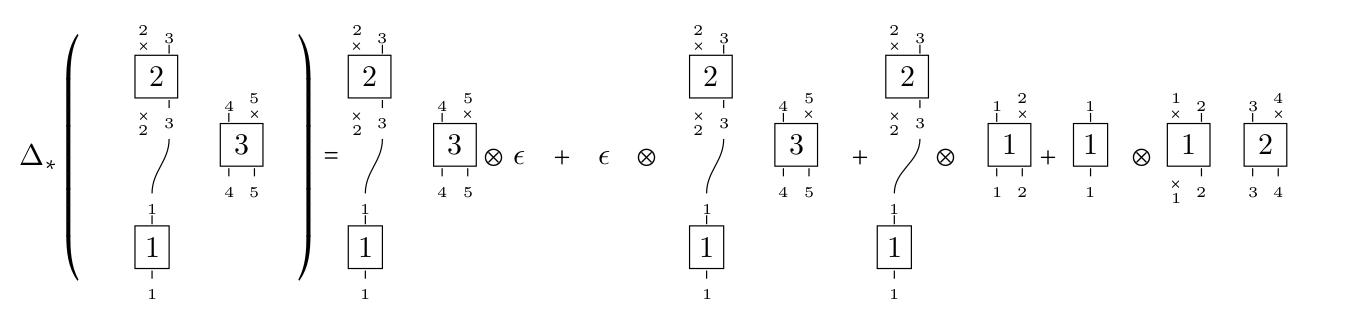}
\caption{An example of coproduct on $\mathcal B^*$ \label{Excomp22}}
\end{figure}

\end{example} 
The co-associativity of the coproduct $\Delta_*$ is due to the associativity of the operation given in formula \eqref{starHW2}.

Obviously, as a graded space, $\mathcal B^*=\bigcup_{n}\mathcal B_n^*$ is isomorphic to the graded dual of $\mathcal B$ because the dimensions of the graded subspaces $\mathcal B_n$ are finite.
The pairing is given by $\big\langle D_G, G' \big\rangle = \delta_{G,G'}$. This pairing naturally extends to a pairing between $\mathcal B^*\otimes\mathcal B^*$ and $\mathcal B\otimes\mathcal B$, i.e. $\langle D_{G_1}\otimes D_{G_2},G'_1\otimes G'_2\rangle=\big\langle D_{G_1}, G'_1 \big\rangle\big\langle D_{G_2}, G'_2 \big\rangle  = \delta_{G_1,G'_1}\delta_{G_2,G'_2}$.
\begin{lemma}\label{DualityPairing}
We have 
\begin{equation}\big\langle D_{G_1} \Cup D_{G_2}, G' \big\rangle= \big\langle D_{G_1} \otimes D_{G_2} , \Delta(G') \big\rangle\end{equation}  and 
\begin{equation}\big\langle \Delta_{*}(D_G), G'_1 \otimes G'_2 \big\rangle = \big\langle D_G, G'_1 \star G'_2 \big\rangle.\end{equation}
\end{lemma}
\begin{proof}
    We apply the definition of $\Cup$ and obtain
\begin{equation}\begin{array}{rcl}
\big\langle D_{G_1} \Cup D_{G_2}, G' \big\rangle  &=& \displaystyle\sum_{\sigma \in 1,n \dots, n \shuffle n+1, \dots, n+n'} 
\big\langle D_{(G_1|G_2)^{\sigma}} , G' \big\rangle \\
& =& \# \bigg\{ \sigma \in 1, \dots, n \shuffle n+1, \dots, n+n' \bigg| G' = (G_1|G_2)^{{\sigma} }\bigg\}.\end{array}\end{equation}
But, by Lemma \ref{DualBDiag}, for any $\sigma \in 1 \cdots n \shuffle n+1 \cdots n+n'$, $G'=(G_1|G_2)^\sigma$ implies 
\begin{equation}G'[\sigma(1), \dots, \sigma(n)] = G_1,\quad G'[\sigma(n+1), \dots, \sigma(n+n')] = G_2,\end{equation} and
\begin{equation}([\sigma(1), \dots, \sigma(n)], [\sigma(n+1), \dots, \sigma(n+n')] \in \mathrm{Split}(G')).\end{equation}
Conversely, straightforwardly from the definition of $\mathrm{Split}$, we obtain that, for any pair $(I,J)$ in $\mathrm{Split}(G')$, $(G'[I],G'[J])=(G_1,G_2)$ implies that there exists 
\begin{equation}\sigma\in 1\cdots n\shuffle n+1\cdots n+n'\end{equation}such that $I=[\sigma(1),\cdots,\sigma(n)]$ and $J=[\sigma(n+1),\cdots,\sigma(n+n')]$. It follows that
\begin{equation}\begin{array}{rcl}
\big\langle D_{G_1} \Cup D_{G_2}, G' \big\rangle &=&\#\{
(I,J)\in\mathrm{Split}(G')\mid G'[I]=G_1\mbox{ and }G'[J]=G_2\}\\
&=&\displaystyle\sum_{(I,J) \in \mathrm{Split}(G')} \big\langle D_{G_1} \otimes D_{G_2} , G'[I] \otimes G'[J] \big\rangle \\
& =&  \big\langle D_{G_1} \otimes D_{G_2} , \Delta(G') \big\rangle.
 \end{array}
\end{equation}
This proves the first equality.\\
From the definition of $\Delta_*$, we have also
\begin{equation}\begin{array}{rcl}
\big\langle \Delta_{*}(D_G), G'_1 \otimes G'_2 \big\rangle &=& \displaystyle\sum_{G \in G_1 \medstar G_2} \big\langle D_{G_1} \otimes D_{G_2} ,  G'_1 \otimes G'_2 \big\rangle  \\
&=&\displaystyle \sum_{G \in G_1 \medstar G_2}  \delta_{G_1,G'_1} \delta_{G_2,G'_2} \\
&=& \left\{\begin{array}{cl} 1 &\mbox{ if } G\in G'_1 \medstar G'_2 \\ 0 &\mbox{ otherwise} \end{array}\right.\\
&=& \big\langle D_G, G'_1 \star G'_2 \big\rangle.
\end{array}
\end{equation}
\end{proof}
Classically, Lemma \ref{DualityPairing} implies that the triple $(\mathcal B^*,\Cup, \Delta_*)$ is a bialgebra.

Since the dimensions of the graded components of $\mathcal B^*$ are finite,  the bialgebra  can be endowed with an antipode which confers to it a structure of Hopf algebra. The following statement summarizes the results of the section.
\begin{theorem}
 The triple $(\mathcal B^*, \Cup, \Delta_*)$ is a Hopf algebra that is dual to $(\mathcal B, \star, \Delta)$.
\end{theorem}
\subsection{A multiplicative basis\label{sec-multbas}}
In the algebra $\mathcal B$, we define the polynomials
\begin{equation}\begin{array}{ll}
\Phi^G:= G&\mbox{ if }G\in\mathbb G\\
\Phi^{G|G'}:=\Phi^G\star\Phi^{G'}.
\end{array}
\end{equation}
\begin{example}\rmfamily
Let us examine the product  for 
$G_1=(1,[1],\sqcup,\{1\},\{1\})$ and
$G_2=(3,[2,1,1],4\sqcup\sqcup\sqcup,\{4\},\{1,2\})$. In terms of $B$-diagrams, we have
\begin{equation*}
\includegraphics[scale=0.28]{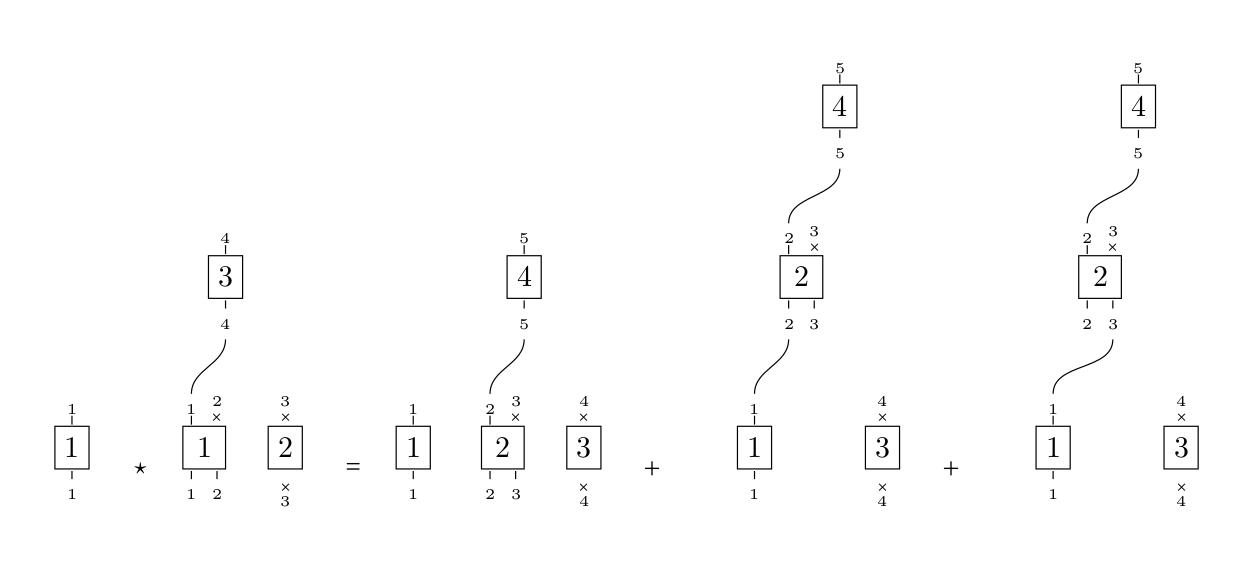}
\end{equation*} 
As a result of this product, we obtain the three $B$-diagrams;
\begin{equation*}\begin{array}{l}
(4,[1,2,1,1],\sqcup5\sqcup\sqcup\sqcup,\{1,5\},\{1,2,3\}),  (4,[1,2,1,1],25\sqcup\sqcup\sqcup,\{5\},\{1,3\})\\ \text{ and } (4,[1,2,1,1],35\sqcup\sqcup\sqcup,\{5\},\{1,2\}).\end{array}
\end{equation*}
\end{example}
We consider the partial order $\preceq$ defined as the transitive closure of $G\preceq G'$, if $G=G_1|G_2$ and $G'\in G_1\medstar G_2$. We denote also \begin{equation}G^\preceq=\{G'\mid G\preceq G'\}=G_1\medstar G_2\medstar \cdots\medstar G_k,\end{equation}
if $G=G_1|G_2|\cdots|G_k$ where $G_1, G_2, \dots ,G_k$ are indivisible $B-$diagrams.
\begin{example}\label{ExPrec}  \rmfamily
Let 
\begin{equation*}\begin{array}{l}G_1=(1,[1],\sqcup,\{1\},\{1\}),\\ 
G_2=(3,[2,1,1],4\sqcup\sqcup\sqcup,\{2\},\{1,2\}),\mbox{ and}\\
G_3=(1,[2],\sqcup\sqcup,\{1,2\},\{1\}).\end{array}\end{equation*} 
We have
\begin{equation*}
\begin{array}{r}
(5,[1,2,1,1,2],\sqcup5\sqcup\sqcup\sqcup\sqcup\sqcup,\{1,3,6,7\},\{1,2,3,6\})^\preceq=(G_1|G_2|G_3)^\preceq\\= \big\{G'_1,G'_2,G'_3,G'_4,G'_5,G'_6,G'_7\big\} \end{array}
\end{equation*}
with
\begin{itemize}
\item $G'_1=G_1|G_2|G_3=(5,[1,2,1,1,2],\sqcup5\sqcup\sqcup\sqcup\sqcup\sqcup,\{1,3,6,7\},\{1,2,3,6\})$
\item $G'_2=\begin{array}{c}G_3\\\comp{1}{1}\\G_1|G_2\end{array}=
\begin{array}{c}G_2|G_3\\\comp{1}{5}\\G_1\end{array}=(5,[1,2,1,1,2],65\sqcup\sqcup\sqcup\sqcup\sqcup,\{3,6,7\},\{1,2,3\})$,
\item $G'_3=G_1|\begin{array}{c}G_3\\\comp{2}{1}\\G_2\end{array}=\begin{array}{c}G_3\\\comp{3}{1}\\ G_1|G_2\end{array}=(5,[1,2,1,1,2],\sqcup 56\sqcup\sqcup\sqcup\sqcup\sqcup,\{1,6,7\},\{1,2,3\}),$
\item $G'_4=\begin{array}{c}G_2\\\comp{1}{1}\\G_1\end{array}|G_3=\begin{array}{c}G_2|G_3\\\comp{1}{1}\\G_1\end{array}=(5,[1,2,1,1,2],25\sqcup\sqcup\sqcup\sqcup\sqcup,\{3,6,5\},\{1,3,6\})$,
\item $G'_5=\begin{array}{c}G_2\\\comp{1}{2}\\G_1\end{array}|G_3=\begin{array}{c}G_2|G_3\\\comp{1}{2}\\G_1\end{array}=(5,[1,2,1,1,2],35\sqcup\sqcup\sqcup\sqcup\sqcup,\{3,6,7\},\{1,2,6\})$,
\item $G'_6=\begin{array}{c}
\left( \begin{array}{c} 
G_3\\\comp 21\\G_2\end{array}\right)\\\comp 11\\G_1
\end{array}=\begin{array}{c}
G_3\\\comp 31\\
\left( \begin{array}{c} 
G_2\\\comp 11\\G_1\end{array}\right)
\end{array}=(5,[1,2,1,1,2],256\sqcup\sqcup\sqcup\sqcup\sqcup,\{6,7\},\{1,3\})$, and
\item $G'_7=\begin{array}{c}
\left( \begin{array}{c} 
G_3\\\comp 21\\G_2\end{array}\right)\\\comp 12\\G_1
\end{array}=\begin{array}{c}
G_3\\\comp 31\\
\left( \begin{array}{c} 
G_2\\\comp 12\\G_1\end{array}\right)
\end{array}=(5,[1,2,1,1,2],356\sqcup\sqcup\sqcup\sqcup\sqcup,\{6,7\},\{1,2\})$.
\end{itemize}
\end{example}
Obviously, we have
\begin{equation}\label{prodPhi}
\Phi^G=\sum_{G\preceq G'}G'=\sum_{G'\in G_1\medstar G_2\medstar\cdots\medstar G_k}G',
\end{equation}
if $G=G_1|G_2|\cdots|G_k$ where $G_1, G_2,\dots, G_k$ are indivisible $B$-diagrams.
The triangular property implies  that the $\Phi^G$'s form a multiplicative basis of $\mathcal B$. Furthermore, we have
\begin{equation}\label{coprodPhi}
\Delta\Phi^G=\sum_{(I,J)\in \mathrm{Split}(G)}\Phi^{G\big[I\big]}\otimes\Phi^{G\big[J\big]}.
\end{equation}
Indeed, if $G$ is indivisible \eqref{coprodPhi} is directly inherited from the formula of $\Delta G$. The general formula is proved by induction as follows
\begin{equation}\begin{array}{rcl}
\Delta\Phi^{G|G'}&=&\Delta\Phi^G\star\Delta\Phi^{G'}\\
&=&\displaystyle \sum_{\genfrac{}{}{0pt}{}{(I,J)\in\mathrm{Split}(G)}{
(I',J')\in\mathrm{Split}(G')}}\Phi^{G[I]}\star\Phi^{G'[I']}\otimes\Phi^{G\big[J\big]}\star\Phi^{G'\big[J'\big]}\\
&=&\displaystyle
\sum_{\genfrac{}{}{0pt}{}{(I,J)\in\mathrm{Split}(G)}{
(I',J')\in\mathrm{Split}(G')}}\Phi^{G[I]|G'[I']}\otimes\Phi^{G\big[J\big]|G'\big[J'\big]}\\
&=&\displaystyle
\sum_{(I,J)\in\mathrm{Split}(G|G')}\Phi^{(G|G')[I]}\otimes\Phi^{(G|G')[J]}
\end{array}
\end{equation}
\begin{example}\rmfamily
Using the notation of Example \ref{ExPrec}.
\begin{equation*}
\Delta\Phi^{G_1|G_2}=\Delta\Phi^{G_1}\star\Delta\Phi^{G_2}
\end{equation*}
with
\begin{equation*}\Delta\Phi^{G_1}=G_1\otimes 1+1\otimes G_1,\end{equation*}
and
\begin{equation*}\Delta\Phi^{G_2}=G_2\otimes 1+G'_2\otimes G''_2+G''_2\otimes G'_2+1\otimes G_2,\end{equation*}
where $G'_2=(1,[1],\sqcup,\emptyset,\emptyset)$ and $G''_2=(2,[2,1],3\sqcup\sqcup,\{2\},\{1,2\})$. So we have
\begin{equation*}\begin{array}{rcl}
\Delta\Phi^{G_1|G_2}&=&G_1\star G_2\otimes 1 +G_1\star G'_2\otimes G''_2+G_1\star G''_2\otimes G'_2+G_1\otimes G_2\\
&&+G_2\otimes G_1+G'_2\otimes G_1\star G''_2+G''_2\otimes G_1\star G'_2+1\otimes G_1\star G_2\\
&=& \Phi^{G_1|G_2}\otimes 1+\Phi^{(G_1|G_2)[1,3]}\otimes \Phi^{(G_1|G_2)[2,4]}\\&&+\Phi^{(G_1|G_2)[1,2,4]}\otimes \Phi^{(G_1|G_2)[3]}+\Phi^{(G_1|G_2)[1]}\otimes \Phi^{(G_1|G_2)[2,3,4]}\\
&&+\Phi^{(G_1|G_2)[2,3,4]}\otimes \Phi^{(G_1|G_2)[1]}+
\Phi^{(G_1|G_2)[3]}\otimes\Phi^{(G_1|G_2)[1,2,4]}
\\&&+\Phi^{(G_1|G_2)[2,4]}\otimes\Phi^{(G_1|G_2)[1,3]}+1\otimes\Phi^{G_1|G_2}.
\end{array}
\end{equation*}
\end{example}
Notice that the freeness of the algebra $\mathcal B$ is more readily apparent when using the basis $\Phi=\{\Phi^G\mid G\in\mathbb B\}$. Indeed, the set $\Phi$ is stable under the product $\star$ and, endowed with this product, is isomorphic to the free monoid over the alphabet $\Sigma_{\mathbb G}=\{\mathtt a_G\mid G\in\mathbb G\}.$  The explicit isomorphism sends $\Phi^G$ to $\mathtt a_G$ for any $G\in\mathbb G$. Hence, as an algebra $\mathcal B$ is isomorphic to $\mathbb C\langle \Sigma_{\mathbb G}\rangle$. 

The dual basis $(\Psi_G)_G$ of $(\Phi_G)_G$, defined by $\langle \Psi_G,\Phi_{G'}\rangle=\delta_{G,G'}$ satisfies
\begin{equation}
\Psi_{G_1}\Cup \Psi_{G_2}=\sum_{G'}\#\bigg\{(I,J)\in\mathrm{Split}(G)\bigg| G'[I]=G_1, G'\big[J\big]=G_2\bigg\}\Psi_{G'}
\end{equation}
and
\begin{equation}
\Delta_\star\Psi_G=\sum_{(G'_1|G'_2)=G}\Psi_{G'_1}\otimes\Psi_{G'_2}.
\end{equation}
\begin{example}\rmfamily
Let $G_1=(1,[1],\sqcup,\{1\},\{1\})$, $G_2=(2,[1,1],\sqcup\sqcup,\{1,2\},\{1,2\})=G_1|G_1$, and $G_3=(2,[1,1],2\sqcup,\{2\},\{1\})=\begin{array}{c}
G_1\\\comp11\\G_1
\end{array}$. The graduation \emph{w.r.t.} the number of the vertices implies $\Psi_{G_1}=D_{G_1}$ and that $\Psi_{G_2}$ and $\Psi_{G_3}$ are linear combinations of $D_{G_1}$ and $D_{G_2}$. Setting $\Psi_{G_2}=\alpha_{2,2}D_{G_2}+\beta_{2,3}D_{G_3}$ and
$\Psi_{G_3}=\alpha_{3,2}D_{G_2}+\beta_{3,3}D_{G_3}$, one obtains
\begin{equation*}
\begin{array}{l}
\langle \Psi_{G_2},\Phi_{G_2}\rangle=\alpha_{2,2}+\alpha_{2,3}=1\\
\langle \Psi_{G_2},\Phi_{G_3}\rangle=\alpha_{2,3}=0\\
\langle \Psi_{G_3},\Phi_{G_2}\rangle=\alpha_{3,2}+\alpha_{3,3}=0\\
\langle \Psi_{G_3},\Phi_{G_3}\rangle=\alpha_{3,3}=1.
\end{array}\end{equation*}
Hence
\begin{equation*}
\Psi_{G_2}=D_{G_2}\mbox{ and }\Psi_{G_3}=D_{G_3}-D_{G_2}.
\end{equation*}
It follows that
\begin{equation*}
\Psi_{G_1}\Cup\Psi_{G_1}=2D_{G_2}=2\Psi_{G_2},
\end{equation*}
\begin{equation*}
\Psi_{G_1}\Cup \Psi_{G_2}=3D_{G_2}=3\Psi_{G_1|G_1|G_1}.
\end{equation*}
In a similar way, one obtains
\begin{equation*}\begin{array}{rcl}
\Psi_{G_1|G_3}&=&D_{G_1|G_3}-D_{G_1|G_1|G_1},\\
\Psi_{G_3|G_1}&=&D_{G_3|G_1}-D_{G_1|G_1|G_1},\\
\Psi_{(3,[1,1,1],3\sqcup\sqcup,\{2,3\},\{1,2\})}&=&D_{(3,[1,1,1],3\sqcup\sqcup,\{2,3\},\{1,2\})}-D_{G_1|G_1|G_1},\\
\Psi_{(3,[1,1,1],12\sqcup,\{3\},\{1\})}&=&
D_{(3,[1,1,1],12\sqcup,\{3\},\{1\})}-
D_{(3,[1,1,1],3\sqcup\sqcup,\{2,3\},\{1,2\})}\\&&-D_{G_1|G_3}-D_{G_3|G_1}+D_{G_1|G_1|G_1},
\end{array}
\end{equation*}
and so on.\\
Let us illustrate the product as follows:
\begin{equation*}\begin{array}{rcl}
\Psi_{G_1}\Cup \Psi_{G_3}&=&D_{G_1}\Cup D_{G_3}- D_{G_1}\cup D_{G_2}\\&=&
D_{G_1|G_3}+D_{(3,[1,1,1],3\sqcup\sqcup,\{2,3\},\{1,2\})}+D_{G_3|G_1}-3D_{G_1|G_1|G_1}\\
&=&\Psi_{G_1|G_3}+\Psi_{G_3|G_1}
+\Psi_{(3,[1,1,1],3\sqcup\sqcup,\{2,3\},\{1,2\})}.
\end{array}
\end{equation*}
We illustrate the coproduct by the following example:
\begin{equation*}
\Delta_\star\Psi_{G_3}=\Delta_\star D_{G_3}-\Delta_\star D_{G_2}.
\end{equation*}
Recall that
\begin{equation*}
\Delta_\star D_{G_2}=D_{G_2}\otimes 1+D_{G_1}\otimes D_{G_1}+1\otimes D_{G_2}
\end{equation*}
and
\begin{equation*} 
\Delta_\star D_{G_3}=D_{G_3}\otimes 1+D_{G_1}\otimes D_{G_1}+1\otimes D_{G_3}
\end{equation*}
because $G_2=G_1|G_1$ and $G_3=\begin{array}{c} G_1\\\comp11\\G_1\end{array}.$
Hence
\begin{equation*}
\Delta_\star\Psi_{G_3}=\Delta_\star\Psi_{G_3}\otimes 1+1\otimes \Psi_{G_3}.
\end{equation*}
\end{example}


\def\colset{\mathrm{colset}}
\section{From B-diagrams to colored set partitions}\label{sec-diagtopart}
\subsection{Word symmetric functions\label{subsec-wsym}}
In \cite{RS}, \emph{Rosas} and \emph{Sagan} defined the Hopf algebra of word symmetric functions in non-commutative variables as the subalgebra of $\mathbb{C}\big\langle\mathbb{A}\big\rangle $ constituted by polynomials which are invariant by permutation of the letters of the alphabet. This algebra was first studied
by \emph{M. C. Wolf} \cite{Wolf} in 1936. We call it $\WSym$ as in \cite{HNT2008}. It is generated by the basis $(M_\pi)$ ($\pi$ is a set partition) with the following product: 
\begin{equation}\label{ProdMWSym}
M_\pi M_{\pi'}=\sum_{\pi''}M_{\pi''},    
\end{equation}
where $\pi=\{\pi_1,\dots,\pi_k\}$ and $\pi^{'}=\{\pi'_1,\dots,\pi'_{k'}\}$ are two set partitions, respectively of $\{1,2, \dots, n\}$ and $\{1,2, \dots, n'\}$, and the sum is over the partitions $\pi''=\{\pi''_1,\dots,\pi''_{k''}\}$ of $\{1,\dots,n+n'\}$ such that for any $i\in\{1,\dots,k''\}$,  $\pi''_i\cap \{1,\dots,n\}\in \pi'$ and if $\pi''_i\cap \{n+1,\dots,n+n'\}=\{j_1,\dots,j_\alpha\}$ then $\{j_1-n,\dots,j_\alpha-n\}\in\pi'$.
$\WSym$ is endowed with the coproduct defined by 
\begin{align}\label{CoProdMWSym}
\Delta(M_{\pi}) = \sum M_{\std(\pi')} \otimes M_{\std(\pi'')}.
\end{align}
The sum is over the pairs $(\pi',\pi'')$ such that $\pi' \cup \pi'' = \pi$, $\pi' \cap \pi'' = \emptyset$ and the standardized, $\std(\pi)$, of a partition $\pi$ is obtained by replacing the $i^{th}$ smallest element by $i$. This confers to $\WSym$ a structure of Hopf algebra (see e.g. \cite{HNT2008}).

\begin{example}\label{ExMpi}\rmfamily
For instance, consider the following computations:
\begin{equation*} \begin{array}{r}M_{\{\{1,3\},\{2\}\}} M_{\{\{1,4\},\{2,3\}\}} = M_{\{\{1,3\},\{2\},\{4,7\},\{5,6\}\}}+ M_{\{\{1,3,4,7\},\{2\},\{5,6\}\}}\\
+ M_{\{\{1,3\},\{2,4,7\},\{5,6\}\}}
+
M_{\{\{1,3,5,6\},\{2\},\{4,7\}\}}
+M_{\{\{1,3\},\{2,5,6\},\{4,7\}\}}\\
+
M_{\{\{1,3,4,7\},\{2,5,6\}\}} +
M_{\{\{1,3,5,6\},\{2,4,7\}\}},
\end{array}
\end{equation*}
and
\begin{equation*}\begin{array}{rcl}
\Delta M_{\{\{1,3,5,6\},\{2,4,7\}\}}&=&M_{\{\{1,3,5,6\},\{2,4,7\}\}}\otimes 1
+M_{\{\{1,2,3,4\}\}}\otimes M_{\{\{1,2,3\}\}}\\ &&+M_{\{\{1,2,3\}\}}\otimes M_{\{\{1,2,3,4\}\}}+1\otimes M_{\{\{1,3,5,6\},\{2,4,7\}\}}
\end{array}\end{equation*}
\end{example}

For any  set partition $\pi=\{\pi_1,\dots,\pi_k\}$, we define
\begin{itemize}
\item $|\pi|=\sum_{i=1}^k\#\pi_i$,
\item $\varphi_\pi:\{1,\dots,|\pi|\} \rightarrow \{1,\dots,|\pi|\}\cup\{\sqcup\}$ such that
\begin{equation}
\varphi_\pi(i)=\left\{
\begin{array}{ll}
i'&\mbox{ if there exists } j \mbox{ such that }i\in\pi_j,\mbox{ and }i'=\min\{i''\in \pi_j\mid i<i'\},\\
\sqcup&\mbox{ otherwise,}
\end{array}
\right.
\end{equation}
\item $\min_\pi=\{\min(\pi_j)\mid j\in\{1,\dots,k\}\}$,
\item  and $\max_\pi=\{\max(\pi_j)\mid j\in\{1,\dots,k\}\}$.\end{itemize}
We associate to any set partition $\pi=\{\pi_1,\cdots,\pi_k\}$ the $B$-diagram
\begin{equation}G_\pi=(|\pi|,[\overbrace{1,\cdots,1}^{\times|\pi|}],\varphi_\pi,\max_\pi,\min_\pi).\end{equation}
If $\pi$ and $\pi'$ are two set partitions, we define \begin{equation}\pi\uplus\pi'=\pi\cup\{\{i_1+|\pi|,\dots, i_\ell+|\pi|\}\mid \{i_1,\dots,i_\ell\}\in\pi'\}\end{equation}
Obviously one has
\begin{align}\label{ProdJuxtGWSym}
G_{\pi\uplus\pi'}=G_\pi|G_{\pi'}.
\end{align}
Conversely if $G_\pi$ is indivisible then $\pi$ cannot be written under the form $\pi=\pi'\uplus\pi''$ with $\emptyset\not\in\{\pi',\pi''\}$.
\begin{example}\rmfamily
If $\pi=\{\{1,4\},\{2,3\},\{5\}\}$ we have $|\pi|=5$, $\varphi_\pi=43\sqcup\sqcup\sqcup$, $\min_\pi=\{1,2,5\}$, $\max_\pi=\{3,4,5\}$. \\
We remark that 
\begin{equation*}
\begin{array}{rcl}
G_\pi&=&(5,[1,1,1,1,1],43\sqcup\sqcup\sqcup,\{3,4,5\},\{1,2,5\})\\&=&(4,[1,1,1,1],43\sqcup\sqcup,\{1,2\},\{3,4\})|(1,[1],\sqcup,\{1\},\{1\})\\&=&G_{\{\{1,4\},\{2,3\}\}} | G_{\{\{1\}\}}.
\end{array}
\end{equation*} In contrast, the diagram
\begin{equation*}
G_{\{\{1,5\},\{2,4\},\{3\}\}}=(5,[1,1,1,1,1],54\sqcup\sqcup\sqcup,\{3,4,5\},\{1,2\})
\end{equation*}
is clearly indivisible.
\end{example}
The subalgebra of $\mathcal B$ generated by the $B$-diagrams $G_\pi$ is isomorphic to $\WSym$.
\begin{example}\rmfamily
Compare the examples below to those given in Example \ref{ExMpi} :
{\small
\begin{equation*}\begin{array}{l}
G_{\{\{1,3\}\{2\}\}}\star G_{\{\{1,4\},\{2,3\}\}}=\\
(3,[1,1,1],3\sqcup\sqcup,\{2,3\},\{1,2\})\star (4,[1,1,1,1],43\sqcup\sqcup,\{3,4\},\{1,2\})=\\
(7,[1^7],3\sqcup\sqcup76\sqcup\sqcup,\{2,3,6,7\},\{1,2,4,5\})+
(7,[1^7],3\sqcup476\sqcup\sqcup,\{2,6,7\},\{1,2,5\})\\
+(7,[1^7],34\sqcup76\sqcup\sqcup,\{3,6,7\},\{1,2,5\})
+(7,[1^7],3\sqcup576\sqcup\sqcup,\{2,6,7\},\{1,2,4\})\\
+(7,[1^7],35\sqcup76\sqcup\sqcup,\{3,6,7\},\{1,2,4\})
+(7,[1^7],35476\sqcup\sqcup,\{6,7\},\{1,2\})\\
+(7,[1^7],34576\sqcup\sqcup,\{6,7\},\{1,2\})=\\
G_{\{\{1,3\},\{2\},\{4,7\},\{5,6\}\}} + G_{\{\{1,3,4,7\},\{2\},\{5,6\}\}}
+ G_{\{\{1,3\},\{2,4,7\},\{5,6\}\}}
+
G_{\{\{1,3,5,6\},\{2\},\{4,7\}\}}\\
+G_{\{\{1,3\},\{2,5,6\},\{4,7\}\}}
+
G_{\{\{1,3,4,7\},\{2,5,6\}\}}
 +
G_{\{\{1,3,5,6\},\{2,4,7\}\}},
\end{array}
\end{equation*}
}
and
\begin{equation*}\begin{array}{rcl}
\Delta G_{\{\{1,3,5,6\},\{2,4,7\}\}}&=&\Delta (7,[1^7],34576\sqcup\sqcup,\{6,7\},\{1,2\})\\
&=& (7,[1^7],34576\sqcup\sqcup,\{6,7\},\{1,2\})\otimes 1\\
&&+ (4,[1^4],234\sqcup,\{4\},\{1\})
\otimes (3,[1^3],23\sqcup,\{3\},\{1\})\\
&&+(3,[1^3],23\sqcup,\{3\},\{1\})\otimes (4,[1^4],234\sqcup,\{4\},\{1\})\\
&&+1\otimes (7,[1^7],34576\sqcup\sqcup,\{6,7\},\{1,2\})\\
&=&
G_{\{\{1,3,5,6\},\{2,4,7\}\}}\otimes 1
+G_{\{\{1,2,3,4\}\}}\otimes G_{\{\{1,2,3\}\}}\\ &&+G_{\{\{1,2,3\}\}}\otimes G_{\{\{1,2,3,4\}\}}+1\otimes G_{\{\{1,3,5,6\},\{2,4,7\}\}}.
\end{array}
\end{equation*}
\end{example}
The $M_\pi$ basis is not the most convenient one to generalize to other types of $B$-diagrams. Instead, we use the basis given, for any $\pi=\{\pi_1,\dots,\pi_k\}$ by
\begin{equation}\label{ProdWsymPhiToM}
\Phi^\pi=\sum_{\pi'}M_{\pi'}    
\end{equation}
where the sum is over the partitions $\pi'=\{\pi'_1,\dots,\pi'_\ell\}$ satisfying $\pi_i\cap\pi'_j\in\{\pi_i,\emptyset\}$ for any $(i,j)\in\{1,\dots,k\}\times\{1,\dots,\ell\}$. In other words, the part of each $\pi'$ is an union of parts of $\pi$.\\
We also note that the algebra $\mathcal B_W$ is also generated by the elements $\Phi^{G_\pi}$ and the explicit isomorphism sends 
$\Phi^{G_\pi}$ to $\Phi^{\pi}$.\\
By duality, there is an isomorphism $\PiQSym\rightarrow \mathcal B^*/_J$, where $J$ is the ideal generated by the elements $\Phi^G$ for any $G$ that cannot be written as a $G_\pi$, sending each $\Psi_\pi$ to the class of $\Psi_{G_\pi}$. In other words, the algebra $\PiQSym$ is isomorphic to the subalgebra of $\mathcal B^*$ generated by the elements $\Psi_{G_\pi}$. This is a special case of proposition \ref{prop_subdual}.

\subsection{Hopf algebras of colored partitions\label{subsec-colorpart}}
Let $a=(a_m)_{m\geq 1}$ be a sequence of integers. Following \cite{ABCLM}, we consider the set $\mathcal {CP}(a)$  of the  \emph{colored partitions} $\Pi=\{[\pi_1,i_1],\dots,[\pi_k,i_k]\}$ such that $\{\pi_1,\dots,\pi_k\}$ is a set partition and each $i_\ell\in\llbracket 1,a_{\#\pi_{\ell}}\rrbracket$ for $1\leq\ell\leq k$. For a colored set partition $\Pi=\{[\pi_1,i_1],\dots,[\pi_k,i_k]\}$, we set $|\Pi|=\sum_{i=1}^k\#\pi_i$. This number is called \emph{size} of $\Pi$. We denote by $\mathcal{CP}_n(a)$ the set of colored partitions of size $n$ associated with the sequence $a$. We also set $\mathcal{CP}(a)=\bigcup_n\mathcal{CP}_n(a)$. We endow $\mathcal{CP}$ with the additional statistic $\#\Pi$ and set $\mathcal{CP}_{n,k}(a)=\big\{\Pi\in\mathcal{CP}_{n}(a)\mid\#\Pi=k\big\}.$
 The \emph{standardized} $\std(\Pi)$ of $\Pi$ is defined as the unique colored set partition obtained by replacing the $i$th smallest integer in  $\pi_j$ by $i$. Let $\Cup$: $\mathcal{CP}_{n,k}\otimes\mathcal{CP}_{n',k'}\longrightarrow \mathcal P(\mathcal{CP}_{n+n',k+k'})$ be defined by
\begin{equation}
\Pi\Cup\Pi'=\{\hat\Pi \cup \hat\Pi' \in\mathcal{CP}_{n+n',k+k'}(  a)\mid \std(\hat\Pi)=\Pi\mbox{ and }
\std(\hat\Pi')=\Pi'\}.
\end{equation}
For any triple $(\Pi,\Pi',\Pi'')$, we denote by $\alpha_{\Pi',\Pi''}^{\Pi}$ the number of pairs of disjoint subsets $(\hat \Pi'$, $\hat\Pi'')$ of $\Pi$ such that $\hat\Pi'\cup\hat\Pi''=\Pi$, $\std(\hat\Pi')=\Pi'$ and $\std(\hat\Pi'')=\Pi''$.
We also define
\begin{equation}
\Pi\uplus\Pi'=\Pi\cup\{\{i_1+|\Pi|,\dots,i_\ell+|\Pi|\}\mid \{i_1,\dots,i_\ell\}\in\Pi'\}.
\end{equation}
The algebra of colored set partitions is defined in \cite{ABCLM}. We recall here its definitions: the bases of the space $\CWSym(a)$ are indexed by colored set partitions.\\
The Hopf algebra of colored set partitions $\CWSym(a)$ is generated by the formal elements $\Phi^\Pi$ for the shifted concatenation product 
\begin{equation}
\Phi^\Pi \Phi^{\Pi'} = \Phi^{\Pi\uplus\Pi'}.
\end{equation}
where $\Pi$ and $\Pi'$ are two colored set partitions of $\{1,\ldots,n\}$ and $\{1,\ldots,n'\}$, respectively. Its coproduct is given by
\begin{equation}
\Delta(\Phi^\Pi) = \sum\limits_{\substack{\hat\Pi_1\cup\hat\Pi_2\Pi\\ \hat\Pi_1\cap\hat\Pi_2=\emptyset}}\Phi^{\std(\hat\Pi_1)}\otimes\Phi^{\std(\hat\Pi_2)}=\sum_{\Pi_1,\Pi_2}\alpha_{\Pi_1,\Pi_2}^{\Pi}\Phi^{\Pi_1}\otimes\Phi^{\Pi_2}.
\end{equation}
A colored set partition $\Pi$ is said \emph{indivisible} if $\Pi=\Pi_1\uplus\Pi_2$ implies $\Pi_1=\Pi$ or $\Pi_2=\Pi$. We denote by $\mathcal{ICP}(a)$ the set of indivisible colored partitions.

The algebra $\CWSym(a)$ is freely generated by $\{\Phi^\Pi\mid \Pi\in\mathcal{ICP}(a)\}$  because the set $\{\Phi^\Pi\mid \Pi\in\mathcal{CP}(a)\}$ endowed with the shifted concatenation is isomorphic to the free monoid over the (formal) alphabet $\Sigma_{\mathcal{ICP}}(a)=\{\mathtt{a}_\Pi\mid \Pi\in\mathcal{ICP}(a)\}$. Moreover, as a bigebra, $\mathbf{CWSym}(a)$ is isomorphic to $(\mathbb C\langle \Sigma_{\mathcal{ICP}}(a)\rangle,\cdot,\Delta_{\Cup}$) where $\cdot$ denotes the concatenation product and $\Delta_\Cup$ is the unique coproduct compatible with the concatenation and satisfying :
\begin{equation}
\Delta_\Cup(\mathtt a_\Pi)=\sum_{\Pi_1,\Pi_2}\alpha_{\Pi_1,\Pi_2}^\Pi\iota(\Phi^{\Pi_1})\otimes \iota(\Phi^{\Pi_2}),
\end{equation}
where $\iota$ is the unique isomorphism (of algebra) from $\CWSym(a)$ to $\mathbb C\langle \Sigma_{\mathcal{ICP}(a)}\rangle$ sending
each $\Phi^\Pi$ to $\mathtt a_\Pi$ for each $\Pi\in\mathcal{ICP}(a)$. This is a special case of the construction described in Section \ref{App_HopfPol}.

The Hopf algebra $\mathrm{C\Pi QSym}(a)$ is formally defined in \cite{ABCLM} as the space spanned by the set $\{\Psi_\Pi\mid \Pi\in\mathcal {CP}(a)\}$ endowed with the product
\begin{equation}\label{Psiprod}
\Psi_{\Pi'}\Psi_{\Pi''}=\sum_{\Pi\in\Pi'\Cup\Pi''}\alpha_{\Pi',\Pi''}^\Pi\Psi_\Pi.
\end{equation} 
Its coproduct is given by
\begin{equation}\label{PsiCo-prod}
\Delta(\Psi_{\Pi})=\sum_{\Pi'\uplus\Pi''=\Pi}\Psi_{\Pi'}\otimes\Psi_{\Pi''}.
\end{equation}

\begin{example}\rmfamily
For instance, consider the following products in $\mathrm{C}\Pi\mathrm{QSym}(a)$, for some $a=(a_i)_{i\in\mathbb N}$ with $a_1>0$ and $a_2>2$ : 
{\small
\begin{equation*} \begin{array}{l}\Psi_{\{[\{1,2\},3]\}} \Psi_{\{[\{1\},1],[\{2,3\},2]\}} = 
\Psi_{\{[\{1,2\},3],[\{3\},1],[\{4,5\},2]\}}\\
+ \Psi_{\{[\{1,3\},3],[\{2\},1],[\{4,5\},2]\}} +\Psi_{\{[\{1,4\},3],[\{2\},1],[\{3,5\},2]\}}+\Psi_{\{[\{1,5\},3],[\{2\},1],[\{3,4\},2]\}}\\
+ \Psi_{\{[\{2,3\},3],[\{1\},1],[\{4,5\},2]\}}+ \Psi_{\{[\{2,4\},3],[\{1\},1],[\{3,5\},2]\}}+\Psi_{\{[\{2,5\},3],[\{1\},1],[\{3,4\},2]\}}\\
+\Psi_{\{[\{3,4\},3],[\{1\},1],[\{2,5\},2]\}}+
\Psi_{\{[\{3,5\},3],[\{1\},1],[\{2,4\},2]\}}+\Psi_{\{[\{4,5\},3],[\{1\},1],[\{2,3\},2]\}}.
\end{array}
\end{equation*}
\begin{equation*} \begin{array}{l}\Psi_{\{[\{1,2\},2]\}} \Psi_{\{[\{1\},1],[\{2,3\},2]\}} = 
\Psi_{\{[\{1,2\},2],[\{3\},1],[\{4,5\},2]\}} + \Psi_{\{[\{1,3\},2],[\{2\},1],[\{4,5\},2]\}}\\
 +\Psi_{\{[\{1,4\},2],[\{2\},1],[\{3,5\},2]\}} +\Psi_{\{[\{1,5\},2],[\{2\},1],[\{3,4\},2]\}}+ 2\Psi_{\{[\{2,3\},2],[\{1\},1],[\{4,5\},2]\}}\\+
2\Psi_{\{[\{2,4\},2],[\{1\},1],[\{3,5\},2]\}}
+2\Psi_{\{[\{2,5\},2],[\{1\},1],[\{3,4\},2]\}}.
\end{array}
\end{equation*}}
\begin{equation*}
\begin{array}{l}
\Delta \Psi_{\{[\{1,3\},1],[\{2,4\},2],[\{5\},1]\}}=
\Psi_{\{[\{1,3\},1],[\{2,4\},2],[\{5\},1]\}}\otimes 1\\ + \Psi_{\{[\{1,3\},1],[\{2,4\},2]\}}\otimes \Psi_{\{[\{1\},1]\}}
+1\otimes \Psi_{\{[\{1,3\},1],[\{2,4\},2],[\{5\},1]\}}.
\end{array}
\end{equation*}
\end{example}
 The Hopf algebra $\mathrm{C\PiQSym}(a)$ is the graded dual Hopf algebra of $\mathbf{CWSym}(a)$. As a special case of   Section \ref{App_HopfPol}, we show that  $\mathrm{C\PiQSym}(a)$ is isomorphic to the Hopf algebra $(\mathbb C\langle \Sigma_{\mathcal{ICP}}(a)\rangle, \Cup, \Delta_\cdot)$, where $\Cup$ is defined inductively by: 
\begin{equation}\label{eq-CupCPiQ}
u_1\Cup u_2=\left\{
\begin{array}{l}
1\quad\mbox{ if }u_1=u_2=1\\\displaystyle
\sum_{\genfrac{}{}{0pt}{}{u_1=p_1v_1, u_2=p_2v_2}{
p_1\neq 1\ \mathrm{or}\ p_2\neq 1}}\sum_{\Pi\in \mathcal{ICP}(a)}(p_1\otimes p_2,\Delta_{\Cup}(\mathtt a_\Pi))\mathtt a_\Pi\cdot (v_1\Cup v_2)\mbox{ otherwise}.
\end{array}
\right.
\end{equation}
\begin{example}
Let us examine the product
\begin{equation*}\label{eq-acupaa}
\ta_{\{[\{1,2\},3]\}}\Cup \big(\ta_{\{[\{1\},1]\}}\ta_{\{[\{1,2\},2]\}}\big).
\end{equation*}
The sum on the right-hand side of \eqref{eq-CupCPiQ} decomposes according to the possible values of $p_1 \otimes p_2$ listed below:
\begin{enumerate}
\item $p_1\otimes p_2=1\otimes \ta_{\{[\{1\},1]\}}$: The only partition $\Pi \in \mathcal{ICP}(a)$ for which 
\begin{equation*}(p_1 \otimes p_2, \Delta_{\Cup}(\ta_\Pi)) \neq 0\end{equation*}
is $\Pi = \{[\{1\},1]\}$, and we have 
$(1 \otimes \ta_{\{[\{1\},1]\}},\, \Delta_{\Cup}(\ta_{\{[\{1\},1]\}})) = 1$.
Consequently, this case contributes
\begin{equation*}
\ta_{\{[\{1\},1]\}}\!\big(\ta_{\{[\{1,2\},3]\}} \Cup \ta_{\{[\{1,2\},2]\}}\big).
\end{equation*} We thus apply the computation inductively to 
$\ta_{\{[\{1,2\},3]\}} \Cup \ta_{\{[\{1,2\},2]\}}$ and obtain
{\small
\begin{equation*}
\begin{array}{l}
\ta_{\{[\{1,2\},3]\}}\Cup \ta_{\{[\{1,2\},2]\}}= \\\big(1\otimes\ta_{\{[\{1,2\},2]\}},\Delta_{\Cup}(\ta_{\{[\{1,2\},2]\}})\big)\ta_{\{[\{1,2\},2]\}}\ta_{\{[\{1,2\},3]\}}+\\
\big(\ta_{\{[\{1,2\},3]\}\otimes 1},\Delta_{\Cup}(\ta_{\{[\{1,2\},3]\}})\big)\ta_{\{[\{1,2\},3]\}}\ta_{\{[\{1,2\},2]\}}+\\
\big(\ta_{\{[\{1,2\},3]\}}\otimes \ta_{\{[\{1,2\},2]\}},\Delta_{\Cup}(\ta_{\{[\{1,3\},3],[\{2,4\},2]\}})\big)\ta_{\{[\{1,3\},3],[\{2,4\},2]\}}+\\
\big(\ta_{\{[\{1,2\},3]\}}\otimes \ta_{\{[\{1,2\},2]\}},\Delta_{\Cup}(\ta_{\{[\{1,4\},3],[\{2,3\},2]\}})\big)\ta_{\{[\{1,4\},3],[\{2,3\},2]\}}+\\
\big(\ta_{\{[\{1,2\},3]\}}\otimes \ta_{\{[\{1,2\},2]\}},\Delta_{\Cup}(\ta_{\{[\{2,3\},3],[\{1,4\},2]\}})\big)\ta_{\{[\{2,3\},3],[\{1,4\},2]\}}+\\
\big(\ta_{\{[\{1,2\},3]\}}\otimes \ta_{\{[\{1,2\},2]\}},\Delta_{\Cup}(\ta_{\{[\{2,4\},3],[\{1,3\},2]\}})\big)\ta_{\{[\{2,4\},3],[\{1,3\},2]\}}\\
=
\ta_{\{[\{1,2\},2]\}}\ta_{\{[\{1,2\},3]\}}+\ta_{\{[\{1,2\},3]\}}\ta_{\{[\{1,2\},2]\}}+
\ta_{\{[\{1,3\},3],[\{2,4\},2]\}}+\\
\ta_{\{[\{1,4\},3],[\{2,3\},2]\}}+
\ta_{\{[\{2,3\},3],[\{1,4\},2]\}}+\ta_{\{[\{2,4\},3],[\{1,3\},2]\}}.
\end{array}
\end{equation*}
}
\item $p_1\otimes p_2=1\otimes\ta_{\{[\{1\},1]\}}\ta_{\{[\{1,2\},2]\}}$ : The contribution of this case to the sum is zero, since $\ta_{\{[\{1\},1]\}}\ta_{\{[\{1,2\},2]\}}$ is not indecomposable.
\item $p_1\otimes p_2=\ta_{\{[\{1,2\},3]\}}\otimes 1$ :The only partition $\Pi \in \mathcal{ICP}(a)$ for which 
$(p_1 \otimes p_2, \Delta_{\Cup}(a_\Pi)) \neq 0$ 
is $\Pi = \{[\{1,2\},3]\}$, and in this case we have 
$(a_{\{[\{1,2\},3]\}} \otimes 1,\, \Delta_{\Cup}(a_{\{[\{1,2\},3]\}})) = 1$.

Therefore, this case contributes
$
a_{\{[\{1,2\},3]\}}\, a_{\{[\{1\},1]\}}\, a_{\{[\{1,2\},2]\}}$.
\item $p_1\otimes p_2=\ta_{[\{1,2\},3]}\otimes \ta_{[\{1\},1]}$: The only partition $\Pi \in \mathcal{ICP}(a)$ for which 
$(p_1 \otimes p_2, \Delta_{\Cup}(\ta_\Pi)) \neq 0$
is $\Pi = \{[\{1,3\},3], [\{2\},1]\}$, and in this case we have
$
\bigl(\ta_{\{[\{1,2\},3]\}} \otimes \ta_{\{[\{1\},1]\}},\,
\Delta_{\Cup}(\ta_{\{[\{1,3\},3],[\{2\},1]\}})\bigr) = 1$.\\
Hence, this case contributes
$
\ta_{\{[\{1,3\},3],[\{2\},1]\}}\,
\ta_{\{[\{1,2\},2]\}}$.

\item $p_1\otimes p_2=\ta_{[\{1,2\},3]}\otimes \ta_{[\{1\},1]}\ta_{[\{1,2\},2]}$: There are twelve partitions $\Pi$ in the set $\mathcal{ICP}(a)$ for which 
$(p_1 \otimes p_2,\Delta_{\Cup}(\ta_\Pi)) \neq 0$. 
They are listed below:
{\small
\begin{equation*}
\begin{array}{cc}
\{[\{1,3\},3],[\{4\},1],[\{2,5\},2]\}& \{[\{1,4\},3],[\{2\},1],[\{3,5\},2]\}\\
 \{[\{1,4\},3],[\{3\},1],[\{2,5\},2]\}& \{[\{1,5\},3],[\{2\},1],[\{3,4\},2]\}\\
\{[\{1,5\},3],[\{3\},1],[\{2,4\},2]\} &\{[\{1,5\},3],[\{4\},1],[\{2,3\},2]\}\\
\{[\{2,3\},3],[\{4\},1],[\{1,5\},2]\}&\{[\{2,4\},3],[\{3\},1],[\{1,5\},2]\}\\
\{[\{2,5\},3],[\{3\},1],[\{1,4\},2]\}&\{[\{2,5\},3],[\{4\},1],[\{1,3\},2]\}\\
\{[\{3,4\},3],[\{2\},1],[\{1,5\},2]\}&\{[\{3,5\},3],[\{2\},1],[\{1,4\},2]\}.
\end{array}
\end{equation*}}
Each of these partitions $\Pi$ satisfies $(p_1\otimes p_2,\Delta_\Cup(\ta_\Pi))= 1$. Hence, the constribution of this case is
{
\begin{equation*}
\begin{array}{l}
\ta_{\{[\{1,3\},3],[\{4\},1],[\{2,5\},2]\}}+\ta_{\{[\{1,4\},3],[\{2\},1],[\{3,5\},2]\}}+\\\ta_{\{[\{1,4\},3],[\{3\},1],[\{2,5\},2]\}}+ \ta_{\{[\{1,5\},3],[\{2\},1],[\{3,4\},2]\}}+\\\ta_{
\{[\{1,5\},3],[\{3\},1],[\{2,4\},2]\}}+\ta_{\{[\{1,5\},3],[\{4\},1],[\{2,3\},2]\}}+\\
\ta_{\{[\{2,3\},3],[\{4\},1],[\{1,5\},2]\}}+\ta_{\{[\{2,4\},3],[\{3\},1],[\{1,5\},2]\}}+\\\ta_{
\{[\{2,5\},3],[\{3\},1],[\{1,4\},2]\}}+
\ta_{\{[\{2,5\},3],[\{4\},1],[\{1,3\},2]\}}
+\\
\ta_{\{[\{3,4\},3],[\{2\},1],[\{1,5\},2]\}}+\ta_{\{[\{3,5\},3],[\{2\},1],[\{1,4\},2]\}}.
\end{array}
\end{equation*}
}
Hence there are $20$ elements in the expansion of
\eqref{eq-acupaa}  :
{
\begin{equation*}
\begin{array}{l}
\ta_{\{[\{1\},1]\}}\ta_{\{[\{1,2\},2]\}}\ta_{\{[\{1,2\},3]\}}+\ta_{\{[\{1\},1]\}}\ta_{\{[\{1,2\},3]\}}\ta_{\{[\{1,2\},2]\}}+\\
\ta_{\{[\{1\},1]\}}\ta_{\{[\{1,3\},3],[\{2,4\},2]\}}+\ta_{\{[\{1\},1]\}}\ta_{\{[\{1,4\},3],[\{2,3\},2]\}}+\\
\ta_{\{[\{1\},1]\}}\ta_{\{[\{2,3\},3],[\{1,4\},2]\}}+\ta_{\{[\{1\},1]\}}\ta_{\{[\{2,4\},3],[\{1,3\},2]\}}+\\
\ta_{\{[\{1,2\},3]\}}\ta_{\{[\{1\},1]\}}\ta_{\{[\{1,2\},2]\}}+\ta_{\{[\{1,3\},3],[\{2\},1]\}}\ta_{\{[\{1,2\},2]\}}
+\\
\ta_{\{[\{1,3\},3],[\{4\},1],[\{2,5\},2]\}}+  \ta_{\{[\{1,4\},3],[\{2\},1],[\{3,5\},2]\}}+ \\ \ta_{\{[\{1,4\},3],[\{3\},1],[\{2,5\},2]\}}+\ta_{\{[\{1,5\},3],[\{2\},1],[\{3,4\},2]\}}+ \\ \ta_{
\{[\{1,5\},3],[\{3\},1],[\{2,4\},2]\}}+\ta_{\{[\{1,5\},3],[\{4\},1],[\{2,3\},2]\}}+\\
\ta_{\{[\{2,3\},3],[\{4\},1],[\{1,5\},2]\}}
+ \ta_{\{[\{2,4\},3],[\{3\},1],[\{1,5\},2]\}}+\\ \ta_{
\{[\{2,5\},3],[\{3\},1],[\{1,4\},2]\}}+
\ta_{\{[\{2,5\},3],[\{4\},1],[\{1,3\},2]\}}
+\\ \ta_{\{[\{3,4\},3],[\{2\},1],[\{1,5\},2]\}}+\ta_{\{[\{3,5\},3],[\{2\},1],[\{1,4\},2]\}}.
\end{array}
\end{equation*}
}
Compare to the expansion of $\Psi_{\{[\{1,2\},3]\}}\Psi_{\{[\{1\},1],[\{2,3\},2]\}}$ in $\CPiQSym(a)$:
{ \begin{equation*}
\begin{array}{l}
\Psi_{\{[\{1\},1]\},[\{2,3\},2],[\{4,5\},3]\}}+\Psi_{\{[\{1\},1],[\{2,3\},3],[\{4,5\},2]\}}+\\
\Psi_{\{[\{1\},1],[\{2,4\},3],[\{3,5\},2]\}}+\Psi_{\{[\{1\},1],[\{2,5\},3],[\{3,4\},2]\}}+\\
\Psi_{\{[\{1\},1],[\{3,4\},3],[\{2,5\},2]\}}+\Psi_{\{[\{1\},1],[\{3,5\},3],[\{2,4\},2]\}}+\\
\Psi_{\{[\{1,2\},3],[\{3\},1],[\{4,5\},2]\}}+\Psi_{\{[\{1,3\},3],[\{2\},1],[\{4,5\},2]\}}
+\\
\Psi_{\{[\{1,3\},3],[\{4\},1],[\{2,5\},2]\}}+ \Psi_{\{[\{1,4\},3],[\{2\},1],[\{3,5\},2]\}}+\\\Psi_{\{[\{1,4\},3],[\{3\},1],[\{2,5\},2]\}}+ \Psi_{\{[\{1,5\},3],[\{2\},1],[\{3,4\},2]\}}+\\ \Psi_{
\{[\{1,5\},3],[\{3\},1],[\{2,4\},2]\}}+\Psi_{\{[\{1,5\},3],[\{4\},1],[\{2,3\},2]\}}+\\
\Psi_{\{[\{2,3\},3],[\{4\},1],[\{1,5\},2]\}}+ \Psi_{\{[\{2,4\},3],[\{3\},1],[\{1,5\},2]\}}+\\\Psi_{
\{[\{2,5\},3],[\{3\},1],[\{1,4\},2]\}}+
\Psi_{\{[\{2,5\},3],[\{4\},1],[\{1,3\},2]\}}
+\\ \Psi_{\{[\{3,4\},3],[\{2\},1],[\{1,5\},2]\}}+\Psi_{\{[\{3,5\},3],[\{2\},1],[\{1,4\},2]\}}.
\end{array}
\end{equation*}
}
\end{enumerate}
\end{example}

\subsection{Subalgebras generated by a finite set of elementary B-diagrams\label{subsec-finiteset}}
Let $\mathbb D$ be a finite set of elementary $B$-diagrams, i.e. diagrams $G$ satisfying $|G|=1$. We denote by $\mathbb B_{\mathbb D}$ the set of the $B$-diagrams which are constructed by using only elements in $\mathbb D$. Again, endowed with the product $|$, $\mathbb B_{\mathbb D}$ is the free monoid generated by the set $\mathbb G_{\mathbb D}$ indivisible diagrams constructed from $\mathbb D$. We consider the subalgebra $\mathcal B_{\mathbb D}$ of $\mathcal B$ generated by $\mathbb B_{\mathbb D}$. We observe that $\mathcal B_{\mathbb D}=\mathbb C\langle\mathbb G_\mathbb D\rangle$. The fact that $\mathbb D$ is finite implies that the dimension of the subspace $\mathbb B_{\mathcal D}^{[n]}$ spanned by the diagrams $G\in\mathbb B_{\mathbb D}$, such that $|G|=n$, is finite. The dimension $\beta^{\mathbb D}_{n}$ of $\mathbb B_{\mathcal D}^{[n]}$ are computed by recurrence as follows. First consider the number $\kappa_{n,p}^{\mathbb D}$ of $B$-diagrams $G$ in $\mathbb B_{\mathbb D}$ with $|G|=n$ and $f^\uparrow(G)=p$.
Recall the alternative inductive definition of a $B$-diagram given in \cite{BCL2016} :
\begin{lemma}\label{recBDiag}
Let $G=(n,\lambda,\varphi,F^\uparrow,F_\downarrow)$ be a B-diagram. Either $G=\varepsilon$ or there exists a B-diagram
$V=(1,[p],\emptyset\rightarrow\emptyset,F^\uparrow_v,{F_v}_\downarrow)$, a B-diagram $\tilde G=(n-1,\tilde \lambda,\tilde\varphi,\tilde F^\uparrow,\tilde F_\downarrow,)$ and two sequences $1\leq a_1<\dots<a_k\leq p$ and $1\leq b_1,\dots,b_k\leq \omega(\tilde G)$ distinct satisfying
\begin{equation}
G=\begin{array}{c}
\tilde G\\
\mcomp{a_1,\dots,a_k}{b_1,\dots,b_k}\\
V
\end{array}.
\end{equation}
\end{lemma}
From Lemma \ref{recBDiag}, 
 we obtain
\begin{equation}
\kappa^{\mathbb D}_{n,p}=\sum_{G\in\mathbb D}\sum_{j=0}^{f_{\downarrow}(G)}j!\binom {f_{\downarrow}(G)}j\binom{p-f^{\uparrow}(G)+j}j\kappa^{\mathbb D}_{n-1,p-f^{\uparrow}(G)+j}
\end{equation}
 and so
 \begin{equation}\label{F1}
 \beta_n^{\mathbb D}=\sum_{p=0}^{mn}\kappa^{\mathbb D}_{n,p}
 \end{equation}
 where $m=\max\{f^\uparrow(G)\mid G\in\mathbb D\}$. Denoting by $d^{\mathbb D}_{n,p}$ the number  of $B$-diagrams $G$ in $\mathbb B_{\mathbb D}$ with $\omega(G)=n$ and $f^\uparrow(G)=p$. Lemma \ref{recBDiag} can be also derived as
 \begin{equation}
 d_{n,p}^{\mathbb D}=\sum_{G\in\mathbb D}\sum_{j=0}^{f_{\downarrow}(G)}j!\binom {f_{\downarrow}(G)}j\binom{p-f^{\uparrow}(G)+j}jd^{\mathbb D}_{n-|G|,p-f^{\uparrow}(G)+j}.
 \end{equation}
 Notice that this formula remains valid even if $\mathbb D$ is an infinite set. Furthermore, in the special case where $\mathbb D$ is the set of all elementary $B$-diagrams, we recover formula \eqref{enumdiag}.
 
 The number $\alpha_n^{\mathbb D}$ of $B$-diagrams $G\in\mathbb B_{\mathbb D}$ such that $\omega(G)=n$ is given by
 \begin{equation}\label{F2}
 \alpha_n^{\mathbb D}=\sum_{k=0}^nd_{n,k}^{\mathbb D}.
 \end{equation}

 Obviously, $\Delta$ sends $\mathcal B_{\mathbb D}$ to $\mathcal B_{\mathbb D}\otimes \mathcal B_{\mathbb D}$. So $\mathcal B_{\mathbb D}$ is a co-commutative connected Hopf algebra. The dual algebra $\mathcal B^*_{\mathbb D}$ is the quotient $\mathcal B^*_{\mathbb D}/_{J_{\mathbb D}}$ where $J_{\mathbb D}$ is the ideal generated by the elements $\{D_G\mid G\not\in\mathbb B_{\mathbb D}\}$. The product $D_G\Cup D_{G'}$ with $G, G'\in \mathbb B_{\mathbb D}$ is a linear combination of $B$-diagrams in $\mathbb B_{\mathbb D}$. So we fall in the scope of Section \ref{App_HopfPol} and, by Proposition \ref{prop_subdual}, we deduce that the algebra $\mathcal B^*_{\mathbb D}$ is isomorphic to the subalgebra of $\mathcal B^*$ generated by $\{D_G\mid G\in\mathbb B_{\mathbb D}\}$.

Let us denote by $c^{\mathbb D}_n$ the number of connected $B$-diagrams in $\mathbb B_{\mathbb D}$ with $n$ vertices and set $c^{\mathbb D}:=(c^{\mathbb D}_n)_{n\geq 1}$.
Comparing \eqref{dualproduct} and \eqref{Psiprod}, we find
\begin{theorem}\label{CPIQ}
The algebras $\mathrm{C\Pi QSym}(c^{\mathbb D})$ and $\mathcal B_{\mathbb D}^*$ are isomorphic
\end{theorem}

\begin{proof}\footnote{Although this result follows directly from Proposition~\ref{prop_subdual}, we also give a fully constructive proof.} 
To each $G \in \mathbb B_{\mathbb D}$, we associate a number $\eta(G) \in \{1, \dots,c^{\mathbb D}_{|G|}\}$, in such a way that the maps 
\begin{equation}
\left\{\begin{array}{llll}   \{G \in  \mathcal B_{\mathbb D}\mid & |G|= n\} & \longrightarrow & \{1, \dots,c^{\mathbb D}_{|G|} \} \\
 &  G & \longrightarrow & \eta(G) \end{array}\right.
\end{equation}
are bijection. Noticing that for any $G\in \mathbb B_{\mathbb D}$, the set
\begin{equation}\colset(G)=\bigg\{[I, \eta(G[I])]\mid I \in \mathrm{Connected}(G) \bigg\}\end{equation}
is a colored set partition\footnote{We use a slight abuse of notation by assimilating a strictly increasing sequence 
$[i_1,\dots,i_k]$ to the set $\{i_1,\dots,i_k\}$.}, we define the following linear map:  
\begin{equation}
\chi:\left\{\begin{array}{llll}   & B_{\mathbb D}* & \longrightarrow & \mathrm{C\Pi QSym}(c^{\mathbb D}) \\
 &  D_G & \longrightarrow & \Psi_{\colset(G)}\end{array}.\right.
\end{equation}
Let us show that $\chi$ is a morphism. We have, for any pair of $B$-diagrams $(G,G')$
\begin{equation}
\begin{array}{rcl}
\chi(D_{G} \Cup D_{G'}) & =& \displaystyle\sum_{\sigma \in 1 \dots n \shuffle n+1 \dots n + n'} \chi \bigg(D_{(G|G'})^{\sigma} \bigg) \\
& =& \displaystyle \sum_{\sigma \in 1 \dots n \shuffle n+1 \dots n + n'} \Psi_{ \colset\big(( G|G' )^{\sigma}\big)} ,
\end{array}
\end{equation}
with $n=|G|$ and $n'=|G'|$.\\
But
\begin{equation}\begin{array}{rcl}
\colset((G|G')^\sigma)&=&\displaystyle\{[I,\eta((G|G')^\sigma[I])\mid I\in \mathrm{Connected}((G|G')^\sigma)\}\\
&=&\displaystyle
\big\{[I,\eta((G|G')^\sigma[I])\mid I\in \mathrm{Connected}((G|G')^\sigma)\\
&&\qquad\qquad\qquad\mbox{ and }(G|G')^\sigma)[I]=G[I^{\sigma^{-1}}]\big\}\cup\\&&
\displaystyle\
\big\{[I,\eta((G|G')^\sigma[I])\mid I\in \mathrm{Connected}((G|G')^\sigma)\\
  &&\qquad\qquad\qquad\mbox{ and }(G|G')^\sigma)[I]=G'[I^{\sigma^{-1}}]\big\}\\
&=& \{[I^{\sigma},\eta(G[I])]\mid I\in\mathrm{Connected}(G)\}\cup\\&&
\{[I^{\sigma},\eta(G'[I])]\mid I\in\mathrm{Connected}(G')\}\end{array}
\end{equation}
Hence,
$
\colset((G|G')^\sigma)\in \colset(G)\Cup\colset(G').
$ So we have
\begin{equation}
\chi(D_G\Cup D_{G'})=\sum_{\Pi\ \in\ \colset(G)\Cup\colset(G')} \mu_{\Pi}\Psi_\Pi
\end{equation}
with 
{
\begin{equation}\begin{array}{rcl}
\mu_\Pi&=&\displaystyle\#\Big\{\sigma\in 1\dots n\shuffle n+1\dots n+n'\mid \colset((G|G')^\sigma)=\Pi,\\&& (G|G')^\sigma[\sigma^{-1}(1),\dots,\sigma^{-1}(n)]=G,\\
&&\displaystyle(G|G')^\sigma[\sigma^{-1}(n+1),\dots,\sigma^{-1}(n+n']=G' \Big\}\\
&=&\displaystyle\#\Big\{\sigma\in 1\dots n\shuffle n+1\dots n+n'\mid \colset((G|G')^\sigma)=\Pi,\\&& \colset((G|G')^\sigma[\sigma^{-1}(1),\dots,\sigma^{-1}(n)])=\colset(G),\\
&&\displaystyle\colset((G|G')^\sigma[\sigma^{-1}(n+1),\dots,\sigma^{-1}(n+n'])=\colset(G') \Big\}
\\&=&\big\{(\hat \Pi,\hat\Pi')\mid \std(\hat\Pi)=\colset(G), \std(\hat\Pi')=\colset(G'), \hat\Pi\cup\hat\Pi'=\Pi\big\}\\
&=&\alpha_{\colset(G),\colset(G')}^\Pi
\end{array}
\end{equation}}
It follows that
\begin{equation}
\begin{array}{rcl}
\chi(D_G\Cup D_{G'})&=&\displaystyle\sum_{\Pi\ \in\ \colset(G)\Cup\colset(G')} \alpha_{\colset(G),\colset(G')}^\Pi\Psi_\Pi\\
&=&\Psi_{\colset(G)}\Psi_{\colset(G')}=\chi(D_G)\chi(D_{G'}).
\end{array}
\end{equation}
\end{proof}
\begin{example}\ \\ \rmfamily
Let $\mathbb D$ be the set of generators $(1,[1],\sqcup ,\{1\},\{1\})$ and $(1,[2],\sqcup\sqcup,\{1,2\},\{1,2\})$.
Consider the diagrams \begin{equation*}G_1=(2,[2,2],43\sqcup\sqcup,\{3,4\},\{1,2\})\mbox{ and }
G_2=(1,[2],\sqcup\sqcup ,\{1,2\},\{1,2\}),\end{equation*} and
compare the formula
\begin{equation*}
\includegraphics[scale=0.23]{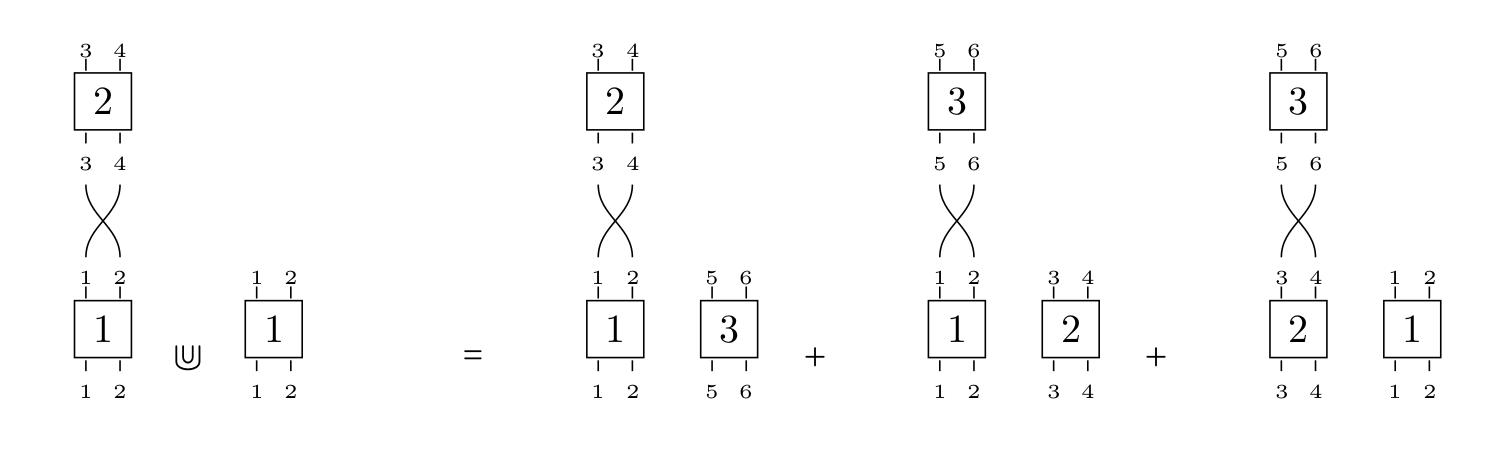}
\end{equation*}

 to the following ones
 {
\begin{equation*}
\begin{array}{l}
\Psi_{\{ [\{1,2\},\eta(G_1)] \}} \Psi_{\{ [\{1\},\eta(G_2)] \}} = \Psi_{\{ [\{1,2\},\eta(G_1)], [\{3\},\eta(G_2)] \}} +\\\qquad\qquad \Psi_{\{ [\{1,3\},\eta(G_1)], [\{2\},\eta(G_2)] \}} + \Psi_{\{ [\{2,3\},\eta(G_1)], [\{1\},\eta(G_2)] \}}\end{array}
\end{equation*}
}
which holds in the algebra of colored set partitions.\\
Notice that, since the number of connected $B$-diagrams of size $1$  is $2$, the number of connected $B$-diagrams of size $2$ is $11$, we can set $\eta(G_1)=6\leq 11$ and $\eta(G_2)=2\leq 2$ (but  any other choice of values would suit $\eta$).
\end{example}

Let $B_{n,k}^{\mathbb D}$ be the sum of the $D_G$ where $G\in \mathbb B_{\mathbb D}$ is a $B$-diagram with exactly $k$ connected components. As a consequence of Theorem \ref{CPIQ}, one has
\begin{equation}
\sum_{n}B_{n,k}^{\mathbb D}z^n=\frac1{k!}\Bigg(\sum_{n}B_{n,1}^{\mathbb D}z_n\Bigg)^{\Cup k}.
\end{equation}
Hence, following the mechanism of specialization as described in  \cite{ABCLM}, one obtains
\begin{corollary}\label{betaser}
\begin{equation}\sum_{n\geq 0}\beta^{\mathbb D}_n\frac{z^n}{n!}=\exp\Bigg\{\sum_{n\geq 1}c^{\mathbb D}_n\frac{z^n}{n!}\Bigg\}.\end{equation}
\end{corollary}

\section{Examples}\label{sec-examples}
In the aim to illustrate our purpose, let us consider several special cases.
\subsection{ Recovering the Hopf algebras WSym and BWSym}
Setting $\mathbb D_W=\{(1,[1],\sqcup,\{1\},\{1\})\}$ we recover the results of Section $5.1$ in \cite{BCL2016}. For any $n$ there is a unique connected $B$-diagram which is $(n,[1,\cdots,1],23\cdots n\sqcup,\{n\},\{1\})$. Hence, $c^{\mathbb D_W}_n=1$. Since the weight and the size of the unique generator are equal, we have $\beta^{\mathbb D_W}_n=\alpha^{\mathbb D_W}_n$. It is easy to recover that $\beta^{\mathbb D_W}_n=b_n$, the $n$th Bell number and $c^{\mathbb D_W}_n=1$ since for any connected component $I=[i_1,\cdots,i_k]$ of a $B$-diagram $G$, we have $G[I]=(n,[1,\cdots,1],23\cdots n\sqcup,\{n\},\{1\})$. Hence, for any set partition $\Pi=\{\pi_1,\cdots,\pi_\ell\}$ there exists a unique $B$-diagram $G$, such that $\mathrm{Connected}(G)=\Pi$.
Note also that $\kappa_{n,p}^{\mathbb D_W}=d_{n,p}^{\mathbb D_W}=S_{n,p}$, the Stirling number of the second kind. We showed in \cite{BCL2016} that the Hopf algebra $\mathcal B_{\mathbb D_W}$ is isomorphic to $\WSym$.

 Setting $\mathbb D_{BW}=\{(1,[2],\sqcup\sqcup,\{1,2\},\{1\})\}$, we recover the results of Section $5.2$ in \cite{BCL2016}. One has $\beta^{\mathbb D_{BW}}_n=L_{n}$, the number of set partitions into lists of $\llbracket 1,n\rrbracket$, $\alpha^{\mathbb D_{BW}}_{2n}=\beta^{\mathbb D_{BW}}_n$, and $\alpha^{\mathbb D_{BW}}_{2n+1}=0$. Notice also that $c^{\mathbb D_{BW}}_n=n!$, this follows from the bijection between the connected elements with $n$ vertices and the permutations of $\S_n$ described in Section $5.2$ in \cite{BCL2016}.  Observe also that the number of connected components of a $B$-diagram $G\in \mathbb B_{\mathbb D_{BW}}$ equals $f^{\uparrow}(G)-n$. Hence, $\kappa_{n,k}^{\mathbb D_{BW}}=0$ if $k<n$ and $\kappa_{n,k}^{\mathbb D_{BW}}=L_{n,k-n}$, the Lah number \footnote{Also called  Stirling number of third kind.} \cite{Lah} which counts the number of set partitions of $\llbracket 1,n\rrbracket$ into $k$ lists, otherwise. Also  we have $d_{2n,p}^{\mathcal D_{BW}}=\kappa_{n,k}^{\mathbb D_{BW}}$ and $d_{2n+1,p}^{\mathcal D_{BW}}=0$. We also recall \cite{BCL2016} that the Hopf algabra $\mathcal B_{\mathbb D_{BW}}$ is isomorphic to $\BWSym$.
\subsection{A simple example involving Fibonacci numbers }
Let $\mathbb F=\{G_1=(1,[1],\sqcup, \emptyset,\{1\}),G_2=(1,[2],\sqcup\sqcup, \emptyset,\{1,2\})\}$. The monoid $\mathbb B_{\mathbb F}$ is freely generated by $\mathbb F$. Indeed, the only connected diagrams are the elements of $\mathbb F$. So each $B$-diagram can be seen as a word on two letters, a letter of weight $1$ and the other of weight $2$. Hence, with no surprise, $\beta^{\mathbb F}_n=2^n$, because it enumerates the words by length, and $\alpha^{\mathbb F}_n=F_n$, the $n$th Fibonacci number, because it enumerates the words by weight. To be more precise a generic $B$-diagrams is under the form
$(n,[\lambda_1,\cdots,\lambda_n],\sqcup,\emptyset,\{1,\dots,\lambda_1+\cdots+\lambda_n\})$ , with $\lambda_1,\dots, \lambda_n\in\{1,2\}$. So it is characterized only by the sequence $[\lambda_1,\dots,\lambda_n]$, assimilated to a word over the alphabet $\{1,2\}$.  The  algebra $\mathcal B_{\mathbb F}$ is isomorphic to the free algebra generated $\mathbb C\langle a_1,a_2\rangle$ by two letters. The explicit isomorphism sends $G_1$ to $a_1$ and $G_2$ to $a_2$. Indeed, the number labeling the vertices give the position of the letter in the word.
For instance, one has
\begin{equation}
\includegraphics[scale=0.25]{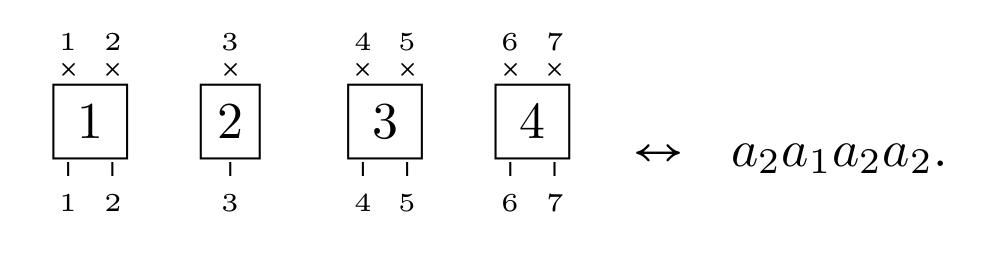}
\end{equation}
The considered graduation sends a word to the sum of the indices of its letters \begin{equation}a_{i_1}\cdots a_{i_n}\rightarrow i_1+\cdots+i_n.\end{equation} The fact that the dimensions of the graded component are equal to the Fibonacci numbers is a very classical exercise. Also classically, the free algebra endowed with the co-shuffle co-product is a Hopf algebra and it is easy to show that is isomorphic to Hopf algebra $\mathcal B_{\mathbb F}$. The gradued dual algebra $\mathcal B_{\mathbb F}^*$ is isomorphic to the shuffle algebra $(\mathbb C\langle a_1,a_2\rangle,\shuffle)$. For instance compare
\begin{equation}a_1a_2\shuffle a_1 =a_1a_2a_1+2a_1a_2 \end{equation}
to
\begin{equation}
\includegraphics[scale=0.22]{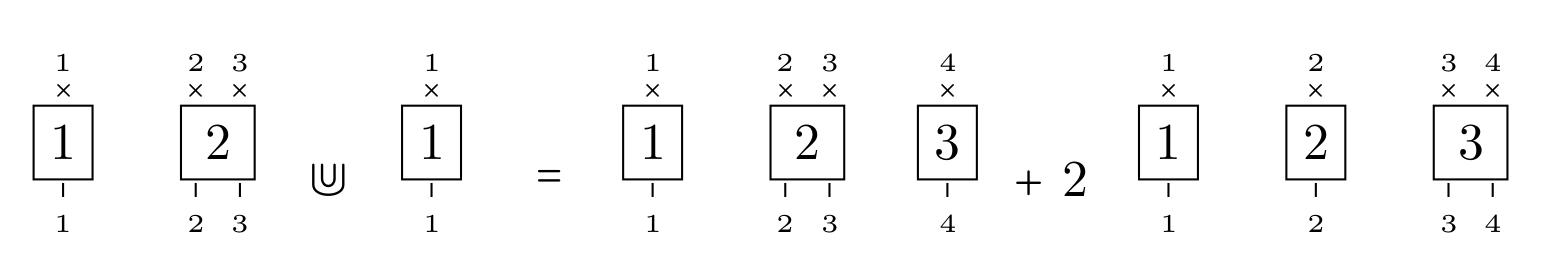}
\end{equation}
\subsection{An algebra generated by an inner cut element  and an outer cut element}
Let $\mathbb S=\{(1,[1],\sqcup,\emptyset,\{1\}),(1,[1],\sqcup,\{1\},\emptyset)\}$. There are only three connected $B$-diagrams in $\mathbb B_{\mathbb S}$: 
\begin{equation}
A_1:=(1,[1],\sqcup,\emptyset,\{1\}), A_2:=(1,[1],\sqcup,\{1\},\emptyset), \mbox{ and } B:=(2,[1,1],2\sqcup,\emptyset,\emptyset).
\end{equation}
Hence, the sequence of the $c_{n}^{\mathbb S}$'s is $(2,1,0,\dots)$. Applying corollary \ref{betaser}, we find \begin{equation}\beta_{n}^{\mathbb S}=\alpha_n^{\mathbb S}=1, 2, 5, 14, 43, 142, 499, 1850, 7193, 29186,\dots.\end{equation}
These numbers appear in \cite{Sloane} (sequence A005425) and are related to several combinatorial problems. The generating series of the numbers $d^{\mathbb S}_{n,k}=\kappa^{\mathbb S}_{n,k}$ is
\begin{equation}
\sum_{n,k}d^{\mathbb S}_{n,k}\frac{z^n}{n!}t^k=\exp\Bigg\{(1+t)z+\frac{z^2}2\Bigg\}.
\end{equation}
To illustrate Theorem \ref{CPIQ}, we consider the algebra $\mathrm{C\Pi QSym}(2,1,0,\dots)$. Compare  the equality
\begin{align}
\sum_{\genfrac{}{}{0pt}{}{\Pi\vDash n}{ \#\Pi=k}}\Psi_{\Pi}=[z^n]\frac1{k!}\Big((\Psi_{\{[\{1\},1]\}}+\Psi_{\{[\{1\},2]\}})z+\Psi_{\{[\{1,2\},1]\}}z^2\Big)^k,
\end{align}
to
\begin{align}
B^{\mathbb S}_{n,k}=[z^n]\frac1{k!}\Big((D_{A_1}+D_{A_2})z+D_Bz^2\Big)^{\Cup k}.
\end{align}
For instance, $B^{\mathbb S}_{3,2}=(D_{A_1}+D_{A_2})\Cup D_B$, that is the formal sum of the diagrams pictured in figure \ref{B32S}, corresponds to the element
\begin{equation}\begin{array}{l}\Psi_{\{[\{1,2\},1],[\{3\},1]\}}+
\Psi_{\{[\{1,3\},1],[\{2\},1]\}}+\Psi_{\{[\{2,3\},1],[\{1\},1]\}}\\
+
\Psi_{\{[\{1,2\},1],[\{3\},2]\}}+
\Psi_{\{[\{1,3\},1],[\{2\},2]\}}+\Psi_{\{[\{2,3\},1],[\{1\},2]\}}
 .\end{array}
 \end{equation}
\begin{figure}[ht]
\centering
\includegraphics[scale=0.28]{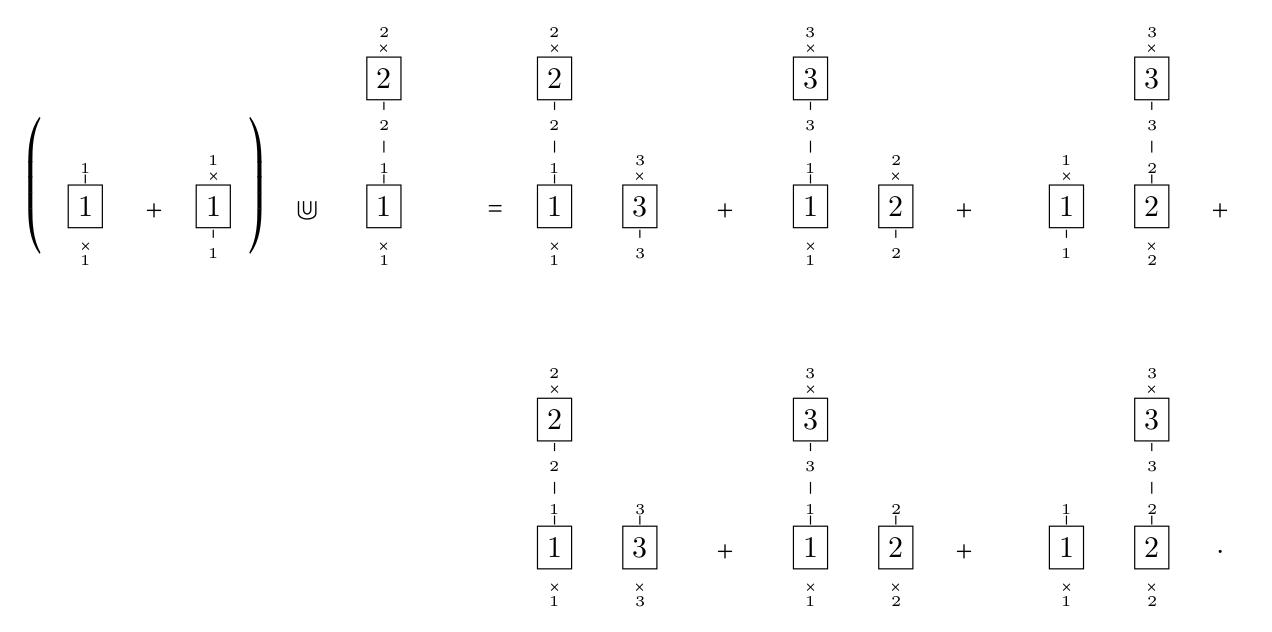}
\caption{$B$-diagrams occurring in $B^{\mathbb S}_{3,2}$ \label{B32S}}
\end{figure} 
\subsection{Hopf algebras of forests}
 By omitting the cut inner half edges, we regard the $B$-diagrams under the form \begin{equation}T_m = (1,[m],\sqcup\cdots\sqcup,\{1,\dots,m\},\{1\})\end{equation} as elementary trees. Hence, 
the $B$-diagrams obtained by connecting or juxtaposing this kind of vertices are forests of increasing trees. Consider a finite set $E$ of  positive integers. We define $\mathbb T(E)=\big\{(1,[m],\sqcup\cdots\sqcup,\{1,\dots,m\},\{1\})\mid m\in E\big\}$. The basis of $\mathcal B_{\mathbb T}$  are indexed by forests of increasing trees and the connected $B$-diagrams in $\mathbb B_{\mathbb T}$ are then regarded as increasing trees. The exponential generating function satisfies the differential equation
\begin{equation}
 \frac d{dz}f_E(z)=\sum_{m\in E}f_E(z)^m,
\end{equation}
 with the initial condition $f_E(0)=1$.
And so,
$
\sum_{n\geq 0}\beta^{\mathbb T(E)}_n\frac{z^n}{n!}=\exp\big\{f_E(z)-1\big\}.$ 
Whilst a closed formula for $f_E$ is not known in the general case, some special interesting cases can be solved. For instance: $f_{\{m\}}(z)=(1-(m-1)z)^{\frac1{1-m}}$  and $f_{\{1,m\}}(z)=\big(2\exp\big\{(1-m)z\big\}-1\big)^{\frac1{1-m}}.$
The first generating functions corresponds to the case of an  algebra generated the set of $m$-ary increasing trees. Notice that these trees are also called $(m-1)$-permutations and appear in some combinatorial Hopf algebras (see \cite{NT}); the links between the algebras of our paper and those of \cite{NT} remain to be investigated. Obviously, the generating series of the coefficients $\displaystyle\kappa^{\mathbb T(\{m\})}_{n,k}$ is \begin{equation}\displaystyle\sum_{n,k} \kappa^{\mathbb T(\{m\})}_{n,k}\frac{z^n}{n!}t^k=\exp\big\{t(f_{\{m\}}(zt^{m-1})-1)\big\}\end{equation} and $d_{n,k}^{\mathbb T(\{m\})}=\kappa_{p,k}^{\mathbb T(\{m\})}$ if $n=3p$ and $0$ otherwise.

\section{Conclusion}
We have shown that the Hopf algebra of 
$B$-diagrams provides a natural algebraic framework for the diagrammatic structures arising in bosonic normal ordering and related problems in algebraic combinatorics. 
Beyond these explicit constructions, a central perspective suggested by our results is that the Hopf algebra of 
$B$-diagrams should enjoy certain universal properties with respect both to diagrammatic models of bosonic systems and to some combinatorial cocommutative Hopf algebras. However, these universal properties are not yet  rigorously characterized. Nevertheless, our results point toward a conjectural universality of the Hopf algebra of $B$-diagrams, which appears to govern a broad class of combinatorial Hopf algebras associated with diagrammatic structures from theoretical physics.

From an algebraic combinatorics point of view, this universality property is reflected in the fact that many cocommutative combinatorial Hopf algebras can be realized as subalgebras and/or quotients of the Hopf algebra of 
$B$-diagrams. We have provided several examples illustrating this phenomenon. It nevertheless remains to clarify the precise scope of this universality property, both at the algebraic level (in connection with Proposition \ref{prop_subdual}) and from a combinatorial viewpoint: Does it concern only objects related to set partitions? or does it extend to a broader framework? 

This touches upon a rather philosophical aspect of algebraic combinatorics, namely the idea of defining objects in a way that cuts across several disciplines. Here, the aim is to bring into resonance themes from theoretical physics, pure algebra, and combinatorics (and also analysis, through the use of formal power series). The guiding questions are: what is the abstract notion underlying these $B$-diagrams? how does it manifest itself in these different fields? and which conceptual dictionary should be used to ``speak'' about it within each of these fields?
Among the possible approaches toward this ideal goal, the study of the characters of the Hopf algebra of $B$-diagrams would be a particularly interesting direction to explore. Indeed, a detailed understanding of the character algebra would make it possible to develop its representation theory.
 Representation theory should provide a transversal tool that can be used in many areas. 
For example, it is well known that the creation and annihilation operators admit a realization as linear operators
$a^\dagger \sim \times x$ (multiplication by $x$) and $a \sim \frac{d}{dx}$ acting on the polynomial space $\mathbb C[x]$. 
These operators satisfy both the commutation relation~\eqref{eq_aa} and the Katriel formula~\eqref{eq-Katriel}. 
A natural question is whether this can be understood within a representation-theoretic framework for algebra $\mathcal B$. 
What are its irreducible representations and how could they provide a bridge between the different fields mentioned above? 
These are among the research directions we plan to pursue in future work.

\def\doi#1{doi: \href{https://doi.org/#1}{#1}}
\def\Zbl#1{Zbl: #1}
\def\MR#1{MR: #1}
\def\arXiv#1{arXiv: #1}
\def\isbn#1{isbn: #1}

\end{document}